\documentclass[12pt]{amsart}
\usepackage{amsmath, amssymb, amsfonts}
\usepackage[all]{xypic}
\usepackage[pdftex]{graphicx}
\usepackage{stmaryrd}
\usepackage{color}
\usepackage{dsfont}
\usepackage{enumitem}
\usepackage{mathtools}

\setlist[enumerate,1]{label={\upshape(\roman*)}}

\theoremstyle{plain}
\newtheorem{theorem}{Theorem}[section]

\newtheorem{lemma}[theorem]{Lemma}

\newtheorem{corollary}[theorem]{Corollary}
\newtheorem{prop}[theorem]{Proposition}

\newtheorem{conj}[theorem]{Conjecture}

\theoremstyle{remark}

\newtheorem{remark}[theorem]{Remark}

\newtheorem*{note*}{Note}
\newtheorem*{remark*}{Remark}
\newtheorem*{example*}{Example}

\theoremstyle{definition}
\newtheorem*{definition*}{Definition}
\newtheorem*{hypothesis*}{Hypothesis}
\newtheorem*{assumption*}{Assumption}
\newtheorem{definition}[theorem]{Definition}

\newcommand{\Z}{\mathbb{Z}}
\newcommand{\R}{\mathbb{R}}
\newcommand{\Q}{\mathbb{Q}}
\newcommand{\C}{\mathbb{C}}

\newcommand{\Ann}{\mathrm{Ann}}
\newcommand{\Aut}{\mathrm{Aut}}
\newcommand{\Gal}{\mathrm{Gal}}

\newcommand{\GL}{\mathrm{GL}}

\newcommand{\cl}{\mathrm{cl}}

\newcommand{\im}{\mathrm{im}}

\newcommand{\nr}{\mathrm{nr}}

\newcommand{\Hom}{\mathrm{Hom}}

\newcommand{\res}{\mathrm{res}}
\newcommand{\quot}{\mathrm{quot}}

\newcommand{\tors}{\mathrm{tors}}
\newcommand{\Irr}{\mathrm{Irr}}
\newcommand{\ind}{\mathrm{ind}}
\newcommand{\PMod}{\mathrm{PMod}}

\newcommand{\Spec}{\mathrm{Spec}}
\newcommand{\ram}{\mathrm{ram}}
\newcommand{\ab}{\mathrm{ab}}
\newcommand{\Br}{\mathrm{Br}}
\newcommand{\half}{{\textstyle \frac{1}{2}}}

\numberwithin{equation}{section}

\newcommand{\onto}{\twoheadrightarrow}
\newcommand{\Fitt}{\mathrm{Fitt}}
\newcommand{\mal}{^{\times}}
\newcommand{\et}{\mathrm{\acute{e}t}}

\newcommand{\perf}{\mathrm{perf}}
\newcommand{\tor}{_{\mathrm{tor}}}
\newcommand{\Det}{\mathrm{Det}}
\newcommand{\aug}{\mathrm{aug}}
\newcommand{\infl}{\mathrm{infl}}

\newcommand{\cone}{\mathrm{cone}}
\newcommand{\Ext}{\mathrm{Ext}}

\newcommand{\cok}{\mathrm{cok}}
\newcommand{\Hyp}{\mathrm{Hyp}}

\title[An unconditional proof of the abelian EIMC and applications]{An unconditional proof of the abelian equivariant\\ Iwasawa main conjecture and applications}

\author{Henri Johnston}
\address{
Department of Mathematics and Statistics\\
University of Exeter\\
Exeter\\
EX4 4QE\\
United Kingdom
}
\email{H.Johnston@exeter.ac.uk}
\urladdr{https://mathematics.exeter.ac.uk/people/profile/index.php?username=hj241}

\author{Andreas Nickel}
\address{Institut f\"ur Theoretische Informatik, Mathematik und Operations Research\\ 
Universit\"at der Bundeswehr M\"unchen\\ 
Werner-Heisenberg-Weg~39\\
85579 Neubiberg\\
Germany} 
\email{andreas.nickel@unibw.de}
\urladdr{https://www.unibw.de/timor/mitarbeiter/univ-prof-dr-andreas-nickel}

\date{6th December 2024}
\begin{document}

\begin{abstract}
Let $p$ be an odd prime.
We give an unconditional proof of the equivariant Iwasawa main conjecture for totally real fields
for every admissible one-dimensional $p$-adic Lie extension whose Galois group has 
an abelian Sylow $p$-subgroup. 
Crucially, this result does not depend on the vanishing of any $\mu$-invariant.
As applications, we deduce the Coates--Sinnott conjecture 
away from its $2$-primary part and new cases of the
equivariant Tamagawa number conjecture for Tate motives.
\end{abstract}

\maketitle

\section{Introduction}

Let $p$ be an odd prime and let $K$ be a totally real number field. 
An admissible $p$\nobreakdash-adic Lie extension $\mathcal{L}$ of $K$ is a Galois extension
$\mathcal{L}$ of $K$ such that (i) $\mathcal{L}/K$ is unramified outside a finite set of primes of $K$,
(ii) $\mathcal{L}$ is totally real, (iii) $\mathcal{G} := \Gal(\mathcal{L}/K)$ is a compact 
$p$\nobreakdash-adic Lie group, and (iv) $\mathcal{L}$ contains the cyclotomic $\Z_{p}$-extension of $K$.
The equivariant Iwasawa main conjecture (EIMC) for such an extension $\mathcal{L}/K$
can be seen as a refinement and generalisation of the classical Iwasawa main conjecture for totally real fields proven by Wiles \cite{MR1053488}.
Roughly speaking, it relates a certain Iwasawa module attached to $\mathcal{L}/K$ to 
special values of Artin $L$-functions via $p$-adic $L$-functions. 
This relationship can be expressed as the existence of a certain element in an algebraic $K$-group; it is also conjectured that this element is unique. 

Let $S$ be a finite set of places of $K$ containing 
all infinite
 places and all places that ramify in $\mathcal{L}$
(thus $S$ necessarily contains all primes above $p$).
Let $M_{S}^{\ab}(p)$ be the maximal abelian pro-$p$-extension of 
$\mathcal{L}$ unramified outside $S$ and set 
$X_{S} = \Gal(M_{S}^{\ab}(p)/\mathcal{L})$. 
The canonical short exact sequence of profinite groups
\begin{equation}\label{eqn:group-extension}
1 \longrightarrow X_{S} \longrightarrow \Gal(M_{S}^{\ab}(p)/K) \longrightarrow
\mathcal{G} \longrightarrow 1
\end{equation}
defines an action of $\mathcal{G}$ on $X_S$
in the usual way so that $X_S$ becomes a module over the Iwasawa algebra
$\Lambda(\mathcal{G}) := \Z_{p}\llbracket\mathcal{G}\rrbracket$.
If $\mathcal{G}$ contains no elements of order $p$, then
$X_{S}$ is of finite projective dimension over $\Lambda(\mathcal{G})$,
and so the EIMC can be stated in terms of $X_{S}$.
In general, however, $X_{S}$ is not of finite projective dimension and so one has to replace
$X_{S}$ by a certain canonical complex $C_{S}^{\bullet}$ 
of $\Lambda(\mathcal{G})$-modules which is perfect
and whose only non-vanishing cohomology modules are isomorphic to $X_{S}$ and
$\Z_{p}$, respectively.

There are several versions of the EIMC. 
The first is due to Ritter and Weiss and deals with the case of one-dimensional extensions  \cite{MR2114937},
and was proven under the hypothesis 
that the $\mu$-invariant of $X_{S}$ vanishes
in a series of articles culminating in \cite{MR2813337}.
In their approach, the complex $C_{S}^{\bullet}$ is obtained from the
canonical group extension \eqref{eqn:group-extension} by applying a certain
`translation functor' \cite[\S 4A]{MR1894887} which essentially transforms
\eqref{eqn:group-extension} into a homomorphism of $\Lambda(\mathcal{G})$-modules
\[
Y_{S} \longrightarrow \Lambda(\mathcal{G})
\]
with kernel $X_S$ and cokernel 
$\Z_p$. It can be shown that this map defines a complex with the
required properties.
The second version follows the framework of Coates, Fukaya, Kato, Sujatha and Venjakob \cite{MR2217048} and 
was proven by Kakde \cite{MR3091976}, again assuming $\mu=0$.
This version is for arbitrary admissible extensions and Kakde's proof uses a strategy of Burns and Kato to reduce to the one-dimensional case
(see Burns \cite{MR3294653}). Here, the choice of complex
appears to be different, but in the one-dimensional case both complexes are isomorphic in the derived category
of $\Lambda(\mathcal{G})$-modules by a result of the second author
\cite[Theorem 2.4]{MR3072281}
(see also Venjakob \cite{MR3068897} for a thorough discussion of the
relation of the work of Ritter and Weiss to that of Kakde.)
As a consequence, it does not matter which of the two complexes we use.
Finally, Greither and Popescu \cite{MR3383600} formulated and proved 
another version of the EIMC,
but they restricted their formulation to abelian one-dimensional extensions and the formulation itself
requires a $\mu=0$ hypothesis.
In \cite{MR3072281}, the second author generalised this formulation (again assuming $\mu=0$) to the non-abelian one-dimensional case. Moreover, he showed that the three formulations are in fact all equivalent in the situation that they make sense, that is, when
the extension is one-dimensional and $\mu=0$. 
In fact, the proof of this result implicitly shows 
that the choice of complex $C_{S}^{\bullet}$ does not matter
when $\mu=0$, as long as it is perfect and has the prescribed cohomology.

From a result of Ferrero and Washington \cite{MR528968}, one can deduce that the $\mu=0$ hypothesis 
holds whenever $\mathcal{L}/K$ is an admissible extension such that $\mathcal{L}$ is a pro-$p$ extension of a finite abelian extension of $\Q$, but unfortunately little is known beyond this case. 
In previous work \cite{MR3749195}, the present authors proved the EIMC unconditionally
for an infinite class of one-dimensional admissible extensions for
which the $\mu=0$ hypothesis is not known to be true. 
However, such extensions must satisfy certain rather restrictive hypotheses.

In the present article, we prove the EIMC (with uniqueness) in important cases without assuming any
$\mu=0$ hypothesis.
The proof relies on the classical (non-equivariant) Iwasawa main conjecture
proven by Wiles \cite{MR1053488} and the recent groundbreaking work of
Dasgupta and Kakde \cite{MR4513146} on the strong Brumer--Stark conjecture.

In an earlier version of this article, the proof also used a formulation of the EIMC
given in the present authors' article \cite{MR3980291}, where explicit calculations
involving the class of $C_{S}^{\bullet}$ in the derived category played a crucial role.
In the present version, we give a simplified proof that dispenses with the results of loc.\ cit.\ 
and instead uses a purely algebraic result that implies that the precise choice of complex used
in the abelian EIMC does not matter, provided that it is perfect and has the prescribed cohomology.
In Appendix \ref{app:independence-of-choice-of-complex}, we give a generalisation of this last result
in the non-abelian case. Our main result is as follows.

\begin{theorem}[Theorem \ref{thm:EIMC-abelian-exts-proof}]\label{thm:EIMC-abelian-exts}
Let $p$ be an odd prime and let $K$ be a totally real number field. 
Let $\mathcal{L}/K$ be an abelian admissible one-dimensional $p$-adic Lie extension.
Then the EIMC with uniqueness holds for $\mathcal{L}/K$.
\end{theorem}

It is natural to ask whether one can deduce the EIMC for all admissible one-dimensional $p$-adic 
Lie extensions from Theorem \ref{thm:EIMC-abelian-exts} by generalising the approaches of 
Ritter and Weiss and of Kakde.
The first step is to reduce to admissible subextensions with $p$\nobreakdash-elementary Galois groups. 
In the aforementioned approaches, this step relied on the $\mu=0$ hypothesis. 
By showing that certain products of maps over subquotients of $\mathcal{G}$
are injective and
exploiting the functorial properties of the EIMC, we obtain a similar result without any such hypothesis. 
We hence deduce the following generalisation of Theorem~\ref{thm:EIMC-abelian-exts}.

\begin{corollary}[Corollary \ref{cor:EIMC-abelian-Sylow-p-proof}]\label{cor:EIMC-abelian-Sylow-p}
Let $p$ be an odd prime and let $K$ be a totally real number field. 
Let $\mathcal{L}/K$ be an admissible one-dimensional $p$-adic Lie extension such that 
$\Gal(\mathcal{L}/K)$ has an abelian Sylow $p$-subgroup.
Then the EIMC with uniqueness holds for~$\mathcal{L}/K$.
\end{corollary}

The further reduction steps of previous approaches do not generalise easily as they rely on the $\mu=0$
hypothesis in a crucial way and hence presently there is no apparent way to deduce the EIMC 
for all admissible one-dimensional extensions without this hypothesis.
Moreover, a serious obstacle to the case of admissible extensions of dimension greater than 
one is that in general a certain `$\mathfrak{M}_{H}(G)$-conjecture' is required to even formulate
the EIMC in this situation, and that this is presently only known to hold under the
$\mu=0$ hypothesis
(see \cite[p.\ 5]{zbMATH06148870} and \cite{MR2905532}).

We remark that if Leopoldt's conjecture holds for $K$ at $p$ then every abelian admissible extension of $K$ must be one-dimensional.
Similarly, if Leopoldt's conjecture holds for $F$ at $p$ for all finite totally real extensions $F/K$ with
$[F:K]$ coprime to $p$ then every admissible extension of $K$ whose Galois group has an abelian Sylow $p$-subgroup must be one-dimensional.
Hence the hypothesis that the extensions considered in Theorem~\ref{thm:EIMC-abelian-exts}
and Corollary~\ref{cor:EIMC-abelian-Sylow-p} are one-dimensional is arguably not so restrictive. 
Moreover, the one-dimensional case of the EIMC often suffices for applications,
some of which we will now discuss.

\medskip

The equivariant Tamagawa number conjecture (ETNC) has been
formulated by Burns and Flach \cite{MR1884523} in vast generality.
In the case of Tate motives, it simply asserts that an
associated canonical element in a relative algebraic $K$-group vanishes.
Roughly speaking, this element relates leading terms of Artin $L$-functions to
certain arithmetic invariants.

Let $L/K$ be a finite Galois CM extension of number fields and let $G=\Gal(L/K)$.
Hence $K$ is totally real, $L$ is totally complex and complex conjugation
induces a unique central automorphism in $G$.
In the case that the $\mu=0$ hypothesis holds
(for the cyclotomic $\Z_p$-extension of the maximal totally real
subfield of $L(\zeta_p)$, where $\zeta_p$ denotes a primitive
$p$th root of unity), it is known by independent work of Burns \cite{MR3294653}
and of the second author \cite{MR3072281} that the EIMC implies
the plus (resp.\ minus) $p$-part of the ETNC
for the pair $(h^{0}(\Spec(L))(r), \Z[G])$ if $r$ is a negative odd (resp.\ even) integer.
In both approaches, the main reason for the $\mu=0$ assumption is to ensure the validity of the EIMC.
Thus at first sight, Theorem \ref{thm:ETNC-at-negative-integers-intro} below appears to be a direct consequence of our results on the EIMC above.
However, Burns' descent argument relies on the formalism developed by Burns and Venjakob in \cite{MR2749572}.
For this, the cohomology of a certain complex at infinite level needs to be `$S$-torsion'
in the terminology of \cite{MR2749572} 
if $p$ divides $|G|$.
(Note this is not related to the set $S$ used in the present article.)
In the context of the EIMC, this is in fact equivalent to $\mu=0$.
Moreover, the approach of the second author in \cite{MR3072281}
relies on the aforementioned version of the abelian EIMC of Greither and Popescu \cite{MR3383600} whose
very formulation depends on the $\mu=0$ hypothesis.
Therefore the fact that Corollary \ref{cor:EIMC-abelian-Sylow-p} implies 
the following result requires a new proof.

\begin{theorem}[Theorem \ref{thm:ETNC-at-negative-integers}]\label{thm:ETNC-at-negative-integers-intro}
Let $p$ be an odd prime.
Let $L/K$ be a finite Galois CM extension of number fields 
and let $G=\Gal(L/K)$.
Suppose that $G$ has an abelian Sylow $p$-subgroup.
Then for each negative odd (resp.\ even) integer $r$ the plus (resp.\ minus) $p$-part of the ETNC 
for the pair $(h^{0}(\Spec(L))(r), \Z[G])$ holds.
\end{theorem}

Now assume that $L/K$ is a finite abelian extension of number fields. 
Let $S$ be a finite set of places of $K$ that contains all 
infinite
 places and
all places that ramify in $L$. 
We write $\mathcal{O}_{L,S}$ for the ring of $S(L)$-integers in $L$,
where $S(L)$ denotes the set of places of $L$ 
that lie above a place in $S$. For an integer $n \geq 0$
we let $K_n(\mathcal{O}_{L,S})$ denote the Quillen $K$-group
of $\mathcal{O}_{L,S}$.
Using $L$-values at negative integers $r$ one can define Stickelberger elements $\theta_S(r)$ in the rational group ring $\Q[G]$.
If we write $K_{1-2r}(\mathcal{O}_{L})_{\tors}$ for the torsion subgroup
of $K_{1-2r}(\mathcal{O}_{L})$,
then by independent work of
Deligne and Ribet \cite{MR579702} and of Pi.\ Cassou-Nogu{\`e}s \cite{MR524276} 
we have
\[
	\Ann_{\Z[G]}(K_{1-2r}(\mathcal{O}_{L})_{\tors}) \theta_{S}(r) \subseteq \Z[G].
\]
Coates and Sinnott \cite{MR0369322} 
formulated the following analogue of Brumer's conjecture for higher $K$-groups.

\begin{conj}[Coates--Sinnott] \label{conj:Coates--Sinnott}
Let $L/K$ be a finite abelian extension of number fields and let $G=\Gal(L/K)$.
Let $r$ be a negative integer and let $S$ be a finite
set of places of $K$ that contains all infinite
 places and all
places that ramify in $L$. Then we have
\[
\Ann_{\Z[G]}(K_{1-2r}(\mathcal{O}_L)_{\tors}) \theta_{S}(r) \subseteq 
\Ann_{\Z[G]}(K_{-2r}(\mathcal{O}_{L,S})).
\]
\end{conj}

Let $p$ be an odd prime and suppose in addition that $S$ contains all
$p$-adic places of $K$.
For any negative integer $r$ and $i=0,1$ Soul\'e \cite{MR553999}
has constructed canonical $G$-equivariant $p$-adic Chern class maps
\begin{equation} \label{eqn:Chern-class-maps}
\Z_{p} \otimes_{\Z} K_{i-2r}(\mathcal{O}_{L,S}) \longrightarrow 
H^{2-i}_{\et} (\Spec(\mathcal{O}_{L,S}), \Z_{p}(1-r)).
\end{equation}
Soul\'e  proved surjectivity
and by the norm residue isomorphism theorem \cite{MR2529300} (formerly
known as the Quillen--Lichtenbaum Conjecture) these maps are
actually isomorphisms. 

This allows us to work with an \'etale cohomological version of Conjecture \ref{conj:Coates--Sinnott}. 
For a variant of this version it has been shown in \cite[\S 6]{MR3383600} that it suffices
to consider abelian CM extensions.
We therefore obtain the following consequence of Theorem
\ref{thm:ETNC-at-negative-integers-intro}.

\begin{theorem}[Corollary \ref{cor:Coates--Sinnott}]\label{thm:Coates--Sinnott}
The Coates--Sinnott conjecture holds away from its $2$\nobreakdash-primary part.
\end{theorem}

\subsection*{Acknowledgements}
The authors wish to thank Samit Dasgupta, Mahesh Kakde 
and Otmar Venjakob
for helpful correspondence; and Mahesh Kakde and Masato Kurihara 
for comments on an early draft of this article.
The authors are also grateful to an anonymous referee for several comments and corrections.
The second author acknowledges financial support provided by the 
Deutsche Forschungsgemeinschaft (DFG) 
within the Heisenberg programme (project no.\ 437113953).

\subsection*{Notation and conventions}
All rings are assumed to have an identity element and all modules are assumed
to be left modules unless otherwise  stated. 
We shall sometimes abuse notation by using the symbol $\oplus$ to 
denote the direct product of rings or orders.
We fix the following notation:

\medskip

\begin{tabular}{ll}
$R^{\times}$ & the group of units of a ring $R$\\
$\zeta(R)$ & the centre of a ring $R$\\
$\Ann_{R}(M)$ & the annihilator of the $R$-module $M$\\
$M_{n} (R)$ & the set of all $n \times n$ matrices with entries in a ring $R$\\
$Quot(R)$ & the field of fractions of the integral domain $R$\\
$\zeta_{n}$ & a primitive $n$th root of unity\\
$K_{\infty}$ & the cyclotomic $\Z_{p}$-extension of the number field $K$\\
$\cl_{K}$ & the class group of a number field $K$ \\
$K^{\mathrm{c}}$ & an algebraic closure of a field $K$ \\
$K^{+}$ & the maximal totally real subfield of a field $K$ embeddable into $\C$\\
$\Irr_{F}(G)$ & the set of $F$-irreducible characters of the (pro)-finite group $G$\\
& (with open kernel) where $F$ is a field of characteristic $0$\\
$\check{\chi}$ & the character contragredient to $\chi$\\
$\mathrm{Re}(s)$ & the real part of the complex number $s$
\end{tabular}

\section{The Brumer--Stark conjecture}

\subsection{Fitting ideals and Pontryagin duals}
If $M$ is a finitely presented module over a commutative ring $R$,
we denote the (initial) Fitting ideal of $M$ over $R$ by $\Fitt_R(M)$.
For basic properties of Fitting ideals, 
we refer the reader to \cite{MR0460383} or \cite[\S 20.2]{MR1322960}.

For an abstract abelian group $A$ we write $A^{\vee}$ for the Pontryagin dual 
$\Hom_{\Z}(A, \Q / \Z)$.
This induces an equivalence between the categories of abelian profinite groups and discrete abelian torsion groups
(see \cite[Theorem 1.1.11]{MR2392026} and the discussion thereafter). 
For a prime $p$, a finite group $G$, and a finitely generated $\Z_{p}[G]$-module $M$, 
we have
\[
M^{\vee}=\Hom_{\Z_{p}}(M, \Q_{p}/\Z_{p}),
\]
and this is endowed with the contragredient $G$-action $(gf)(m) = f (g^{-1} m)$ for $f \in M^{\vee}$, $g \in G$ and $m \in M$. 

\subsection{Equivariant Artin $L$-functions and values}\label{subsec:equiv-Artin-L-functions}
Let $L/K$ be a finite Galois extension of number fields and let $G=\Gal(L/K)$.
For each place $v$ of $K$ we fix a place $w$ of $L$ above $v$ and write $G_{w}$ and $I_{w}$
for the decomposition group and the inertia subgroup of $G$ at $w$, respectively.  
When $w$ is a finite place, we choose a lift $\sigma_{w} \in G_{w}$ of the Frobenius automorphism at $w$ and write $\mathfrak{P}_{w}$ for the associated prime ideal in $L$.
For a finite place $v$ of $K$ we denote the cardinality of its residue field by $\mathrm{N}v$.

Let $S$ be a finite set of places of $K$ containing the infinite places $S_{\infty}=S_{\infty}(K)$.
Let $\Irr_{\C}(G)$ denote the set of complex irreducible characters of $G$. 
For $\chi \in \Irr_{\C}(G)$ let $V_{\chi}$ be a $\C[G]$-module with character $\chi$.
The $S$-truncated Artin $L$-function $L_{S}(s,\chi)$
is defined as the meromorphic extension to the whole complex plane 
of the holomorphic function given by the Euler product
\[
L_{S}(s,\chi) = \prod_{v \notin S} \det (1 - (\mathrm{N}v)^{-s}\sigma_{w} \mid V_{\chi}^{I_{w}})^{-1}, \quad \mathrm{Re}(s)>1.
\]
The primitive central idempotents of $\C[G]$ attached to elements of $\Irr_{\C}(G)$ form a $\C$-basis of
its centre $\zeta(\C[G])$ and thus there is a canonical isomorphism $\zeta(\C[G]) \cong \prod_{\chi \in \Irr_{\C}(G)} \C$.
The equivariant $S$-truncated Artin $L$-function is defined to be the meromorphic $\zeta(\C[G])$-valued function
\[
L_{S}(s) := (L_{S}(s,\chi))_{\chi \in \Irr_{\C}(G)}.
\]

Now suppose that $T$ is a second finite set of places of $K$ such that $S$ and $T$ are disjoint.
Then we define 
\begin{equation}\label{eq:def-delta-T}
\delta_{T}(s,\chi) := \prod_{v \in T} \det(1 - (\mathrm{N}v)^{1-s} \sigma_{w}^{-1} \mid V_{\chi}^{I_{w}}) \quad  \textrm{ and }  \quad
\delta_{T}(s) := (\delta_{T}(s,\chi))_{\chi \in \Irr_{\C}(G)}. 
\end{equation}
Let $x \mapsto x^{\#}$ denote the anti-involution on $\C[G]$ induced by $g \mapsto g^{-1}$ for $g \in G$.
The $(S,T)$-modified $G$-equivariant Artin $L$-function is defined to be
\[
\Theta_{S,T}(s) := \delta_{T}(s) \cdot L_{S}(s)^{\#}.
\]
Note that $L_{S}(s)^{\#} = (L_{S}(s,\check{\chi}))_{\chi \in \Irr_{\C} (G)}$, 
where $\check \chi$ denotes the character contragredient to $\chi$.
Evaluating $\Theta_{S,T}(s)$ at $s=0$ gives an $(S,T)$-modified Stickelberger element
\[
  \theta_{S}^{T} := \Theta_{S,T}(0) \in \zeta(\Q[G]).
\]
Note that a priori we only have $\theta_{S}^{T} \in \zeta(\C[G])$, but by a result of Siegel \cite{MR0285488} we know that
$\theta_{S}^{T}$ in fact belongs to $\zeta(\Q[G])$. 
If $T$ is empty, we abbreviate $\theta^{T}_{S}$ to $\theta_{S}$.
If the extension $L/K$ is not clear from context, we will also write
$\theta_{S}^{T}(L/K)$, $L_S(L/K,s)$, $\delta_{T}(L/K,s)$, etc.

\subsection{Ray class groups}\label{subsec:ray-class-groups}
Let $T$ be a finite set of finite places of $K$ and let $T(L)$ denote 
the set of places of $L$ above those in $T$. 
Let $I_{T}(L)$ denote the group of fractional ideals of $L$ relatively prime to $\mathfrak{M}_{L}^{T} := \prod_{w \in T(L)} \mathfrak{P}_{w}$.
Let $P_{T}(L)$ denote the subgroup of $I_{T}(L)$ generated by 
principal ideals $(\alpha)$ where $\alpha \in \mathcal{O}_{L}$ satisfies $\alpha \equiv 1 \bmod \mathfrak{M}_{L}^{T}$. Then
\[
\cl_{L}^{T} := I_{T}(L)/P_{T}(L)
\]
is the ray class group of $L$ associated to the modulus $\mathfrak{M}_{L}^{T}$.
We denote the group $\mathcal{O}_L^{\times}$ of units in $L$
by $E_{L}$ and define
$E_{L}^{T} := \left\{x \in E_{L}: x \equiv 1 \bmod \mathfrak{M}_{L}^{T} \right\}$.
If $T$ is empty we abbreviate $\cl_{L}^{T}$ to $\cl_{L}$.
All these modules are equipped with a natural $G$-action
and we have the following exact sequence of finitely generated $\Z[G]$-modules
\begin{equation}\label{eqn:ray-class-sequence}
0 \longrightarrow E_{L}^{T} \longrightarrow E_{L} \longrightarrow (\mathcal{O}_{L} / \mathfrak{M}_{L}^{T})\mal
\stackrel{\nu}{\longrightarrow} \cl_{L}^{T} \longrightarrow \cl_{L} \longrightarrow 0,
\end{equation}
where the map $\nu$ lifts an element $\overline x \in (\mathcal{O}_{L} / \mathfrak{M}_{L}^{T})^{\times}$ to
$x \in \mathcal{O}_{L}$ and
sends it to the ideal class $[(x)] \in \cl_{L}^{T}$ of the principal ideal $(x)$.

\subsection{The Brumer and Brumer--Stark conjectures for abelian extensions}\label{subsec:B-and-BS-for-abelian-extns}
We now specialise to the case in which $L/K$ is an abelian CM extension of number fields.
In other words,
$K$ is totally real and $L$ is a finite abelian extension of $K$ that is a CM field. 
Let $\mu_{L}$ and $\cl_{L}$ denote the roots of unity and the class group of $L$, respectively.

Let $S_{\ram}(L/K)$ be the set of all places of $K$ that ramify in $L/K$.
It was shown independently by Pi.\ Cassou-Nogu\`es \cite{MR524276} and by Deligne and Ribet 
\cite{MR579702} that 
\begin{equation}\label{eq:integrality-of-stickelbeger-element}
\Ann_{\Z[G]}(\mu_{L})\theta_{S} \subseteq \Z[G]. 
\end{equation}
Brumer's conjecture simply asserts that $\Ann_{\Z[G]}(\mu_{L})\theta_{S}$ annihilates $\cl_{L}$ and 
in the case $K=\Q$ this is essentially Stickelberger's theorem \cite{MR1510649}.

\begin{hypothesis*}
Let $S$ and $T$ be finite sets of places of $K$. 
We say that $\Hyp(S,T)=\Hyp(L/K,S,T)$ is satisfied if
(i)
$S_{\ram}(L/K) \cup S_{\infty} \subseteq S$,
(ii)
$S \cap T = \emptyset$, and
(iii)
$E_{L}^T$ is torsionfree.
\end{hypothesis*}

\begin{remark}\label{rmk:conditions-on-T}
Condition (iii) means that there are no non-trivial roots of unity of $L$ congruent to $1$ modulo all primes in $T(L)$. 
In particular, this forces $T$ to be non-empty and will be satisfied if $T$ contains primes of at least two different residue characteristics or at least one prime of sufficiently large norm. 
\end{remark}

If $S$ and $T$ are finite sets of places of $K$ satisfying $\Hyp(S,T)$ then \eqref{eq:integrality-of-stickelbeger-element}
implies that $\theta_{S}^{T} \in \Z[G]$. Moreover, given a finite set $S$ of places of $K$ such that
$S_{\ram}(L/K) \cup S_{\infty} \subseteq S$, Brumer's conjecture for $S$ holds if and only if  $\theta_{S}^{T} \in \Ann_{\Z[G]}(\cl_{L})$
for every finite set of places $T$ of $K$ such that $\Hyp(S,T)$ is satisfied (see \cite[Corollary 2.9]{nickel-conjectures}). 
The following strengthening of Brumer's conjecture was stated by Tate and
is known as the Brumer--Stark conjecture.

\begin{conj}\label{conj:Brumer--Stark}
For every pair $S,T$ of finite sets of places of $K$ satisfying $\Hyp(S,T)$ we have
$\theta_{S}^{T} \in \Ann_{\Z[G]}(\cl_{L}^{T})$.
\end{conj}

In fact, as explained in \cite[\S 1]{MR4513146}, Conjecture \ref{conj:Brumer--Stark} is slightly different from the actual statement proposed by Tate  \cite[Conjecture IV.6.2]{MR782485}, but it is the former that will be the most convenient for our purposes. We also note that Conjecture \ref{conj:Brumer--Stark} decomposes into local conjectures at each prime $p$ after replacing
$\cl_{L}^{T}$ by $\Z_{p} \otimes_{\Z} \cl_{L}^{T}$.

For generalisations of the Brumer--Stark conjecture
to not necessarily abelian extensions, we refer the interested reader
to the survey article \cite{nickel-conjectures}.

\subsection{The strong Brumer--Stark conjecture for abelian extensions}\label{subsec:strong-BS-for-abelian-extns}
Let $j$ denote the unique complex conjugation in $G$. 
For a $G$-module $M$ we write $M^{+}$ and $M^{-}$ for the submodules of $M$ upon which $j$ acts as $1$ and $-1$, respectively.
In particular, we shall be interested in $(\Z_{p} \otimes_{\Z} \cl_{L}^{T})^-$ for odd primes $p$; 
we will abbreviate this module to $A_{L}^{T}$ when $p$ is clear from context;
if $T$ is empty we further abbreviate this to $A_{L}$.
Note that $A_{L}^{T}$ and $(A_{L}^{T})^{\vee}$ are modules of finite cardinality over the ring $\Z_{p}[G]_{-} := \Z_{p}[G]/(1+j)$.
The following result was conjectured independently by Greither and by Kurihara \cite{MR4402386}
and is known as the strong Brumer--Stark conjecture;
it was recently proven in groundbreaking work of Dasgupta and Kakde \cite[Corollary 3.8]{MR4513146}
(in fact, they also prove an even stronger conjecture of Kurihara \cite[Conjecture 3.2]{MR4402386}).

\begin{theorem}\label{thm:strong-brumer-stark}
Let $p$ be an odd prime and let $S,T$ be finite sets of places of $K$ such that $\Hyp(S,T)$ is satisfied.
Then $(\theta_{S}^{T})^{\#} \in \Fitt_{\Z_{p}[G]_{-}}((A_{L}^{T})^{\vee})$. 
\end{theorem}

Theorem \ref{thm:strong-brumer-stark} can be seen as a refinement of the `$p$-part' of 
Conjecture \ref{conj:Brumer--Stark} (with $p$ odd), once we observe that: 
(i) the (initial) Fitting ideal of a module is contained in its annihilator; 
(ii) 
$
\Ann_{\Z_{p}[G]_{-}}(M) = \Ann_{\Z_{p}[G]_{-}}(M^{\vee})^{\#}
$
for every $\Z_{p}[G]_{-}$-module $M$ of finite cardinality; and (iii) $j$ acts as $-1$ on $\theta_{S}^{T}$, 
so the element $\theta_{S}^{T}$ annihilates a $\Z_{p}[G]$-module $M$ if and only if it annihilates $M^{-}$.

\begin{remark}\label{rmk:GK-counter-example-to-dual-stong-BS}
Greither and Kurihara \cite{MR2443336} have given counterexamples to 
the `dual' version of Theorem \ref{thm:strong-brumer-stark}, which asserts that 
$\theta_{S}^{T} \in \Fitt_{\Z_{p}[G]_{-}}(A_{L}^{T})$ under the same hypotheses. 
They have also given counterexamples to the assertion
$\theta_{S}^{\#} \in \Fitt_{\Z_{p}[G]_{-}}(A_{L}^{\vee})$ \cite[\S 0.1]{MR3589224}
(see also \cite{MR2918914}).
\end{remark}

\begin{remark}\label{rmk:only-weaker-version-of-SBS-needed}
For the proof of Theorem \ref{thm:EIMC-abelian-exts},
we shall only require a weaker version of
Theorem \ref{thm:strong-brumer-stark} with the additional hypothesis that $S$ contains
all the places of $K$ above $p$.
\end{remark}

\section{Algebraic $K$-theory and complexes}\label{sec:K-theory-complexes}

\subsection{Preliminaries on algebraic $K$-theory}\label{subsec:K-theory}
Let $\Lambda$ be a ring. We denote the category of finitely generated projective (left) $\Lambda$-modules by 
$\PMod(\Lambda)$.
We write $K_{0}(\Lambda)$ for the Grothendieck group of $\PMod(\Lambda)$ (see \cite[\S 38]{MR892316})
and $K_{1}(\Lambda)$ for the Whitehead group (see \cite[\S 40]{MR892316}).
Let $K_{0}(\Lambda, \Lambda')$ denote the relative algebraic $K$-group associated to a ring homomorphism
$\Lambda \rightarrow \Lambda'$.
We recall that $K_{0}(\Lambda, \Lambda')$ is an abelian group with generators $[X,g,Y]$ where
$X$ and $Y$ are objects of $\PMod(\Lambda)$
and $g:\Lambda' \otimes_{\Lambda} X \rightarrow \Lambda' \otimes_{\Lambda} Y$ is an isomorphism of $\Lambda'$-modules;
for a full description in terms of generators and relations, we refer the reader to \cite[p.\ 215]{MR0245634}.
Moreover, there is a long exact sequence of relative $K$-theory (see \cite[Chapter 15]{MR0245634})
\begin{equation}\label{eqn:long-exact-seq}
K_{1}(\Lambda) \longrightarrow K_{1}(\Lambda') \stackrel{\partial}{\longrightarrow} K_{0}(\Lambda, \Lambda')
\longrightarrow K_{0}(\Lambda) \longrightarrow K_{0}(\Lambda').
\end{equation}

\subsection{Complexes of modules over orders in separable algebras}\label{subsec:orders-sep-al-nr}
Let $R$ be a noetherian integral domain.
Let $A$ be a finite-dimensional separable
$Quot(R)$-algebra and let $\mathfrak{A}$ be an $R$-order in $A$.
Note that $\mathfrak{A}$ is both left and right noetherian, since $\Lambda$ is finitely generated over $R$.
The reduced norm map $\nr = \nr_{A}: A \rightarrow \zeta(A)$ is defined componentwise on the Wedderburn decomposition of $A$
and extends to matrix rings over $A$ (see \cite[\S 7D]{MR632548}); thus
it induces a map $K_{1}(A) \rightarrow \zeta(A)^{\times}$, which we also denote by $\nr$.

Let $\mathcal{C}^{b} (\PMod (\mathfrak{A}))$ be the category of 
bounded (cochain) complexes of finitely generated projective $\mathfrak{A}$-modules,
and let $\mathcal{C}^{b}\tor (\PMod (\mathfrak{A}))$ be the full subcategory of complexes
whose cohomology modules are $R$-torsion.
The relative algebraic $K$-group $K_{0}(\mathfrak{A}, A)$ identifies with the Grothendieck group whose generators are $[C^{\bullet}]$, where $C^{\bullet}$
is an object of $\mathcal{C}^{b}\tor(\PMod(\mathfrak{A}))$, and whose relations are as follows: $[C^{\bullet}] = 0$ if $C^{\bullet}$ is acyclic, and
$[C_{2}^{\bullet}] = [C_{1}^{\bullet}] + [C_{3}^{\bullet}]$ for every short exact sequence
\begin{equation}\label{eq:SES-of-complexes}
0 \longrightarrow C_{1}^{\bullet} \longrightarrow C_{2}^{\bullet} \longrightarrow C_{3}^{\bullet} \longrightarrow 0
\end{equation}
in $\mathcal C^{b}\tor(\PMod(\mathfrak{A}))$ (see \cite[Chapter 2]{MR3076731} or \cite[\S 2]{MR3068893}, for example). 

Let $\mathcal{D} (\mathfrak{A})$ be the derived category of $\mathfrak{A}$-modules.
A complex of $\mathfrak{A}$-modules is said to be perfect if it is
isomorphic in $\mathcal{D} (\mathfrak{A})$ to an element of $\mathcal{C}^{b}(\PMod (\mathfrak{A}))$.
We denote the full triangulated subcategory of
$\mathcal{D} (\mathfrak{A})$ comprising perfect complexes by $\mathcal{D}^{\perf} (\mathfrak{A})$,
and the full triangulated subcategory
comprising perfect complexes whose cohomology modules are $R$-torsion by $\mathcal{D}^{\perf}\tor (\mathfrak{A})$.
Then any object of $\mathcal{D}^{\perf}\tor (\mathfrak{A})$ defines an element in $K_{0}(\mathfrak{A}, A)$.
In particular, a finitely generated $R$-torsion $\mathfrak{A}$-module $M$ of finite projective dimension 
considered as a complex concentrated in degree $0$ defines an element $[M]$ in $K_{0}(\mathfrak{A}, A)$.

For any integer $n$ and any cochain complex $C^{\bullet}$ of $\mathfrak{A}$-modules we write
$C[n]^{\bullet}$ for the $n$-shifted complex 
given by $C[n]^{i}=C^{n+i}$ with differential $d_{C[n]}^{i}=(-1)^{n}d_{C}^{n+i}$.
Note that if $C^{\bullet} \in \mathcal{D}^{\perf}\tor (\mathfrak{A})$
then $[C[n]^{\bullet}] = (-1)^n [C^{\bullet}]$ in $K_{0}(\mathfrak{A}, A)$.

The following result is certainly well known to
experts, but there does not appear to be a precise reference in the literature
(it essentially follows from \cite[Proposition 3.1]{MR2200204}).

\begin{prop}\label{prop:perfect-cohomology-fpd}
Suppose that every finitely generated $\mathfrak{A}$-module is of finite projective dimension.
Then for each object $C^{\bullet}$ of
$\mathcal{D}^{\perf}\tor (\mathfrak{A})$ we have
$[C^{\bullet}]=\sum_{i \in \Z} (-1)^{i}[H^{i}(C^{\bullet})]$ in
$K_{0}(\mathfrak{A},A)$.
\end{prop}

\begin{proof}
Note that \cite[Chapter XX, \S 1, Proposition 1.1]{MR1878556} and the hypothesis on $\mathfrak{A}$
imply that every bounded (cochain) complex of finitely generated
$\mathfrak{A}$-modules is perfect.
We claim that the complex $C^{\bullet}$ is isomorphic in $\mathcal{D}(\mathfrak{A})$
to a bounded complex $D^{\bullet}$ of finitely generated $R$-torsion modules.
Then $[C^{\bullet}] = [D^{\bullet}]$ in $K_{0}(\mathfrak{A},A)$
and thus we can and do assume that $C^{\bullet} = D^{\bullet}$.
We write $ZC^{\bullet}$, $BC^{\bullet}$ and $HC^{\bullet}$ for the complexes
of cocycles, coboundaries and cohomologies of $C^{\bullet}$, respectively,
each with zero differentials.
Note that all these complexes are in  $\mathcal{D}^{\perf}\tor (\mathfrak{A})$ 
and hence define elements in $K_{0}(\mathfrak{A},A)$.
Now the short exact sequences of complexes
\[
0 \longrightarrow BC^{\bullet} \longrightarrow ZC^{\bullet} \longrightarrow 
HC^{\bullet} \longrightarrow 0, \qquad 0 \longrightarrow
ZC^{\bullet} \longrightarrow C^{\bullet} \longrightarrow BC^{\bullet}[1] \longrightarrow 0
\]
imply that we have
\[
[C^{\bullet}] = [ZC^{\bullet}] - [BC^{\bullet}] = [HC^{\bullet}] = \sum_{i \in \Z} (-1)^{i}[H^{i}(C^{\bullet})]
\]
in $K_{0}(\mathfrak{A},A)$, as desired. 

It remains to show the claim.
We can and do assume that $C^{\bullet}$ is of the form
\[
\cdots \rightarrow 0 \rightarrow P^0 
\rightarrow \dots \rightarrow P^{j-2} \xrightarrow{d^{j-2}} P^{j-1} 
\xrightarrow{d^{j-1}} P^j \xrightarrow{d^j} T^{j+1} \rightarrow
T^{j+2} \rightarrow \dots \rightarrow T^n \rightarrow 0 \rightarrow \cdots
\]
for some integers $j \leq n$, where each $P^{i}$ and each $T^{i}$ is a finitely generated 
$\mathfrak{A}$-module placed in degree $i$, the $P^{i}$ are projective
and the $T^{i}$ are $R$-torsion (one may always take $j=n$ so that all modules
in the complex are in fact projective). 
We now do downward induction on $j$ to show the claim.
Since both $T^{j+1}$ and $H^j(C^{\bullet})$ are $R$-torsion, there is a
non-zero $r \in R$ such that $rP^j \subseteq \im(d^{j-1})$. As $P^j$ is projective,
there is a necessarily injective map $\varphi: P^j \rightarrow P^{j-1}$
such that $d^{j-1} \circ \varphi$ is multiplication by $r$.  
We set $T^j := P^j/rP^j$. If $j>1$ we define a complex
\[
(\tilde C)^{\bullet}: \cdots \rightarrow P^{j-2} \oplus P^{j} \xrightarrow{\tilde d^{j-2}}
P^{j-1} \xrightarrow{\tilde d^{j-1}} T^{j} \stackrel{\tilde d^j}{\longrightarrow} T^{j+1}
\rightarrow \cdots
\]
where $\tilde d^{j-1}$ and $\tilde d^j$ are induced from $d^{j-1}$ and $d^j$,
respectively, and $\tilde d^{j-2} = (d^{j-2}, \varphi)$. 
Now a straightforward, but rather lengthy diagram chase shows that
the canonical map of complexes $C^{\bullet} \rightarrow (\tilde C)^{\bullet}$ is a quasi-isomorphism.
For this observe that $\varphi(P^j) \cap \ker(d^{j-1}) = 0$ by construction
and hence $\ker(\tilde d^{j-2}) \cong \ker(d^{j-2})$.
For $j=1$ the argument is slightly different.
We set $T^{0} := P^{0} / \varphi(P^{1})$ and let $\tilde d^{0}: T^{0} \rightarrow T^{1}$ be induced from $d^{0}$. 
Then the canonical map of complexes is again a quasi-isomorphism. 
Moreover, since the $\mathfrak{A}$-modules $T^{1}$ and
$\ker(\tilde d^{0}) \simeq H^{0}(C^{\bullet})$ are $R$-torsion, so is $T^{0}$.
\end{proof}

Again, the following result is surely well known to experts, but there does not appear to be a precise reference in the literature.

\begin{prop}\label{prop:perfect-cohomology-acyclic-outside-one-degree}
Let $k \in \Z$ and suppose that $C^{\bullet}$ is an object of
$\mathcal{D}^{\perf}\tor (\mathfrak{A})$ such that $H^{i}(C^{\bullet}) = 0$ for all $i \in \Z - \{ k \}$.
Then $H^{k}(C^{\bullet})$ is of finite projective dimension over $\mathfrak{A}$
and $[C^{\bullet}]=(-1)^{k}[H^{k}(C^{\bullet})]$ in $K_{0}(\mathfrak{A},A)$. 
\end{prop}

\begin{proof}
We can and do assume without loss of generality that $k>0$ and that $C^{\bullet}$ is of the form
\[
\cdots \rightarrow 0 \rightarrow P^{0} 
\rightarrow \dots \rightarrow P^{k-1} 
\xrightarrow{d^{k-1}} P^{k} \stackrel{d^k}{\longrightarrow} P^{k+1} \rightarrow
\dots \rightarrow P^{n-1} \xrightarrow{d^{n-1}} P^{n} \rightarrow 0 \rightarrow \cdots,
\]
where $k \leq n$ and each $P^{i}$ is a finitely generated projective $\mathfrak{A}$-module
placed in degree $i$.
If $k < n$ then $C^{\bullet}$ is exact in degree $n$, so we have the short exact sequence 
\[
0 \longrightarrow \ker(d^{n-1}) \longrightarrow P^{n-1} \xrightarrow{d^{n-1}} P^{n} \longrightarrow 0,
\] 
and since $P^{n}$ is projective we conclude that $\ker(d^{n-1})$ is also projective.
Thus without loss of generality we can and do replace $P^{n}$ by $0$ and $P^{n-1}$ by $\ker(d^{n-1})$.
Therefore by downward induction on $n$ we can and do assume that in fact $k=n$.
This immediately implies that $H^{k}(C^{\bullet})$ is of finite projective dimension over $\mathfrak{A}$.
Moreover, it is straightforward to check that the canonical morphism of complexes from $C^{\bullet}$
to the complex consisting of $H^{k}(C^{\bullet})$ concentrated in degree $k$ is a quasi-isomorphism.
\end{proof}

\section{Algebraic $K$-theory for Iwasawa algebras}\label{sec:K-theory-Iwasawa-algebras}

\subsection{Iwasawa algebras of admissible $p$-adic Lie groups of dimension one}\label{subsec:Iwasawa-algebras}
Let $p$ be a prime and let $\mathcal{G}$ be a one-dimensional $p$-adic Lie group. 
Suppose that $\mathcal{G}$ is admissible, which we define to mean that 
$\mathcal{G}$ contains a finite normal subgroup $H$ 
such that $\overline{\Gamma} := \mathcal{G}/H$ is a pro-$p$-group isomorphic to $\Z_{p}$.
Then $\mathcal{G}$ is compact and
the argument given in \cite[\S 1]{MR2114937} shows that the short exact sequence
\[
1 \longrightarrow H \longrightarrow \mathcal{G} \longrightarrow \overline{\Gamma} \longrightarrow 1
\]
splits. Thus we obtain a semidirect product $\mathcal{G} = H \rtimes \Gamma$ where $\Gamma \leq \mathcal{G}$ and $\Gamma \simeq \overline{\Gamma} \simeq \Z_{p}$.
Note that the image under the canonical projection map
$\mathcal{G} \twoheadrightarrow \overline{\Gamma}$ 
of any element of $\mathcal{G}$ of finite order is also of finite order and hence must be trivial.
Thus $H$ is equal to the subset of $\mathcal{G}$ of elements of finite order.
Therefore $H$ and $\overline{\Gamma}$ are uniquely determined by $\mathcal{G}$, though the choice of 
$\Gamma$ need not be. 
We fix a topological generator $\gamma$ of $\Gamma$. 
Let $\overline{\gamma} := \gamma \bmod H$ and note that this a topological generator of 
$\overline{\Gamma}$.
Since any homomorphism $\Gamma \rightarrow \Aut(H)$ must have open kernel, we may choose a 
non-negative integer $n$ such that $\gamma^{p^n}$ is central in $\mathcal{G}$. 
We fix such an $n$ and put $\Gamma_{0} := \Gamma^{p^n} \simeq \Z_{p}$.

The Iwasawa algebra of $\mathcal{G}$ is
$
\Lambda(\mathcal{G}) := \Z_{p}\llbracket\mathcal{G}\rrbracket = \varprojlim \Z_{p}[\mathcal{G}/\mathcal{N}]
$,
where the inverse limit is taken over all open normal subgroups $\mathcal{N}$ of $\mathcal{G}$.
Let $F$ be a finite field extension of $\Q_{p}$  with ring of integers $\mathcal{O}=\mathcal{O}_{F}$,
and put 
\[
\Lambda^{\mathcal{O}}(\mathcal{G}) := \mathcal{O} \otimes_{\Z_{p}} \Lambda(\mathcal{G}) = \mathcal{O}\llbracket\mathcal{G}\rrbracket.
\]
Since $\Gamma_{0} \simeq \Z_{p}$, there is a ring isomorphism
$R:=\mathcal{O}\llbracket\Gamma_{0}\rrbracket \simeq 
\mathcal{O}\llbracket T \rrbracket$ induced by $\gamma^{p^n} \mapsto 1+T$
where $\mathcal{O}\llbracket T \rrbracket$ denotes the power series ring in one variable over $\mathcal{O}$.
If we view $\Lambda^{\mathcal{O}}(\mathcal{G})$ as an $R$-module (or indeed as an $R[H]$-module), there is a decomposition
\begin{equation}\label{eq:Lambda-R-decomp}
\Lambda^{\mathcal{O}}(\mathcal{G}) = \bigoplus_{i=0}^{p^{n}-1} R[H] \gamma^{i}.
\end{equation}
Hence $\Lambda^{\mathcal{O}}(\mathcal{G})$ is free of finite rank as an $R$-module and is an $R$-order in the separable $Quot(R)$-algebra
$\mathcal{Q}^{F} (\mathcal{G})$, the total ring of fractions of $\Lambda^{\mathcal{O}}(\mathcal{G})$, obtained
from $\Lambda^{\mathcal{O}}(\mathcal{G})$ by adjoining inverses of all central regular elements.
Note that $\mathcal{Q}^{F} (\mathcal{G}) =  Quot(R) \otimes_{R} \Lambda^{\mathcal{O}}(\mathcal{G})$ and that by
\cite[Lemma 1]{MR2114937} we have $\mathcal{Q}^{F} (\mathcal{G}) = F \otimes_{\Q_{p}} \mathcal{Q}(\mathcal{G})$,
where $\mathcal{Q}(\mathcal{G}) := \mathcal{Q}^{\Q_{p}}(\mathcal{G})$.

\subsection{Algebraic $K$-theory for Iwasawa algebras}\label{subsec:K-theory-Iwasawa-algebras}
We now specialise \S \ref{subsec:K-theory} to the situation of \S \ref{subsec:Iwasawa-algebras}.
Let $p$ be a prime and let $\mathcal{G}$ be an admissible 
one-dimensional $p$-adic Lie group.
Let $\Gamma_{0}$ be an open subgroup of $\Gamma$ that is central in $\mathcal{G}$ and
let $F$ be a finite field extension of $\Q_{p}$ with ring of integers $\mathcal{O}=\mathcal{O}_{F}$.
Let $A = \mathcal{Q}^{F}(\mathcal{G})$, let $\mathfrak{A} = \Lambda^{\mathcal{O}}(\mathcal{G})=\mathcal{O}\llbracket\mathcal{G}\rrbracket$ and
let $R=\mathcal{O}\llbracket\Gamma_{0}\rrbracket$.
In this situation, \cite[Corollary 3.8]{MR3034286}
shows that the map $\partial$ in \eqref{eqn:long-exact-seq}
is surjective;
thus we obtain an exact sequence
\begin{equation}\label{eqn:Iwasawa-K-sequence}
K_{1}(\Lambda^{\mathcal{O}}(\mathcal{G})) \longrightarrow K_{1}(\mathcal{Q}^{F}(\mathcal{G})) \stackrel{\partial}{\longrightarrow}
K_{0}(\Lambda^{\mathcal{O}}(\mathcal{G}),\mathcal{Q}^{F}(\mathcal{G})) \longrightarrow 0.
\end{equation}

For a finite normal subgroup $N$ of $\mathcal{G}$, 
there are canonical maps $\quot^{\mathcal{G}}_{\mathcal{G}/N}$ that 
fit into the following commutative diagram
\begin{equation*}\label{eqn:quotient-on-K-groups} 
\xymatrix{
K_{1}(\Lambda^{\mathcal{O}}(\mathcal{G})) \ar[r] \ar[d]_{\quot^{\mathcal{G}}_{\mathcal{G}/N}}  & K_{1}(\mathcal{Q}^{F}(\mathcal{G})) \ar[r]^{\partial \qquad} \ar[d]_{\quot^{\mathcal{G}}_{\mathcal{G}/N}} &
K_{0}(\Lambda^{\mathcal{O}}(\mathcal{G}),\mathcal{Q}^{F}(\mathcal{G})) \ar[d]_{\quot^{\mathcal{G}}_{\mathcal{G}/N}} \\
K_{1}(\Lambda^{\mathcal{O}}(\mathcal{G}/N)) \ar[r]  & K_{1}(\mathcal{Q}^{F}(\mathcal{G}/N)) \ar[r]^{\partial \qquad} &
K_{0}(\Lambda^{\mathcal{O}}(\mathcal{G}/N),\mathcal{Q}^{F}(\mathcal{G}/N)).
} 
\end{equation*}

Similarly, for an open subgroup $\mathcal{H}$ of $\mathcal{G}$,
there are canonical maps $\res^{\mathcal{G}}_{\mathcal{H}}$
that fit into the following commutative diagram
\begin{equation*}\label{eqn:restriction-on-K-groups} 
\xymatrix{
K_{1}(\Lambda^{\mathcal{O}}(\mathcal{G})) \ar[r] \ar[d]_{\res^{\mathcal{G}}_{\mathcal{H}}} & K_{1}(\mathcal{Q}^{F}(\mathcal{G})) \ar[r]^{\partial \qquad} \ar[d]_{\res^{\mathcal{G}}_{\mathcal{H}}} &
K_{0}(\Lambda^{\mathcal{O}}(\mathcal{G}),\mathcal{Q}^{F}(\mathcal{G})) \ar[d]_{\res^{\mathcal{G}}_{\mathcal{H}}} \\
K_{1}(\Lambda^{\mathcal{O}}(\mathcal{H})) \ar[r]  & K_{1}(\mathcal{Q}^{F}(\mathcal{H})) \ar[r]^{\partial \qquad} &
K_{0}(\Lambda^{\mathcal{O}}(\mathcal{H}),\mathcal{Q}^{F}(\mathcal{H})).
}
\end{equation*}

\subsection{Characters and central primitive idempotents}\label{subsec:idempotents}
Fix a character $\chi \in \Irr_{\Q_{p}^{c}}(\mathcal{G})$ 
(i.e.\ an irreducible $\Q_{p}^{c}$-valued character of $\mathcal{G}$ with open kernel)
and let $\eta$ be an irreducible constituent of
$\res^{\mathcal{G}}_{H} \chi$.
Then $\mathcal{G}$ acts on $\eta$ as $\eta^{g}(h) = \eta(g^{-1}hg)$
for $g \in \mathcal{G}$, $h \in H$, and following \cite[\S 2]{MR2114937} we set
\[
St(\eta) := \{g \in \mathcal{G}: \eta^g = \eta \}, \quad e(\eta) := \frac{\eta(1)}{|H|} \sum_{h \in H} \eta(h^{-1}) h,
\quad e_{\chi} := \sum_{\eta \mid \res^{\mathcal{G}}_{H} \chi} e(\eta).
\]
By \cite[Corollary to Proposition 6]{MR2114937} $e_{\chi}$ is a primitive central idempotent of
$\mathcal{Q}^{c}(\mathcal{G}) := \Q_{p}^{c} \otimes_{\Q_{p}} \mathcal{Q}(\mathcal{G})$.
In fact, every primitive central idempotent of $\mathcal{Q}^{c}(\mathcal{G})$ is of this form
and $e_{\chi} = e_{\chi'}$ if and only if $\chi = \chi' \otimes \rho$ for some character $\rho$ of $\mathcal{G}$ of type $W$
(i.e.~$\res^{\mathcal{G}}_{H} \rho = 1$).
Let $w_{\chi} = [\mathcal{G} : St(\eta)]$ and note that this is a power of $p$ since $H$ is a subgroup of $St(\eta)$.

Let $E$ be a finite field extension of $\Q_{p}$ over which both characters $\chi$ and $\eta$ have realisations
and let $V_{\chi}$ denote a realisation of $\chi$ over $E$.	
By \cite[Propositions 5 and 6]{MR2114937} and \cite[Lemma 3.1]{MR3749195}, there exists a unique element $\gamma_{\chi} \in \zeta(\mathcal{Q}^{E}(\mathcal{G})e_{\chi})$ 
such that $\gamma_{\chi}$ acts trivially on $V_{\chi}$ and $\gamma_{\chi} = g_{\chi}c_{\chi}$ where $g_{\chi} \in \mathcal{G}$ with $(g_{\chi} \bmod H) = \overline{\gamma}^{w_{\chi}}$
and with $c_{\chi} \in (E[H]e_{\chi})^{\times}$.
Moreover, $\gamma_{\chi}$ generates a pro-cyclic $p$-subgroup $\Gamma_{\chi}$ of $\mathcal{Q}^{E}(\mathcal{G})e_{\chi}$ and induces an isomorphism
\[
\mathcal{Q}^{E}(\Gamma_{\chi}) \cong \zeta(\mathcal{Q}^{E} (\mathcal{G})e_{\chi}).
\]

\subsection{Determinants and reduced norms}\label{subsec:dets-and-nr}
Following \cite[Proposition 6]{MR2114937}, we define
\[
j_{\chi}: \zeta(\mathcal{Q}^{E} (\mathcal{G})) \twoheadrightarrow \zeta(\mathcal{Q}^{E} (\mathcal{G})e_{\chi}) \cong \mathcal{Q}^{E}(\Gamma_{\chi}) \longrightarrow  \mathcal{Q}^{E}(\overline{\Gamma}),
\]
where the last arrow is induced by mapping $\gamma_{\chi}$ to $\overline{\gamma}^{w_{\chi}}$.
It follows from loc.\ cit.\ that $j_{\chi}$ is independent of the choice of $\overline{\gamma}$.

Let $F$ be any finite field extension of $\Q_{p}$ and let $G_{F}=\Gal(F^{c}/F)$.
By enlarging $E$ if necessary, we can and do assume that $F$ is a subfield of $E$.
Define
\begin{eqnarray*}
\Det(~)(\chi): K_{1}(\mathcal{Q}^{F}(\mathcal{G})) & \longrightarrow & \mathcal{Q}^{E}(\overline{\Gamma})^{\times} \\
 {[P,\alpha]}& \longmapsto & \mathrm{det}_{\mathcal{Q}^{E}(\overline{\Gamma})} (\alpha \mid \Hom_{E[H]}(V_{\chi},  E \otimes_{F} P)),
\end{eqnarray*}
where $P$ is a finitely generated projective $\mathcal{Q}^{F}(\mathcal{G})$-module and $\alpha$ is a 
$\mathcal{Q}^{F}(\mathcal{G})$-automorphism of $P$.
Here $\alpha$ acts on $f \in \Hom_{E[H]}(V_{\chi}, E \otimes_{F} P)$ 
via its action on $E \otimes_{F} P$, and $\overline{\gamma}$ acts via 
$(\overline{\gamma} f)(v) = \gamma \cdot f(\gamma^{-1} v)$ for all $v \in V_{\chi}$,
which is easily seen to be independent of the choice of $\gamma$.
Let  $\nr : K_{1}(\mathcal{Q}^{F}(\mathcal{G})) \rightarrow \zeta(\mathcal{Q}^{E} (\mathcal{G}))^{\times}$
be the map induced by the reduced norm.
Then $\Det(~)(\chi)$ is just $j_{\chi} \circ \nr$ (see \cite[\S 3, p.\ 558]{MR2114937} for more details).

If $\rho$ is a character of $\mathcal{G}$ of type $W$,
then we denote by
$\rho^{\#}$ the automorphism of the field $\mathcal{Q}^{c}(\overline{\Gamma})$ induced by
$\rho^{\#}(\overline{\gamma}) = \rho(\overline{\gamma}) \overline{\gamma}$. 
We denote the additive group generated by all $\Q_{p}^{c}$-valued
characters of $\mathcal{G}$ with open kernel by $R_{p}(\mathcal{G})$ and
equip this with a Galois action defined by ${}^{\sigma}\chi(g) = \sigma(\chi(g))$
for $g \in \mathcal{G}$ and $\sigma \in G_{\Q_{p}}$. 
Following \cite[Theorem 8]{MR2114937}, we define
$\Hom^{\ast}_{G_F}(R_{p}( \mathcal{G}), \mathcal{Q}^{c}(\overline{\Gamma})^{\times})$
to be the group of all homomorphisms 
$f: R_{p}(\mathcal{G}) \rightarrow \mathcal{Q}^{c}(\overline{\Gamma})\mal$ satisfying
\[
\begin{array}{ll}
f(\chi \otimes \rho) = \rho^{\#}(f(\chi)) & \mbox{ for all characters } \rho \mbox{ of type } W, \mbox{ and}\\
f({}^{\sigma}\chi) = \sigma(f(\chi)) & \mbox{ for all Galois automorphisms } \sigma \in G_{F}.
\end{array}
\]
By applying \cite[Theorem 7]{MR2114937} and taking $G_{F}$-invariants as in the proof of \cite[Theorem 8]{MR2114937}, we obtain an isomorphism
\begin{eqnarray*}
\zeta(\mathcal{Q}^{F}(\mathcal{G}))\mal & \cong & 
\Hom^{\ast}_{G_{F}}(R_{p}(\mathcal{G}), \mathcal{Q}^{c}(\overline{\Gamma})^{\times})\\
x & \mapsto & [\chi \mapsto j_{\chi}(x)].
\end{eqnarray*}
As $\Det(-)(\chi)$ is just the composite map $j_{\chi} \circ \nr$, the map
which sends $\Theta \in K_{1}(\mathcal{Q}^F(\mathcal{G}))$ to 
$[\chi \mapsto \Det(\Theta)(\chi)]$
defines a homomorphism
\[
\Det: K_{1}(\mathcal{Q}^F(\mathcal{G})) \longrightarrow \Hom^{\ast}_{G_{F}}(R_{p}(\mathcal{G}), \mathcal{Q}^{c}(\overline{\Gamma})\mal)
\]
such that we obtain a commutative triangle
\begin{equation}\label{eqn:Det_triangle}
\begin{gathered}
\xymatrix{
& K_{1}(\mathcal{Q}^{F}(\mathcal{G})) \ar[dl]_{\nr} \ar[dr]^{\Det} &\\
{\zeta(\mathcal{Q}^{F}(\mathcal{G}))^{\times}} \ar[rr]^-{\cong} & & {\Hom^{\ast}_{G_{F}}(R_{p}( \mathcal{G}), \mathcal{Q}^{c}(\overline{\Gamma})^{\times})}.} 
\end{gathered}
\end{equation}

Let $N$ be a finite normal subgroup of $\mathcal{G}$. 
Following \cite[\S 3]{MR2114937}, we define a map
\[
\quot^{\mathcal{G}}_{\mathcal{G}/N}:  \Hom^{\ast}_{G_{F}}(R_{p}(\mathcal{G}), \mathcal{Q}^{c}(\overline{\Gamma})\mal) \longrightarrow
\Hom^{\ast}_{G_{F}}(R_{p}(\mathcal{G}/N), \mathcal{Q}^{c}(\overline{\Gamma})\mal),
\]
by $(\quot^{\mathcal{G}}_{\mathcal{G}/N} f)(\chi) := f(\infl^{\mathcal{G}}_{\mathcal{G}/N} \chi)$
for $f \in \Hom^{\ast}_{G_F}(R_{p}(\mathcal{G}), \mathcal{Q}^{c}(\overline{\Gamma})\mal)$
and $\chi \in R_{p}(\mathcal{G}/N)$.

Let $\mathcal{H}$ be an open subgroup of $\mathcal{G}$. 
As explained in \S \ref{subsec:Iwasawa-algebras}, the subset $H'$ of $\mathcal{H}$
of elements of finite order is in fact a finite normal subgroup 
such that
$\overline{\Gamma}_{\mathcal{H}} := \mathcal{H}/H'$ is a pro-$p$-group isomorphic to $\Z_{p}$.
Moreover, there is a canonical embedding
$\iota_{\mathcal{H}} : \overline{\Gamma}_{\mathcal{H}} \hookrightarrow \overline{\Gamma}$
defined as follows:
given any element $x \in \overline{\Gamma}_{\mathcal{H}}$, let $y \in \mathcal{H}$ be any lift and define
$\iota_{\mathcal{H}}(x)$ to be the image of $y$ under the composition of canonical maps
$\mathcal{H} \hookrightarrow \mathcal{G} \twoheadrightarrow \overline{\Gamma}$.
It is straightforward to check that this map is well defined. 
Again following \cite[\S 3]{MR2114937}, we define a map
\[
\res^{\mathcal{G}}_{\mathcal{H}}: \Hom^{\ast}_{G_{F}}(R_{p}(\mathcal{G}), \mathcal{Q}^{c}(\overline{\Gamma})\mal) \longrightarrow
\Hom^{\ast}_{G_{F}}(R_{p}(\mathcal{H}), \mathcal{Q}^{c}(\overline{\Gamma}_{\mathcal{H}})\mal),
\]
by $(\res^{\mathcal{G}}_{\mathcal{H}} f)(\chi') = f(\ind^{\mathcal{G}}_{\mathcal{H}} \chi')$  for $f \in \Hom^{\ast}_{G_F}(R_{p}(\mathcal{G}), \mathcal{Q}^{c}(\overline{\Gamma})\mal)$
and $\chi' \in R_{p}(\mathcal{H})$. Here we view $\mathcal{Q}^{c}(\overline{\Gamma}_{\mathcal{H}})$
as a subfield of $\mathcal{Q}^{c}(\overline{\Gamma})$ via 
the embedding
$\iota_{\mathcal{H}} : \overline{\Gamma}_{\mathcal{H}} \hookrightarrow \overline{\Gamma}$. 

Via diagram \eqref{eqn:Det_triangle} the maps just defined induce canonical 
group homomorphisms
\begin{eqnarray*}\label{eq:quot-res-maps-on-centres}
\quot^{\mathcal{G}}_{\mathcal{G}/N}: & \zeta(\mathcal{Q}^{F}(\mathcal{G}))\mal \longrightarrow &
\zeta(\mathcal{Q}^{F}(\mathcal{G}/N))\mal,\\
\res^{\mathcal{G}}_{\mathcal{H}}: & \zeta(\mathcal{Q}^{F}(\mathcal{G}))^{\times} \longrightarrow &
\zeta(\mathcal{Q}^{F} (\mathcal{H}))\mal.
\end{eqnarray*}
The first map is easily seen to be induced by the canonical projection
$\mathcal{Q}^{F}(\mathcal{G}) \rightarrow  \mathcal{Q}^{F}(\mathcal{G}/N)$.
Moreover, by (an obvious generalisation of) 
\cite[Lemma 9]{MR2114937} we have commutative diagrams
\begin{equation*}\label{eqn:quot-res-maps-nr-diagram} 
\xymatrix{
K_{1}(\mathcal{Q}^{F}(\mathcal{G}))  \ar[r]^{\nr} \ar[d]_{\quot^{\mathcal{G}}_{\mathcal{G}/N}}  &
\zeta(\mathcal{Q}^{F}(\mathcal{G}))^{\times} \ar[d]_{\quot^{\mathcal{G}}_{\mathcal{G}/N}} & &
K_{1}(\mathcal{Q}^{F}(\mathcal{G}))  \ar[r]^{\nr} \ar[d]_{\res^{\mathcal{G}}_{\mathcal{H}}} &
\zeta(\mathcal{Q}^{F}(\mathcal{G}))^{\times} \ar[d]_{\res^{\mathcal{G}}_{\mathcal{H}}} \\
K_{1}(\mathcal{Q}^{F}(\mathcal{G}/N)) \ar[r]^{\nr}  & \zeta(\mathcal{Q}^{F}(\mathcal{G}/N))^{\times} & &
K_{1}(\mathcal{Q}^{F}(\mathcal{H}))  \ar[r]^{\nr} & \zeta(\mathcal{Q}^{F}(\mathcal{H}))^{\times}.
} 
\end{equation*}

\section{The equivariant Iwasawa main conjecture}

\subsection{Admissible one dimensional $p$-adic Lie extensions}\label{subsec:admissible-one-dim-extns}
Let $p$ be an odd prime and let $K$ be a totally real number field.
We henceforth assume that $\mathcal{L}/K$ is an admissible one-dimensional $p$-adic Lie extension.
In other words, $\mathcal{L}$ is a Galois extension of $K$
such that (i) $\mathcal{L}$ is totally real,
(ii) $\mathcal{L}$ contains the cyclotomic $\Z_{p}$-extension $K_{\infty}$ of $K$, and
(iii) $[\mathcal{L} : K_{\infty}]$ is finite.
Let $\mathcal{G}=\Gal(\mathcal{L}/K)$, let $H=\Gal(\mathcal{L}/K_{\infty})$ and let $\Gamma_{K}=\Gal(K_{\infty}/K)$.
Let $\gamma_{K}$ be a topological generator of $\Gamma_{K}$.
As in \S \ref{subsec:Iwasawa-algebras}, we obtain a semidirect product 
$\mathcal{G} = H \rtimes \Gamma$ where $\Gamma \leq \mathcal{G}$ and $\Gamma \simeq \Gamma_{K} \simeq \Z_{p}$, 
and we choose an open subgroup $\Gamma_{0} \leq \Gamma$ that is central in $\mathcal{G}$. 
Let $R=\Z_{p}\llbracket \Gamma_{0} \rrbracket$ and let
$\Lambda(\mathcal{G}) = \Z_{p}\llbracket \mathcal{G} \rrbracket$.

\subsection{An Iwasawa module and the $\mu=0$ hypothesis}\label{subsec:Iwasawa-module}
Let $S_{\infty}$ be the set of infinite
 places of $K$ and let $S_{p}$ be the set of places of $K$ above $p$.
Let $S$ be a finite set of places of $K$ containing $S_{p} \cup S_{\infty}$.
Let $M_{S}^{\ab}(p)$ be the maximal abelian pro-$p$-extension of 
$\mathcal{L}$ unramified outside $S$ and let
\[
X_{S}=X_{S}(\mathcal{L}/K)=\Gal(M_{S}^{\ab}(p)/\mathcal{L}).
\] 
As usual $\mathcal{G}$ acts on $X_{S}$ by $g \cdot x = \tilde{g}x\tilde{g}^{-1}$, 
where $g \in \mathcal{G}$, and $\tilde{g}$ is any lift of $g$ to $\Gal(M_{S}^{\ab}(p)/K)$. 
This action extends to a left action of $\Lambda(\mathcal{G})$ on $X_{S}$.
Since $\mathcal{L}$ is totally real, \cite[Theorems 10.3.25 and 11.3.2]{MR2392026} show that,
as an $R$-module, $X_{S}$ is finitely generated, torsion and of projective dimension at most one.
If $\mathcal{G}$ contains no element of order $p$ then 
$X_{S}$ is also of projective dimension at most one over $\Lambda(\mathcal{G})$.
In general, however, $X_{S}$ is not of finite projective dimension as a $\Lambda(\mathcal{G})$-module.

\begin{definition}\label{def:mu=0-hypothesis}
We say that $\mathcal{L}/K$ satisfies the $\mu=0$ hypothesis if $X_{S}$
is finitely generated as a $\Z_{p}$-module.
\end{definition}

The $\mu=0$ hypothesis is independent of the choice of $S$ and is conjecturally always true. 
Moreover, it is known to hold when $\mathcal{L}/\Q$ is abelian
as follows from work of Ferrero and Washington \cite{MR528968}.
For the relation to the classical Iwasawa $\mu = 0$ conjecture see \cite[Remark 4.3]{MR3749195}, for instance.
In the sequel, we shall \emph{not} assume the $\mu=0$ hypothesis for $\mathcal{L}/K$, except where explicitly stated.

\subsection{The $p$-adic cyclotomic character and its projections}\label{subsec:cyclotomic-char}
Let $\chi_{\mathrm{cyc}}$ be the $p$-adic cyclotomic character
\[
\chi_{\mathrm{cyc}}: \Gal(\mathcal{L}(\zeta_{p})/K) \longrightarrow \Z_{p}^{\times},
\]
defined by $\sigma(\zeta) = \zeta^{\chi_{\mathrm{cyc}}(\sigma)}$ for any $\sigma \in \Gal(\mathcal{L}(\zeta_{p})/K)$ and any $p$-power root of unity $\zeta$.
Let $\omega$ and $\kappa$ denote the composition of $\chi_{\mathrm{cyc}}$ with the projections onto the first and second factors of the canonical decomposition $\Z_{p}^{\times} = \mu_{p-1} \times (1+p\Z_{p})$, respectively;
thus $\omega$ is the Teichm\"{u}ller character.
We note that $\kappa$ factors through $\Gamma_{K}$ 
(and thus also through $\mathcal{G}$) and by abuse of notation we also 
use $\kappa$ to denote the associated maps with these domains.
For a non-negative integer $r$ divisible by $p-1$ 
(or more generally divisible by the degree $[\mathcal{L}(\zeta_{p}) : \mathcal{L}]$), 
up to the canonical inclusion map of codomains, 
we have $\chi_{\mathrm{cyc}}^{r}=\kappa^{r}$. 

\subsection{A canonical complex}\label{subsec:canonical-complex}
Let $S$ be a finite set of places of $K$ containing $S_{p} \cup S_{\infty}$.
Let $\mathcal{O}_{\mathcal{L},S}$ denote the ring of integers $\mathcal{O}_{\mathcal{L}}$ in $\mathcal{L}$ localised at all primes above those in $S$.
There is a canonical complex 
\[
C_{S}^{\bullet}(\mathcal{L}/K) := R\Hom_{\Z_{p}}(R\Gamma_{\et}(\Spec(\mathcal{O}_{\mathcal{L},S}), \Q_{p} / \Z_{p}), \Q_{p} / \Z_{p}),
\]
where $\Q_{p} / \Z_{p}$ denotes the constant sheaf of the abelian group $\Q_{p} / \Z_{p}$ on the \'{e}tale site
of $\Spec(\mathcal{O}_{\mathcal{L},S})$.
Since $\Q_{p}/\Z_{p}$ is a direct limit of finite abelian groups of $p$-power order, we have an isomorphism with Galois cohomology 
\[
R\Gamma_{\et}(\Spec(\mathcal{O}_{\mathcal{L},S}), \Q_{p} / \Z_{p}) \simeq R\Gamma(\Gal(M_{S}(p)/\mathcal{L}),\Q_{p}/\Z_{p}),
\] 
where $M_{S}(p)$ is the maximal pro-$p$-extension of $\mathcal{L}$ unramified outside $S$.
(Apply \cite[Chapter II, Proposition 2.9]{MR2261462}, \cite[Chapter III, Lemma 1.16]{MR559531}
and \cite[Proposition 1.5.1]{MR2392026}.)
The cohomology modules of $C_{S}^{\bullet}(\mathcal{L}/K)$ are
\begin{equation}\label{eq:cohomology-of-FK-complex}
H^{i}(C_{S}^{\bullet}(\mathcal{L}/K)) \simeq \left\{
\begin{array}{lll}
X_{S} & \mbox{ if } & i=-1\\
\Z_{p} & \mbox{ if } & i=0\\
0 & \mbox{ if } & i \neq -1,0.\\
\end{array}
\right.
\end{equation}
Note that $C_{S}^{\bullet}(\mathcal{L}/K)$ and the complex used by Ritter and Weiss (as constructed in \cite{MR2114937}) become isomorphic in $\mathcal{D}(\Lambda(\mathcal{G}))$ by
\cite[Theorem 2.4]{MR3072281} (see also \cite{MR3068897} for more on this topic).
Hence it makes no real difference which of these two complexes we use.

Let $S_{\ram}(\mathcal{L}/K)$ be the (finite) set of places of $K$ that ramify in $\mathcal{L}/K$.
Note that since $\mathcal{L}$ contains the cyclotomic $\Z_{p}$-extension $K_{\infty}$ we must have $S_{p} \subseteq S_{\ram}(\mathcal{L}/K)$. The following result is well known to experts, but we include a proof for the convenience of the reader.

\begin{prop}\label{prop:complex-res-quot}
Suppose that $S$ contains $S_{\ram}(\mathcal{L}/K) \cup S_{\infty}$.

\begin{enumerate}
\item The complex $C_{S}^{\bullet}(\mathcal{L}/K)$ belongs to $\mathcal{D}^{\perf}\tor(\Lambda(\mathcal{G}))$.
\item The complex $C_{S}^{\bullet}(\mathcal{L}/K)$ defines a class 
$[C_{S}^{\bullet}(\mathcal{L}/K)]$ in $K_{0}(\Lambda(\mathcal{G}), \mathcal{Q}(\mathcal{G}))$.
\item Let $N$ be a finite normal subgroup of $\mathcal{G}$ and put $\mathcal{L}' := \mathcal{L}^{N}$.
Then 
\[
\quot^{\mathcal{G}}_{\mathcal{G}/N}([C_{S}^{\bullet}(\mathcal{L}/K)]) = [C_{S}^{\bullet}(\mathcal{L}'/K)].
\]
\item  Let $\mathcal{H}$ be an open subgroup of $\mathcal{G}$ and put $K' := \mathcal{L}^{\mathcal{H}}$.
Then 
\[
\res^{\mathcal{G}}_{\mathcal{H}}([C_{S}^{\bullet}(\mathcal{L}/K)]) = [C_{S'}^{\bullet}(\mathcal{L}/K')],
\]
where $S'$ is the set of places of $K'$ lying above those in $S$.
\end{enumerate}
\end{prop}

\begin{proof}
Let $G_{K,S}=\Gal(K_{S}/K)$ where $K_{S}$ is the maximal algebraic extension of
$K$ that is unramified outside the primes in $S$. 
Note that $\mathcal{G}$ is a quotient of  $G_{K,S}$ since $S$ contains $S_{\ram}(\mathcal{L}/K)$.
Let $\Lambda(\mathcal{G})^{\#}(1)$ be the free 
$\Lambda(\mathcal{G})$-module of rank one upon which $\sigma \in G_{K,S}$ acts on the right via
multiplication by the element $\chi_{\mathrm{cyc}}(\sigma)\overline \sigma^{-1}$,
where $\overline \sigma$ denotes the image of $\sigma$ in $\mathcal{G}$.
Observe that by the middle row of \cite[Theorem on p.\ 2638]{MR2954997}, the isomorphism
$(\Lambda(\mathcal{G})^{\#}(1))^{\vee}(1) \cong (\Lambda(\mathcal{G})^{\#})^{\vee}$
and a Shapiro lemma argument, we have
\begin{equation}\label{eq:AV-duality}
R\Gamma_{c}(G_{K,S}, \Lambda(\mathcal{G})^{\#}(1))[3]
\simeq
R\Hom_{\Z_{p}}(R\Gamma(G_{K,S},(\Lambda(\mathcal{G})^{\#})^{\vee}),\Q_{p}/\Z_{p})
\simeq
C_{S}^{\bullet}(\mathcal{L}/K)
\end{equation}
in $\mathcal{D}(\Lambda(\mathcal{G}))$, where the left-hand side denotes the 
compactly supported cohomology complex with coefficients in $\Lambda(\mathcal{G})^{\#}(1)$.
Thus $C_{S}^{\bullet}(\mathcal{L}/K)$ is perfect by \cite[Proposition 1.6.5]{MR2276851}.
Moreover, the cohomology groups of $C_{S}^{\bullet}(\mathcal{L}/K)$ are torsion as $R$-modules. 
Therefore (i) holds. Part (ii) follows from (i) and the discussion in \S \ref{subsec:orders-sep-al-nr}.
Furthermore, \eqref{eq:AV-duality} and loc.\ cit.\ together imply that there is an isomorphism
\[
\Lambda(\mathcal{G}/N) \otimes^{\mathbb {L}}_{\Lambda(\mathcal{G})}
C_{S}^{\bullet}(\mathcal{L}/K) \simeq C_{S}^{\bullet}(\mathcal{L}'/K)
\]
in $\mathcal{D}(\Lambda(\mathcal{G}/N))$, which gives part (iii).
Part (iv) is clear.
\end{proof}

\subsection{Power series and $p$-adic Artin $L$-functions}\label{subsec:power-series-p-adic-L-functions}
Recall that $S$ is a finite set of places of $K$ containing $S_{p} \cup S_{\infty}$.
Fix a character $\chi \in \Irr_{\Q_{p}^{c}}(\mathcal{G})$. 
Each topological generator $\gamma_{K}$ of  $\Gamma_{K}$ permits the definition of a 
power series $G_{\chi,S}(T) \in \Q_{p}^{c} \otimes_{\Q_{p}} Quot(\Z_{p} \llbracket T \rrbracket )$ 
by starting out from the Deligne-Ribet power series for one-dimensional characters of open subgroups 
of $\mathcal{G}$ (see \cite{MR579702}; also see \cite{ MR525346, MR524276}) 
and then extending to the general case by using Brauer induction (see \cite{MR692344}).
We put $u := \kappa(\gamma_{K})$.
Then we have an equality
\begin{equation} \label{eqn:p-adic-L-series}
L_{p,S}(1-s,\chi) = \frac{G_{\chi,S}(u^s-1)}{H_{\chi}(u^s-1)},
\end{equation}
where $L_{p,S}(s,\chi)$ denotes the `$S$-truncated $p$-adic Artin $L$-function' attached to $\chi$ constructed by Greenberg \cite{MR692344},
and where, for irreducible $\chi$, we have
\[
H_{\chi}(T) = \left\{\begin{array}{ll} \chi(\gamma_{K})(1+T)-1 & \mbox{ if }  H \subseteq \ker \chi\\
1 & \mbox{ otherwise.}  \end{array}\right.
\]
Note that $L_{p,S}(s, \chi) : \Z_{p} \rightarrow \C_{p}$ 
is the unique $p$-adic meromorphic
function with the property that for each strictly negative 
integer $r$ and each field isomorphism $\iota : \C \simeq \C_{p}$, we have
\begin{equation}\label{eqn:interpolation-property}
L_{p,S}(r, \chi) = \iota \left( L_{S}(r, \iota^{-1} \circ (\chi \otimes \omega^{r-1})) \right).
\end{equation}
By a result of Siegel \cite{MR0285488}, the right-hand side of  \eqref{eqn:interpolation-property}
does not depend on the choice $\iota$.
If $\chi$ is linear, then \eqref{eqn:interpolation-property} is also valid when $r=0$.

Now \cite[Proposition 11]{MR2114937} implies that
\begin{equation} \label{eqn:L_K,S}
L_{K,S} : \left[\chi \mapsto \frac{G_{\chi,S}(\gamma_{K}-1)}{H_{\chi}(\gamma_{K}-1)}\right]
\in \Hom^{\ast}_{G_{\Q_p}}(R_{p}( \mathcal{G}), \mathcal{Q}^{c}(\Gamma_{K})^{\times})
\end{equation}
and is independent of the topological generator $\gamma_{K}$.
Diagram \eqref{eqn:Det_triangle} implies that there is a unique element 
$\Phi_{S} = \Phi_{S}(\mathcal{L}/K) \in \zeta(\mathcal{Q}(\mathcal{G}))^{\times}$
such that
\begin{equation}\label{eqn:Phi_S-definition}
j_{\chi}(\Phi_{S}) = L_{K,S}(\chi)
\end{equation}
for every $\chi \in \Irr_{\Q_{p}^{c}}(\mathcal{G})$.
The following result is a special case of \cite[Proposition 12]{MR2114937}.

\begin{prop}\label{prop:phi-S-res-quot}
\emph{(i)} Let $N$ be a finite normal subgroup of $\mathcal{G}$ and put $\mathcal{L}' := \mathcal{L}^{N}$.
Then 
\[
\quot^{\mathcal{G}}_{\mathcal{G}/N}(\Phi_{S}(\mathcal{L}/K)) = \Phi_{S}(\mathcal{L}'/K).
\]
\emph{(ii)} 
Let $\mathcal{H}$ be an open subgroup of $\mathcal{G}$ and put $K' := \mathcal{L}^{\mathcal{H}}$.
Then 
\[
\res^{\mathcal{G}}_{\mathcal{H}}(\Phi_{S}(\mathcal{L}/K)) = \Phi_{S'}(\mathcal{L}/K'),
\]
where $S'$ is the set of places of $K'$ lying above those in $S$.
\end{prop}

\subsection{Statement and known cases of the EIMC}\label{subsec:statement-EIMC}
Recall that $p$ is an odd prime and $\mathcal{L} / K$ is an admissible one-dimensional $p$-adic Lie extension.
Let $S$ be a finite set of places of $K$ containing $S_{\ram}(\mathcal{L}/K) \cup S_{\infty}$.
Let $\nr: K_{1}(\mathcal{Q}(\mathcal{G})) \rightarrow \zeta(\mathcal{Q}(\mathcal{G}))^{\times}$
be the map induced by the reduced norm.

\begin{conj}[EIMC]\label{conj:EIMC}
There exists $\zeta_{S} \in K_{1}(\mathcal{Q}(\mathcal{G}))$ such that $\partial(\zeta_{S}) = -[C_{S}^{\bullet}(\mathcal{L}/K)]$
and $\nr(\zeta_{S}) = \Phi_{S}(\mathcal{L}/K)$.
\end{conj}

It can be shown that the truth of Conjecture \ref{conj:EIMC} is independent of the choice of $S$, provided that $S$ is finite and contains $S_{\ram}(\mathcal{L}/K) \cup S_{\infty}$.
Crucially, this version of the EIMC does not require the $\mu=0$ hypothesis for its formulation.
The following theorem has been shown independently by Ritter and Weiss \cite{MR2813337} and by Kakde \cite{MR3091976}.

\begin{theorem}\label{thm:EIMC-with-mu}
If $\mathcal{L}/K$ satisfies the $\mu=0$ hypothesis then the EIMC holds for $\mathcal{L}/K$.
\end{theorem}

By considering the cases in which the $\mu=0$ hypothesis is known, we obtain the following corollary
(see \cite[Corollary 4.6]{MR3749195} for further details).

\begin{corollary}
Let $\mathcal{P}$ be a Sylow $p$-subgroup of $\mathcal{G}$.
If $\mathcal{L}^{\mathcal{P}}/\Q$ is abelian then the EIMC holds for $\mathcal{L}/K$.
\end{corollary}

We shall also consider the EIMC with its uniqueness statement.

\begin{conj}[EIMC with uniqueness]\label{conj:EIMC-unique}
There exists a unique $\zeta_{S} \in K_{1}(\mathcal{Q}(\mathcal{G}))$
such that $\nr(\zeta_{S}) = \Phi_{S}(\mathcal{L}/K)$.
Moreover, $\partial(\zeta_{S}) = -[C_{S}^{\bullet}(\mathcal{L}/K)]$.
\end{conj}

\begin{remark}\label{rmk:SK1}
If $SK_{1}(\mathcal{Q}(\mathcal{G}))
:= \ker(\nr: K_{1}(\mathcal{Q}(\mathcal{G})) \rightarrow \zeta(\mathcal{Q}(\mathcal{G}))^{\times})$
vanishes then it is clear that the uniqueness statement of the EIMC follows from
its existence statement.
In particular, this is the case if $\mathcal{G}$ is abelian because then $\nr : K_{1}(\mathcal{Q}(\mathcal{G})) \rightarrow \mathcal{Q}(\mathcal{G})^{\times}$ is equal to the usual determinant map $\det$ and is an isomorphism by \cite[Proposition 45.12]{MR892316}.
\end{remark}

In \cite{MR3749195},
the present authors proved the EIMC unconditionally for an infinite class of one-dimensional admissible extensions for which the $\mu=0$ hypothesis is not known to be true.
We now recall the special case of these results given by {\cite[Theorem 4.12]{MR3749195}},
whose proof relies crucially on a result of Ritter and Weiss \cite[Theorem 16]{MR2114937}.

\begin{theorem}\label{thm:EIMC-p-does-not-divide-order-of-H}
If $p \nmid |H|$ then the EIMC with uniqueness holds for $\mathcal{L}/K$. 
\end{theorem}

\subsection{The EIMC for $\mathcal{L}/K$ implies the EIMC for all admissible subextensions}

The following result is well known, but we include a proof for the convenience of the reader.

\begin{lemma}\label{lem:EIMC-implies-EIMC-for-subextensions}
Let $p$ be an odd prime and let $\mathcal{L}/K$ be an admissible
one-dimensional $p$-adic Lie extension of a totally real number field $K$.
If the EIMC holds for $\mathcal{L}/K$ then the EIMC holds for all admissible sub-extensions of $\mathcal{L}/K$.
\end{lemma}

\begin{proof}
It suffices to show the result for admissible sub-extensions of the form
$\mathcal{L}'/K$ and of the form $\mathcal{L}/K'$.
We shall only prove the former case as the proof of the latter case is entirely analogous.
Let $S$ be a finite set of places of $K$ containing $S_{\infty} \cup S_{\ram}(\mathcal{L}/K)$.
Since the EIMC holds for $\mathcal{L}/K$,
there exists $\zeta_{S} \in K_{1}(\mathcal{Q}(\mathcal{G}))$ such that 
 $\partial(\zeta_{S}) = -[C_{S}^{\bullet}(\mathcal{L}/K)]$
and $\nr(\zeta_{S}) = \Phi_{S}(\mathcal{L}/K)$.
Let $N=\Gal(\mathcal{L}/\mathcal{L}')$.
Specialising and combining the appropriate commutative diagrams from \S \ref{subsec:Iwasawa-algebras}
and the end of \S \ref{subsec:dets-and-nr}, we obtain a commutative diagram
\[
\xymatrix{
\zeta(\mathcal{Q}(\mathcal{G}))^{\times} \ar[d]_{\quot^{\mathcal{G}}_{\mathcal{G}/N}}  & \ar[l]_{\nr} K_{1}(\mathcal{Q}(\mathcal{G})) \ar[r]^{\partial \qquad} \ar[d]_{\quot^{\mathcal{G}}_{\mathcal{G}/N}} &
K_{0}(\Lambda(\mathcal{G}),\mathcal{Q}(\mathcal{G})) \ar[d]_{\quot^{\mathcal{G}}_{\mathcal{G}/N}} \\
\zeta(\mathcal{Q}(\mathcal{G}/N))^{\times}  & \ar[l]_{\nr}  K_{1}(\mathcal{Q}(\mathcal{G}/N)) \ar[r]^{\partial \qquad} &
K_{0}(\Lambda(\mathcal{G}/N),\mathcal{Q}(\mathcal{G}/N)).
}
\]
Moreover, Propositions \ref{prop:complex-res-quot} and \ref{prop:phi-S-res-quot} give
\[
\quot^{\mathcal{G}}_{\mathcal{G}/N}([C_{S}^{\bullet}(\mathcal{L}/K)]) = [C_{S}^{\bullet}(\mathcal{L}'/K)]
\quad \textrm{ and } \quad
\quot^{\mathcal{G}}_{\mathcal{G}/N}(\Phi_{S}(\mathcal{L}/K)) = \Phi_{S}(\mathcal{L}'/K).
\]
Therefore $\quot^{\mathcal{G}}_{\mathcal{G}/N}(\zeta_{S})$ has the desired properties and so the EIMC holds for 
$\mathcal{L}'/K$.
\end{proof}

\section{Fitting ideals, complexes and commutative Iwasawa algebras}\label{sec:Fitt-complexes-commutative-Iwasawa-algs}

\subsection{A lemma on integral extensions and principal ideals}

The following lemma is well known; see \cite[p.\ 526]{MR1750935}, for example.
It will be used in the proof of Proposition~\ref{prop:abelian-EIMC-equiv-local-at-(p)}
to compute the Fitting ideal of a certain Iwasawa module.

\begin{lemma}\label{lemma:ideals-equal-over-integral-ext}
Let $B$ be an integral extension of a commutative ring $A$. 
If $x,y \in A$ such that $y$ is a nonzerodivisor in $B$, 
$Ax \subseteq Ay$ and $Bx=By$, then in fact $Ax=Ay$. 
\end{lemma}

\begin{proof}
Since $Ax \subseteq Ay$ there exists $z \in A$ such that $x=yz$. 
Then $Byz=Bx=By$ so there exists $w \in B$ such that $y=yzw$.
As $y$ is a nonzerodivisor in $B$ we have $1=zw$. 
Thus $z \in A \cap B^{\times}$.
But $A \cap B^{\times}=A^{\times}$ 
by \cite[Chapter 5, Exercise 5 (i)]{MR0242802} (see also \cite[Lemma 9.7]{MR703486}).
Therefore $z \in A^{\times}$ and so $Ax=Ayz=Ay$, as desired. 
\end{proof}

\subsection{Fitting ideals of complexes}\label{subsec:Fitting-ideals-of-complexes}
Let $R$ be a local noetherian integral domain and let $\mathfrak{A}$ be an
$R$-order in a finite-dimensional separable commutative $Quot(R)$-algebra $A$.
In other words, we consider the situation of \S \ref{subsec:orders-sep-al-nr},
but assume in addition that $R$ is local and that $A$ is commutative.
Since $\mathfrak{A}$ and $A$ are both noetherian commutative semilocal rings,
the reduced norm on $A$ is equal to the usual determinant map, and 
by \cite[Proposition~45.12]{MR892316} this 
induces isomorphisms $K_{1}(\mathfrak{A}) \cong \mathfrak{A}^{\times}$ and
$\det : K_{1}(A) \cong A^{\times}$.

Using this fact, 
specialising \eqref{eqn:long-exact-seq} to the case at hand gives an exact sequence
\begin{equation}\label{eqn:K-theory-SES-comm-Iwasawa-algebra}
	0 \longrightarrow
	K_{1}(\mathfrak{A}) \longrightarrow K_{1}(A) \stackrel{\partial}{\longrightarrow}
	K_{0}(\mathfrak{A},A). 
\end{equation}
Now let $C^{\bullet} \in \mathcal{D}^{\perf}\tor(\mathfrak{A})$ and recall from \S \ref{subsec:K-theory}
that $C^{\bullet}$ defines an element $[C^{\bullet}]$ in $K_{0}(\mathfrak{A},A)$.
Assume that there is an $x \in K_{1}(A)$ such that $\partial(x) = [C^{\bullet}]$ and put
\begin{equation*}\label{eqn:fitt-of-complex}
	\Fitt_{\mathfrak{A}}(C^{\bullet}) := \det(x) \mathfrak{A} 
	\quad
	\text{ and }
	\quad
	\Fitt_{\mathfrak{A}}^{-1}(C^{\bullet}) := \det(x)^{-1} \mathfrak{A}. 
\end{equation*}
Note that these are well defined by the exactness of \eqref{eqn:K-theory-SES-comm-Iwasawa-algebra}.
If $C^{\bullet}_{i} \in \mathcal{D}^{\perf}\tor(\mathfrak{A})$ for $i=1,2,3$ such that 
$[C_{2}^{\bullet}] = [C_{1}^{\bullet}] + [C_{3}^{\bullet}]$ in $K_{0}(\mathfrak{A},A)$
(this is the case in the situation of  \eqref{eq:SES-of-complexes}, for example)
then it is straightforward to show that
\begin{equation}\label{eqn:fitt-of-sum-of-complexes-in-rel-K-zero}
	\Fitt_{\mathfrak{A}}(C_{2}^{\bullet}) = \Fitt_{\mathfrak{A}}(C_{1}^{\bullet}) \cdot \Fitt_{\mathfrak{A}}(C_{3}^{\bullet})
\end{equation}
whenever the Fitting ideals of the complexes are defined.

\begin{remark} \label{rem:Fitt-complex-vs-module}
Let $k \in \Z$ and suppose that $C^{\bullet}$ is an object of $\mathcal{D}^{\perf}\tor(\mathfrak{A})$ such that $H^{i}(C^{\bullet}) = 0$ 
for all $i \in \Z - \{ k \}$.
Then $H^{k}(C^{\bullet})$ is of finite projective dimension over $\mathfrak{A}$ and $[C^{\bullet}]=(-1)^{k}[H^{k}(C^{\bullet})]$ in $K_{0}(\mathfrak{A}, A)$
by Proposition \ref{prop:perfect-cohomology-acyclic-outside-one-degree}. 
If we assume in addition that $H^{k}(C^{\bullet})$ has projective dimension at most one, then it follows easily from the definitions that we have
$\Fitt_{\mathfrak{A}}(C^{\bullet}) = \Fitt_{\mathfrak{A}}^{(-1)^k}(H^{k}(C^{\bullet}))$ whenever the Fitting ideal of the complex is defined.
\end{remark}

\subsection{Fitting ideals of Iwasawa modules}\label{subsec:Fitting-ideals-Iwasawa-modules}
Let $p$ be a prime and let $\mathcal{G}$ be an abelian one-dimensional compact $p$-adic Lie group.
Then $\mathcal{G} = H \times \Gamma$ where $H$ is a finite abelian group and $\Gamma \simeq \Z_{p}$. In particular, $\mathcal{G}$ is admissible.
Let $R=\Z_{p}\llbracket\Gamma\rrbracket$.
Then $\Lambda(\mathcal{G})=R[H]$ is a commutative $R$-order in the separable $Quot(R)$-algebra $\mathcal{Q}(\mathcal{G})$.
Let $\mathcal{M}(\mathcal{G})$ denote the unique maximal $R$-order in $\mathcal{Q}(\mathcal{G})$
and note that $\mathcal{M}(\mathcal{G})$ is the integral closure of $\Lambda(\mathcal{G})$ in $\mathcal{Q}(\mathcal{G})$ (see \cite[Theorem 8.6]{MR1972204}).

Now let $e$ be any idempotent element of $\Lambda(\mathcal{G})$ and define
\begin{equation}\label{eqn:Lambda-e-def}
\Lambda := e\Lambda(\mathcal{G}), 
\qquad \mathcal{M} := e\mathcal{M}(\mathcal{G}),
\quad \textrm{ and }  \quad \mathcal{Q} := e\mathcal{Q}(\mathcal{G}). 
\end{equation}
Then $\Lambda$ and $\mathcal{M}$ are both $R$-orders in $\mathcal{Q}$ and $\mathcal{M}$ is maximal.

It easily follows from \eqref{eqn:Iwasawa-K-sequence} that we have an exact sequence
\begin{equation}\label{eqn:Iwasawa-K-sequence-with-e}
K_{1}(\Lambda) \longrightarrow K_{1}(\mathcal{Q}) \stackrel{\partial}{\longrightarrow}
K_{0}(\Lambda,\mathcal{Q}) \longrightarrow 0.
\end{equation}
Similarly, \cite[Corollary 2.14]{MR4098596} implies that for every height one prime ideal $\mathfrak{p}$
of $R$, we have an exact sequence
\begin{equation}\label{eqn:local-Iwasawa-K-sequence-with-e}
K_{1}(\Lambda_{\mathfrak{p}}) \longrightarrow K_{1}(\mathcal{Q}) \stackrel{\partial_{\mathfrak{p}}}{\longrightarrow}
K_{0}(\Lambda_{\mathfrak{p}},\mathcal{Q}) \longrightarrow 0,
\end{equation}
where $\Lambda_{\mathfrak{p}} = R_{\mathfrak{p}} \otimes_{R} \Lambda$ 
and $R_{\mathfrak{p}}$ is the localisation of $R$ at $\mathfrak{p}$.

Apart from its final claim, the following lemma is well known.

\begin{lemma}\label{lemma:fitting-ideal-is-principal}
Let $M$ be a finitely generated $\Lambda$-module that is of projective dimension at most one and that is also $R$-torsion.
Then $M$ has a quadratic presentation of the form 
\begin{equation}\label{eq:quad-pres}
0 \longrightarrow \Lambda^{n} \stackrel{h}{\longrightarrow} \Lambda^{n} \longrightarrow M \longrightarrow 0 
\end{equation}
for some $n \geq 1$. Moreover, $\Fitt_{\Lambda}(M)$ is a principal ideal generated by a nonzerodivisor. The same statement holds if we replace the pair
$(\Lambda,R)$ by a pair $(\Lambda_{\mathfrak{p}}, R_{\mathfrak{p}})$
for a height one prime ideal $\mathfrak{p}$ of $R$.
\end{lemma}

\begin{proof}
If $M=0$ then the claims are trivial, so we henceforth suppose that $M \neq 0$.
Let $0 \rightarrow P \rightarrow \Lambda^{n} \rightarrow M \rightarrow 0$ be a projective resolution of $M$.
Since $M$ is $R$-torsion, the class $[P] - [\Lambda^n] \in K_0(\Lambda)$
is mapped to zero in $K_0(\mathcal{Q})$. 
It follows from \eqref{eqn:Iwasawa-K-sequence-with-e} that the map 
$K_0(\Lambda) \rightarrow K_0(\mathcal{Q})$
is injective. Hence $P$ and $\Lambda^{n}$
are stably isomorphic.
By enlarging $n$ if necessary, we then can and do
assume that $P=\Lambda^{n}$ and so we have a presentation of the form \eqref{eq:quad-pres}.
Thus $\Fitt_{\Lambda}(M)$ is principal by definition of Fitting ideal and
any generator is a nonzerodivisor since $h$ is injective. 
The final claim is shown analogously, where \eqref{eqn:Iwasawa-K-sequence-with-e}
is replaced by \eqref{eqn:local-Iwasawa-K-sequence-with-e}.
\end{proof}

We recall the following result of Greither and Kurihara \cite[Theorem 2.1]{MR2443336}.
We caution that the notation here differs from that of loc.\ cit.\ (the roles of $R$ and $\Lambda$ are reversed).
Let $\gamma$ be a topological generator of $\Gamma$.
For $n \geq 1$ define $\omega_{n} = \gamma^{p^{n}}-1 \in R$
and $\Lambda_{n} = \Lambda/\omega_{n}\Lambda$. 
Then $(\Lambda_{n})_{n}$ is a projective system with limit $\Lambda$ and we make the canonical identification $\Lambda \cong \varprojlim_{n} \Lambda_{n}$
We shall consider projective systems $(A_{n})_{n}$ of modules $A_{n}$ over $\Lambda_{n}$ such that the transition maps
$A_{m} \rightarrow A_{n}$ ($m \geq n$) are $\Lambda_{m}$-linear in the obvious sense.
The limit $M := \varprojlim_{n} A_{n}$ will then be a $\Lambda$-module.  

\begin{theorem}[Greither and Kurihara]\label{thm:GK-inverse-limit-Fitting-ideals}
Suppose that the limit $M$ is a finitely generated $\Lambda$-module that is $R$-torsion
and that there exists $n_{0} \geq 1$ such that the transition maps $A_{m} \rightarrow A_{n}$ are surjective
for all $m \geq n \geq n_{0}$. Then $\Fitt_{\Lambda}(M) = \varprojlim_{n}(\Fitt_{\Lambda_{n}}(A_{n}))$.
\end{theorem}
 
\begin{proof}
In \cite[Theorem 2.1]{MR2443336}, this is stated in the case $\Lambda=\Lambda(\mathcal{G})$.
It is clear that this implies the desired result for any choice of $\Lambda$ as defined in \eqref{eqn:Lambda-e-def}.
\end{proof}

\subsection{Algebraic $K$-theory for commutative Iwasawa algebras}
The main goal of this subsection is to prove a purely algebraic result that implies that the precise
choice of complex used in the abelian EIMC does not matter, provided that it is perfect and has the 
prescribed cohomology.
The following results are generalised in Appendix \ref{app:independence-of-choice-of-complex}.

\begin{prop}\label{prop:injectivity-rel-K-localization-in-abelian-case}
The canonical map
\[
K_{0}(\Lambda,\mathcal{Q})
\longrightarrow
\bigoplus_{\mathfrak{p}}
K_{0}(\Lambda_{\mathfrak{p}},\mathcal{Q})
\] 
is injective,
where the direct sum is taken over all height one prime ideals of $R$.
\end{prop}

\begin{proof}
Recall from \S \ref{subsec:Fitting-ideals-of-complexes} that the determinant induces isomorphisms $K_{1}(\Lambda) \cong \Lambda^{\times}$
and $K_{1}(\mathcal{Q}) \cong \mathcal{Q}^{\times}$.
Since $\partial : K_{1}(\mathcal{Q}) \rightarrow K_{0}(\Lambda, \mathcal{Q})$ 
is surjective by \eqref{eqn:Iwasawa-K-sequence-with-e}, the determinant also induces
a canonical isomorphism $K_0(\Lambda, \mathcal{Q}) \cong \mathcal{Q}^{\times} / \Lambda^{\times}$. 
Similar reasoning using \eqref{eqn:local-Iwasawa-K-sequence-with-e} shows there is an isomorphism
$K_{0}(\Lambda_{\mathfrak{p}},\mathcal{Q}) \cong \mathcal{Q}^{\times} / \Lambda_{\mathfrak{p}}^{\times}$ for each height one prime ideal $\mathfrak{p}$ of $R$.
Therefore the claim is equivalent to the injectivity of canonical map
\[
	\mathcal{Q}^{\times} / \Lambda^{\times} \longrightarrow
	\prod_{\mathfrak{p}} \mathcal{Q}^{\times} / \Lambda_{\mathfrak{p}}^{\times}.
\]
Since $\Lambda$ is free as an $R$-module, it is reflexive 
and so by \cite[Lemma 5.1.2(iii)]{MR2392026} we have
\[
	\bigcap_{\mathfrak{p}} \Lambda_{\mathfrak{p}} = \Lambda.
\]
Now let $x \in \mathcal{Q}^{\times}$ and assume that
$x \in \bigcap_{\mathfrak{p}}\Lambda_{\mathfrak{p}}^{\times}$.
Then $x^{-1} \in \bigcap_{\mathfrak{p}}\Lambda_{\mathfrak{p}}^{\times} \subseteq \Lambda$
and therefore $x \in \Lambda^{\times}$.
\end{proof}

\begin{corollary}\label{cor:complex-in-rel-K0-in-abelian-case}
Let $k \in \Z$.
Let $C^{\bullet}, D^{\bullet} \in \mathcal{D}\tor^{\perf}(\Lambda)$  such that
\begin{enumerate}
\item 
$H^{i}(C^{\bullet}) \simeq H^{i}(D^{\bullet})$ as $\Lambda$-modules for all $i \in \Z$; and
\item 
$H^{i}(C^{\bullet})$ and $H^{i}(D^{\bullet})$ are finitely generated over $\Z_{p}$
for all $i \in \Z - \{k\}$.
\end{enumerate}
Then $[C^{\bullet}]=[D^{\bullet}]$ in $K_{0}(\Lambda,\mathcal{Q})$.
\end{corollary}

\begin{proof}
We claim that $[C^{\bullet}]$ and $[D^{\bullet}]$ have the same image
in $K_0(\Lambda_{\mathfrak{p}}, \mathcal{Q})$ for each height one prime ideal
$\mathfrak{p}$ of $R$. The result then follows from
Proposition~\ref{prop:injectivity-rel-K-localization-in-abelian-case}.

We have $H^{i}(\Lambda_{\mathfrak{p}} \otimes_{\Lambda}^{\mathbb{L}} C^{\bullet})
\simeq \Lambda_{\mathfrak{p}} \otimes_{\Lambda} H^{i}(C^{\bullet})$ for all $i \in \Z$
since localisation at $\mathfrak{p}$ is exact.
The analogous statement also holds for $D^{\bullet}$. 
We first consider the case $\mathfrak{p} \not= (p)$. 
Then $\Lambda_{\mathfrak{p}}$ is a maximal order over the discrete valuation ring
$R_{\mathfrak{p}}$, and thus is hereditary by \cite[Theorem 18.1]{MR1972204}. 
Hence every finitely generated $\Lambda_{\mathfrak{p}}$-module
is of finite projective dimension and thus the claim follows from
(i) and Proposition \ref{prop:perfect-cohomology-fpd} in this case.
Since (ii) implies that 
$\Lambda_{(p)} \otimes_{\Lambda} H^{i}(C^{\bullet}) = \Lambda_{(p)} \otimes_{\Lambda} H^{i}(D^{\bullet})=0$ for $i \in \Z - \{ k \}$, 
the claim for $\mathfrak{p} = (p)$ follows from Proposition
\ref{prop:perfect-cohomology-acyclic-outside-one-degree} and (i).
\end{proof}

\section{The abelian EIMC and localisation at $(p)$}\label{sec:abelian-EIMC-localisation}

Let $p$ be an odd prime and let $K$ be a totally real number field. 
Let $\mathcal{L}/K$ be an abelian admissible one-dimensional $p$-adic Lie extension of $K$.
Let $\mathcal{G}=\Gal(\mathcal{L}/K)$ and write $\mathcal{G} = H \times \Gamma$
where $H$ is a finite abelian group and $\Gamma \simeq \Z_{p}$.
Let $R = \Z_{p} \llbracket \Gamma \rrbracket$.
Let $S$ be a finite set of places of $K$ containing $S_{\ram}(\mathcal{L}/K) \cup S_{\infty}$.

Since $\mathcal{G}$ is abelian, the reduced norm map $\nr : K_{1}(\mathcal{Q}(\mathcal{G})) \rightarrow \mathcal{Q}(\mathcal{G})^{\times}$ is equal to the usual determinant map $\det$
and is an isomorphism by \cite[Proposition 45.12]{MR892316}. 
Let $\zeta_{S} = \zeta_{S}(\mathcal{L}/K) := \det^{-1}(\Phi_{S}(\mathcal{L}/K))$
and let 
\[
\omega_{S} = \omega_{S}(\mathcal{L}/K):= 
\partial(\zeta_{S}(\mathcal{L}/K)) + [C_{S}^{\bullet}(\mathcal{L}/K)]
\in K_0(\Lambda(\mathcal{G}), \mathcal{Q}(\mathcal{G})).
\]
It follows easily from its statement (Conjecture \ref{conj:EIMC}) 
that the EIMC for $\mathcal{L}/K$ is equivalent to the assertion that 
$\partial(\zeta_{S}(\mathcal{L}/K))=-[C_{S}^{\bullet}(\mathcal{L}/K)]$.
Hence we obtain the following equivalent formulation of the abelian EIMC.

\begin{lemma}\label{lem:equiv-abelian-EIMC}
The EIMC for $\mathcal{L}/K$ holds if and only if $\omega_{S}$ vanishes.
\end{lemma}

Let $\mathcal{M}(\mathcal{G})$ denote the unique maximal $R$-order in $\mathcal{Q}(\mathcal{G})$
and note that $\mathcal{M}(\mathcal{G})$ is the integral closure of $\Lambda(\mathcal{G})$ in $\mathcal{Q}(\mathcal{G})$ (see \cite[Theorem 8.6]{MR1972204}).
The following result states that the EIMC for $\mathcal{L}/K$ holds `over the maximal order'.
Variants of this result for
arbitrary admissible one-dimensional $p$-adic Lie extensions are due to
Ritter and Weiss \cite[Theorem 16]{MR2114937}, the present authors \cite[Theorem 4.9]{MR3749195}, and 
(in terms of Selmer groups) Greenberg \cite[Proposition 9]{MR3586815}; these are all 
ultimately reformulations of 
the classical (non-equivariant) 
Iwasawa main conjecture and a result on comparison of $\mu$\nobreakdash-invariants,
both proven by Wiles \cite[Theorems 1.3 and 1.4]{MR1053488}.

\begin{prop}\label{prop:version-of classical MC}
The element $\omega_{S}$ maps to zero under the canonical map
\begin{equation} \label{eqn:extension-to-max-order}
K_{0}(\Lambda(\mathcal{G}), \mathcal{Q}(\mathcal{G})) \longrightarrow
K_{0}(\mathcal{M}(\mathcal{G}), \mathcal{Q}(\mathcal{G}))
\end{equation}
induced by extension of scalars.
\end{prop}

\begin{proof}
Let $x_{S} \in K_{1}(\mathcal{Q}(\mathcal{G}))$ such that $\partial(x_{S}) = -[C_{S}^{\bullet}(\mathcal{L}/K)]$.
By \cite[Theorem 16]{MR2114937} (whose proof simplifies in the abelian case) we have that
\[
	\Det(x_{S})L_{K,S}^{-1} \in \Hom^{\ast}(R_{p}(\mathcal{G}), \Lambda^{c}(\Gamma_{K})^{\times}),
\]
where $\Lambda^{c}(\Gamma_{K}) := \Z_{p}^{c} \otimes_{\Z_{p}} \Lambda(\Gamma_{K})$ and $\Z_{p}^{c}$ denotes the integral closure
of $\Z_{p}$ in $\Q_{p}^{c}$.
Moreover, $\Hom^{\ast}(R_{p}(\mathcal{G}), \Lambda^{c}(\Gamma_{K})^{\times})$ identifies with
$\mathfrak{M}(\mathcal{G})^{\times}$
under the isomorphism in diagram \eqref{eqn:Det_triangle} as explained in \cite[Remark H]{MR2114937}
and $L_{K,S}$ corresponds to $\Phi_{S}(\mathcal{L}/K)$ by definition - see \eqref{eqn:Phi_S-definition}.
Thus $y_{S}:=\det(x_{S})\Phi_{S}^{-1} \in \mathfrak{M}(\mathcal{G})^{\times}$.
The desired result now follows once one observes that 
$\partial(\det^{-1}(y_{S})) = -\omega_{S}$
and that $\det^{-1}(\mathfrak{M}(\mathcal{G})^{\times})$ is equal to the kernel of the canonical map $K_{1}(\mathcal{Q}(\mathcal{G})) \rightarrow K_{0}(\mathcal{M}(\mathcal{G}), \mathcal{Q}(\mathcal{G}))$, which is the composition of $\partial$ and the map \eqref{eqn:extension-to-max-order}.
\end{proof}

For a height one prime ideal $\mathfrak{p}$ of $R$, let $R_{\mathfrak{p}}$ be denote localisation of $R$ at 
$\mathfrak{p}$, let $\Lambda_{\mathfrak{p}}(\mathcal{G}) = R_{\mathfrak{p}} \otimes_{R} \Lambda(\mathcal{G})$,
let $\mathcal{M}_{\mathfrak{p}}(\mathcal{G}) = R_{\mathfrak{p}} \otimes_{R} \mathcal{M}(\mathcal{G})$, and 
let
\[
\mathrm{loc}_{\mathfrak{p}}: K_{0}(\Lambda(\mathcal{G}), \mathcal{Q}(\mathcal{G}))
\longrightarrow K_{0}(\Lambda_{\mathfrak{p}}(\mathcal{G}), \mathcal{Q}(\mathcal{G}))
\]
be the canonical map induced by $\Lambda_{\mathfrak{p}}(\mathcal{G}) \otimes_{\Lambda(\mathcal{G})} -$.

\begin{corollary}\label{cor:vanishing-omega-ht-1-not-p}
Let $\mathfrak{p} \not= (p)$ be a height one prime ideal of $R$.
Then $\mathrm{loc}_{\mathfrak{p}}(\omega_{S})=0$.
\end{corollary}

\begin{proof}
For each height one prime ideal $\mathfrak{p} \not= (p)$ we have
$\Lambda_{\mathfrak{p}}(\mathcal{G}) = \mathcal{M}_{\mathfrak{p}}(\mathcal{G})$.
Hence the map $\mathrm{loc}_{\mathfrak{p}}$ factors through the map
\eqref{eqn:extension-to-max-order} and so the result follows from
Proposition \ref{prop:version-of classical MC}.
\end{proof}

Let $\partial_{(p)} : K_{1}(\mathcal{Q}(\mathcal{G})) \longrightarrow K_{0}(\Lambda_{(p)}(\mathcal{G}), \mathcal{Q}(\mathcal{G}))$ be the canonical map associated to the height one prime ideal $(p)$ of $R$.
Note that $\partial_{(p)} = \mathrm{loc}_{(p)} \circ \partial$.
We henceforth abbreviate 
$\Lambda_{(p)}(\mathcal{G}) \otimes_{\Lambda(\mathcal{G})} X_{S}$ to $(X_{S})_{(p)}$.

\begin{prop}\label{prop:abelian-EIMC-equiv-local-at-(p)}
The following are equivalent.
\begin{enumerate}
\item The EIMC holds for $\mathcal{L}/K$.
\item $\omega_{S}=0$.
\item $\mathrm{loc}_{(p)}(\omega_{S})=0$.
\item $\partial_{(p)}(\zeta_{S})=[(X_S)_{(p)}]$.
\item $\Phi_{S}$ generates $\Fitt_{\Lambda_{(p)}(\mathcal{G})}((X_S)_{(p)})$.
\item $\Phi_{S} \in \Fitt_{\Lambda_{(p)}(\mathcal{G})}((X_S)_{(p)})$.
\end{enumerate}
\end{prop}

\begin{proof}
The equivalence of (i) and (ii) is just Lemma \ref{lem:equiv-abelian-EIMC}.
The equivalence of (ii) and (iii) follows from
Proposition \ref{prop:injectivity-rel-K-localization-in-abelian-case} and
Corollary \ref{cor:vanishing-omega-ht-1-not-p}.
We have
\[
H^{i}(\Lambda_{(p)}(\mathcal{G}) \otimes_{\Lambda(\mathcal{G})}^{\mathbb{L}} C_{S}^{\bullet}(\mathcal{L}/K))
\simeq 
\Lambda_{(p)}(\mathcal{G}) \otimes_{\Lambda(\mathcal{G})} H^{i}( C_{S}^{\bullet}(\mathcal{L}/K))
\simeq \left\{
\begin{array}{lll}
(X_{S})_{(p)} & \mbox{ if } & i=-1\\
0 & \mbox{ if } & i \neq -1,
\end{array}
\right.
\]
where the first isomorphism follows from the fact that localisation at $(p)$ is exact, and the second
isomorphism follows from \eqref{eq:cohomology-of-FK-complex} and the fact that
$\Lambda_{(p)}(\mathcal{G}) \otimes_{\Lambda(\mathcal{G})} \Z_{p} = 0$.
Moreover, Proposition 
\ref{prop:perfect-cohomology-acyclic-outside-one-degree} then implies that
$(X_S)_{(p)}$ is of finite projective
dimension over $\Lambda_{(p)}(\mathcal{G})$ and that we have an equality 
\[
[\Lambda_{(p)}(\mathcal{G}) \otimes^{\mathbb{L}}_{\Lambda(\mathcal{G})}
C_{S}^{\bullet}(\mathcal{L}/K)] = -[(X_S)_{(p)}] \in
K_{0}(\Lambda_{(p)}(\mathcal{G}), \mathcal{Q}(\mathcal{G})).
\]
It follows easily that (iii) and (iv) are equivalent.
In fact, the projective dimension of $(X_S)_{(p)}$ is at most one by \cite[Lemma 5.2]{MR4098596}.
Thus $\Fitt_{\Lambda_{(p)}(\mathcal{G})}((X_S)_{(p)})$ is a principal ideal generated by a nonzerodivisor
by Lemma~\ref{lemma:fitting-ideal-is-principal}.
By Remark \ref{rem:Fitt-complex-vs-module} and the fact that ${\Phi_{S} =\det(\zeta_{S})}$,
we see that (iv) is equivalent to (v).
Analogous reasoning and Proposition~\ref{prop:version-of classical MC} 
show that $\Phi_{S}$ generates
$\Fitt_{\mathcal{M}_{(p)}(\mathcal{G})}
(\mathcal{M}_{(p)}(\mathcal{G}) \otimes_{\Lambda_{(p)}(\mathcal{G})}(X_S)_{(p)})$.
Therefore the fact that Fitting ideals commute with base change 
(see for example \cite[Corollary 20.5]{MR1322960}) and
an application of Lemma \ref{lemma:ideals-equal-over-integral-ext}
now show that (v) and (vi) are equivalent.
\end{proof}

\begin{remark}
Suppose $\mathcal{L}/K$ satisfies the $\mu=0$ hypothesis.
Then $X_{S}$ is finitely generated as a $\Z_{p}$-module and so $(X_{S})_{(p)}$ vanishes.
Thus by Proposition \ref{prop:abelian-EIMC-equiv-local-at-(p)} the EIMC for $\mathcal{L}/K$
is equivalent to the assertion that
\[
\zeta_{S} \in \ker(\partial_{(p)}) = K_{1}(\Lambda_{(p)}(\mathcal{G})).
\]
Since $\Phi_{S} =\det(\zeta_{S})$ and $\det : K_{1}(\mathcal{Q}(\mathcal{G})) \rightarrow \mathcal{Q}(\mathcal{G})^{\times}$
is an isomorphism, this in turn is equivalent to the assertion that 
$\Phi_{S} \in \Lambda_{(p)}(\mathcal{G})^{\times}$.
A proof of this last assertion (again under the $\mu=0$ hypothesis) can be found in
\cite[\S 6]{MR1935024} or \cite[Lemma 1.14]{MR2819672} 
(we caution that the meaning of $S$ here differs from that in \cite{MR2819672}).
\end{remark}

\section{Inverse limits of Stickelberger elements}\label{sec:inverse-limits-stickelberger}

Let $p$ be an odd prime and  let $K$ be a totally real number field. 
Let $L/K$ be a finite abelian CM extension and let $G=\Gal(L/K)$.
Let $S$ be a finite set of places of $K$
containing  $S_{p} \cup S_{\ram}(L/K) \cup S_{\infty}$.
Assume that $\zeta_{p} \in L$.

Let $L_{\infty}$ denote the cyclotomic $\Z_{p}$-extension of $L$.
Let $\mathcal{G} = \Gal(L_{\infty}/K)$, which we write as $\mathcal{G} = H \times \Gamma$,
where $\Gamma \simeq \Z_{p}$
and $H := \Gal(L_{\infty} / K_{\infty})$ canonically identifies with 
a subgroup of $G$.
Let $R = \Z_{p} \llbracket \Gamma \rrbracket$.
Let $j \in \mathcal{G}$ denote complex conjugation (this an abuse of notation because its image in the quotient group $G$ is also denoted by $j$) and let
$\mathcal{G}^{+} := \mathcal{G} / \langle j \rangle = \Gal(L_{\infty}^{+}/K)$.
Then $j \in H$ and so again $\Lambda(\mathcal{G}^{+})$ is a free $R$-order in $\mathcal{Q}(\mathcal{G}^{+})$. 
Moreover, $\Lambda(\mathcal{G})_{-} := \Lambda(\mathcal{G}) / (1+j)$ is also a free $R$-order.
For any $\Lambda(\mathcal{G})$-module $M$ we write $M^{+}$ and $M^{-}$ for the submodules of $M$ upon 
which $j$ acts as $1$ and $-1$, respectively, and consider these as modules over $\Lambda(\mathcal{G}^{+})$ and
$\Lambda(\mathcal{G})_{-}$, respectively. 

Let $\chi_{\mathrm{cyc}}:\mathcal{G} \rightarrow \Z_{p}^{\times}$ denote the $p$-adic cyclotomic character.
Let $\mu_{p^{n}}=\mu_{p^{n}}(L_{\infty})$ denote the group of $p^{n}$th roots of unity in $L_{\infty}^{\times}$
and let $\mu_{p^{\infty}}$ be the nested union (or direct limit) of these groups.
Let $\Z_{p}(1):= \varprojlim_{n} \mu_{p^{n}}$ be endowed with the action of $\mathcal{G}$ given by 
$\chi_{\mathrm{cyc}}$.
For any $r \geq 0$ define $\Z_{p}(r) := \Z_{p}(1)^{\otimes r}$ and $\Z_{p}(-r) := \Hom_{\Z_{p}}(\Z_{p}(r),\Z_{p})$
endowed with the naturally associated actions. 
For any $\Lambda(\mathcal{G})$-module $M$, we define the $r$th Tate twist to be $M(r):= \Z_{p}(r) \otimes_{\Z_{p}} M$
with the natural $\mathcal{G}$-action; hence $M(r)$ is simply $M$ with the modified $\mathcal{G}$-action 
$g \cdot m = \chi_{\mathrm{cyc}}(g)^{r} g(m)$ for $g \in \mathcal{G}$ and $m \in M$.
In particular, we have $\Q_{p} / \Z_{p} (1) \simeq \mu_{p^{\infty}}$ and $\Lambda(\mathcal{G}^{+})(-1) \cong \Lambda(\mathcal{G})_{-}$. 

For every place $v$ of $K$ we denote the decomposition subgroup of $\mathcal{G}$ at a chosen prime $w_{\infty}$
above $v$ by $\mathcal{G}_{w_{\infty}}$ (everything will only depend on $v$ and not on $w_{\infty}$ in the following).
In a cyclotomic $\Z_{p}$-extension only finitely many primes lie above any rational prime,
and hence the index $[\mathcal{G}:\mathcal{G}_{w_{\infty}}]$ is finite when $v$ is a finite place of $K$.
Let $\sigma_{w_{\infty}}$ denote the Frobenius automorphism at $w_{\infty}$.

For $n \geq 0$ let $G_{n} = \Gal(L_{n}/K)$, where $L_{n}$ is the $n$th layer of $L_{\infty}$. 
Denote the canonical projection map
$\Lambda(\mathcal{G}) \rightarrow \Z_{p}[G_{n}]$ by $\aug_{n}$.
We let $\Upsilon$ be the multiplicatively closed subset of $\Lambda(\mathcal{G})$
comprising all $x \in \Lambda(\mathcal{G})$ such that
$\aug_{n}(x)$ is a nonzerodivisor in $\Z_{p}[G_{n}]$ for all $n \geq 0$. 
Then $\aug_{n}$ extends uniquely to a $\Q_{p}$-algebra morphism 
$\aug_{n}: \Upsilon^{-1} \Lambda(\mathcal{G}) \rightarrow \Q_{p}[G_{n}]$.
Similar observations hold for the projection map
$\aug_{\Gamma_{K}}: \Lambda(\Gamma_{K}) \rightarrow \Z_{p}$.
For an irreducible 
$\Q_{p}^{c}$-valued character $\chi$ of $G_{n}$,
let $e_{n}(\chi)$ denote the associated primitive idempotent of $\Q_{p}^{c}[G_{n}]$.
Then for each $n \geq 0$ and each $x \in \Upsilon^{-1} \Lambda(\mathcal{G})$
we have an equality
\begin{equation}\label{eqn:aug-maps}
	\aug_{n}(x) = \sum_{\chi} \aug_{\Gamma_{K}}(j_{\chi}(x)) e_{n}(\chi),
\end{equation}
where the sum runs over all irreducible 
$\Q_{p}^{c}$-valued characters $\chi$ of $G_{n}$.
This is essentially a consequence of the facts that 
$\aug_{\Gamma_{K}}(j_{\chi}(\gamma_{\chi})) = 1$ and
$\gamma_{\chi}$ acts trivially upon $V_{\chi}$. 
(See \cite[Theorem 6.4]{MR2609173} and its proof for an even more general result.)
Note that in particular $\Phi_{S} \in \Upsilon^{-1} \Lambda(\mathcal{G})$.
 
For any $r \in \Z$, let $t_{\mathrm{cyc}}^{r}$ denote the $\Q_{p}$-algebra automorphism of $\Upsilon^{-1} \Lambda(\mathcal{G})$
induced by
$g \mapsto \chi_{\mathrm{cyc}}^{r}(g) g$ for $g \in \mathcal{G}$.
This restricts to an $\Z_{p}$-algebra automorphism of $\Lambda(\mathcal{G})$ and for $r=1$ induces an isomorphism 
$\Lambda(\mathcal{G}^{+})(-1) \cong \Lambda(\mathcal{G})_{-}$.
Let $x \mapsto x^{\#}$ denote the anti-involution on 
$\Upsilon^{-1} \Lambda(\mathcal{G})$
induced by $g \mapsto g^{-1}$ for $g \in \mathcal{G}$.

Let $T$ be a finite set of places of $K$ such that $S \cap T = \emptyset$.
We define $\Psi_{S,T} \in \Upsilon^{-1} \Lambda(\mathcal{G})$ by
\begin{equation}\label{eqn:definition-of-Psi}
\Psi_{S,T} = \Psi_{S,T}(L_{\infty} / K) :=  t_{\mathrm{cyc}}^{1}(\Phi_{S}) \cdot \textstyle{\prod_{v \in T} \xi_{v}},
\end{equation}
where
$\xi_{v} := 1 - \chi_{\mathrm{cyc}}(\sigma_{w_{\infty}}) \sigma_{w_{\infty}}$.
The following result for $r=0$ is essential for the proof of Theorem \ref{thm:EIMC-abelian-exts}. 
The case $r<0$ will be needed in \S \ref{sec:ETNC-and-CS}.

\begin{prop}\label{prop:PsiST-is-inverse-limit}
Assume that $\zeta_{p} \in L$. 
For every integer $r \leq 0$
we have
\[
t_{\mathrm{cyc}}^{r}(\Psi_{S,T}^{\#}) = \varprojlim_{n} \Theta_{S,T}(L_{n}/K, r).
\]
\end{prop}

\begin{proof}
Variants of this result are certainly well known.
The case at hand is stated in \cite[Lemma 5.14 (2)]{MR3383600}, but
it relies on \cite[Proposition 4.1]{MR2571700} where the case $r<0$
is left to the reader. 
We include the short argument for convenience as it provides the crucial
link between complex and $p$-adic Artin $L$-functions.

We have to show that the image of $t_{\mathrm{cyc}}^{r}(\Psi_{S,T}^{\#})$
under the projection map $\aug_n$ is equal to $\Theta_{S,T}(L_{n}/K, r)$
for each $n \geq 0$.
In fact, it suffices to show this for sufficiently large $n$.

It is straightforward to check that $t_{\mathrm{cyc}}^{r}(\xi_v^{\#})
= 1-\chi_{\mathrm{cyc}}^{1-r}(\sigma_{w_{\infty}})\sigma_{w_{\infty}}^{-1}$.
Hence since $\sigma_{w_{\infty}}$ acts by $N(v)$ on $p$-power
roots of unity, for each $v \in T$ and $n \geq 0$ we have 
\begin{equation} \label{eqn:aug-of-xi}
\aug_{n}(t_{\mathrm{cyc}}^{r}(\xi_v^{\#}))
= 1 - N(v)^{1-r}\sigma_{w_{n}}^{-1} 
= \delta_{\left\{v\right\}}(L_{n}/K, r),
\end{equation}
where $w_{n}$ is the place of $L_n$
below $w_{\infty}$ and 
$\delta_{\left\{v\right\}}(L_{n}/K, r)$ is defined in \eqref{eq:def-delta-T}.

Since a sufficiently large $p$-power of
$\Gamma_{L} = \Gal(L_{\infty}/L) \leq \mathcal{G}$ is contained in $\Gamma$,
there exists $m \geq 0$ such that for all $n \geq m$,
we have a decomposition $G_{n} \simeq H \times \Gamma / \Gamma_{L_n}$.
Now fix $n \geq m$. 
Then the $\Q_{p}^{c}$-valued irreducible characters of $G_{n}$ are
precisely those of the form $\chi \otimes \rho$, 
where $\chi$ and $\rho$ are irreducible characters of $H$ and $\Gamma / \Gamma_{L_n}$,
respectively. 
We denote the associated primitive idempotent of $\Q_{p}^{c}[G_{n}]$
by $e_{n}(\chi \otimes \rho)$ and compute
\begin{eqnarray*}
\aug_n(t_{\mathrm{cyc}}^{r}(\Psi_{S,T}^{\#}))
& = & \delta_{T}(L_{n}/K,r) \cdot \aug_n(t_{\mathrm{cyc}}^{r-1}(\Phi_{S}^{\#})) \\
& = & \delta_{T}(L_{n}/K,r) \cdot \aug_n((t_{\mathrm{cyc}}^{1-r}(\Phi_{S}))^{\#}) \\
& = & \delta_{T}(L_{n}/K,r) \cdot  \sum_{\chi, \rho} \aug_{\Gamma_{K}} \left(
j_{\chi \otimes \rho}(t_{\mathrm{cyc}}^{1-r}(\Phi_{S})) \right)
e_{n}(\check \chi \otimes \check \rho) \\
& = & \delta_{T}(L_{n}/K,r) \cdot \sum_{\chi, \rho} 
\aug_{\Gamma_{K}} \left(
\frac{G_{\chi\omega^{1-r} \otimes \rho, S}(u^{1-r}\gamma_{K}-1)}{H_{\chi\omega^{1-r} \otimes \rho}(u^{1-r}\gamma_{K}-1)} \right) e_{n}(\check\chi \otimes \check\rho) \\
& = & \delta_{T}(L_{n}/K,r) \cdot \sum_{\chi, \rho} L_{p,S}(r,\chi\omega^{1-r} \otimes \rho) e_{n}(\check\chi \otimes \check\rho) \\
& = & \delta_{T}(L_{n}/K,r) \cdot \sum_{\chi, \rho} \iota(L_{S}(r,\iota^{-1}(\chi \otimes \rho))) e_{n}(\check\chi \otimes \check\rho) \\
& = & \Theta_{S,T}(L_{n}/K,r).
\end{eqnarray*}
Here the sums run over all irreducible $\Q_{p}^{c}$-valued characters
$\rho$ of $\Gamma/\Gamma_{L_n}$ and $\chi$ of $H$ with $\chi$ odd if
$r$ is even and $\chi$ even otherwise.
The first equality is a consequence of \eqref{eqn:definition-of-Psi} and \eqref{eqn:aug-of-xi}.
The second is clear and the third is \eqref{eqn:aug-maps}
with $x = (t_{\mathrm{cyc}}^{1-r}(\Phi_{S}))^{\#}$.
The fourth equality is implied by \cite[Lemma 6.1]{MR3980291}, \eqref{eqn:L_K,S}
and \eqref{eqn:Phi_S-definition}.
(Note that $u := \kappa(\gamma_K)
= \chi_{\mathrm{cyc}}(\gamma)$ if $\gamma$ maps to $\gamma_K$ under the
canonical isomorphism $\Gamma \cong \Gamma_K$.)
The fifth and sixth equalities follow from 
\eqref{eqn:p-adic-L-series} and \eqref{eqn:interpolation-property}, respectively. 
(Note that when the characters in question are linear, the
interpolation property \eqref{eqn:interpolation-property} also holds for $r=0$.)
The last equality follows from the definition of $\Theta_{S,T}(L_{n}/K,r)$.
\end{proof}

\section{The proof of Theorem \ref{thm:EIMC-abelian-exts}}\label{sec:proof-of-EIMC-abelian-exts}

\subsection{A reduction step}
The following lemma will allow us to perform an important reduction step.

\begin{lemma}\label{lem:enlarge-extension}
Let $p$ be an odd prime and let $K$ be a totally real number field. 
Let $\mathcal{L}/K$ be an abelian admissible one-dimensional $p$-adic Lie extension of $K$.
There exists a finite abelian CM extension $L/K$ such that 
{\upshape(i)} $\zeta_{p} \in L$, {\upshape(ii)} $L \cap K_{\infty} = K$ and {\upshape(iii)} 
$\mathcal{L} \subseteq L_{\infty}^{+}$,
where $L_{\infty}$ is the cyclotomic $\Z_{p}$-extension of $L$ 
and $L_{\infty}^{+}$ is its maximal totally real subfield.
\end{lemma}

\begin{proof}
Let $\mathcal{G}=\Gal(\mathcal{L}/K)$, let $H=\Gal(\mathcal{L}/K_{\infty})$ and let $\Gamma_{K}=\Gal(K_{\infty}/K)$. 
As in \S \ref{subsec:admissible-one-dim-extns}, we obtain a semidirect product 
$\mathcal{G} = H \rtimes \Gamma$ where $\Gamma \leq \mathcal{G}$ and
$\Gamma \simeq \Gamma_{K} \simeq \Z_{p}$. 
Since $\mathcal{G}$ is abelian, the semidirect product is in fact direct.
Let $F$ be the subfield of $\mathcal{L}$ fixed by $\Gamma$. 
Then $F_{\infty} = \mathcal{L}$ and $ F \cap K_{\infty}=K$.
Now let $L=F(\zeta_{p})$. Then $L/K$ is a finite abelian CM extension and $L$ satisfies properties (i), (ii) and (iii).
We note that the choice of $\Gamma$ and hence of $L$ is non-canonical.
\end{proof}

\subsection{Minus $p$-parts of ray class groups in cyclotomic $\Z_{p}$-extensions}\label{subsec:minus-p-parts-ray-class-groups}
Let $p$ be an odd prime and let $K$ be a totally real number field. 
Let $L/K$ be a finite abelian CM extension.
Let $T$ be a finite set of finite places of $K$. 
For $n \geq 0$, let $L_{n}$ denote the $n$th layer of the cyclotomic $\Z_{p}$-extension $L_{\infty}$ of $L$ and
let $A_{L_{n}}^{T} = (\Z_{p} \otimes_{\Z} \cl_{L_{n}}^{T})^{-}$ 
(see~\S \ref{subsec:strong-BS-for-abelian-extns}).
The following result is \cite[Lemma 2.9]{MR3383600};
we include the short proof for the convenience of the reader. 

\begin{lemma}\label{lem:injectivity-ray-class-groups-cyc-Zp-exts}
The canonical maps $A_{L_{n}}^{T} \rightarrow A_{L_{n+1}}^{T}$ are injective for all $n \geq 0$.
\end{lemma}

\begin{proof}
In the case $T=\emptyset$ this is \cite[Proposition 13.26]{MR1421575}. 
The general case follows from the case $T=\emptyset$ by a snake lemma argument applied to 
the exact sequence \eqref{eqn:ray-class-sequence} after taking `minus $p$-parts' (this is an exact functor since $p$ is odd).
\end{proof}

\subsection{An exact sequence of Iwasawa modules}\label{subsec:ex-seq-Iwasawa}
Let $A_{L_{\infty}}^{T} = \varinjlim_{n} A_{L_{n}}^{T}$ where the transition maps are the canonical ones.
If $T$ is empty we further abbreviate $A_{L_{\infty}}^{T}$ to $A_{L_{\infty}}$.
We henceforth assume the notation of \S \ref{sec:inverse-limits-stickelberger}.

\begin{lemma}\label{lem:exact-seq-R-torsion-modules}
Assume that $\zeta_{p} \in L$. 
Let $T$ be a finite set of finite places of $K$ containing primes of at least two different residue characteristics.
Then we have an exact sequence of $\Lambda(\mathcal{G})_{-}$-modules
\begin{equation}\label{eqn:four-term-exact-seq}
0 \rightarrow X_{S_p}(-1) \rightarrow
\Hom(A_{L_{\infty}}^{T}, \Q_{p} / \Z_{p}) \rightarrow \bigoplus_{v \in T} 
\left( \ind_{\mathcal{G}_{w_{\infty}}}^{\mathcal{G}} \Z_{p}(-1) \right)^{-} \rightarrow
\Z_{p}(-1) \rightarrow 0,
\end{equation}
where each term is finitely generated and torsion over $R$. 
Moreover, the last two non-trivial terms are finitely generated over $\Z_{p}$.
\end{lemma}

\begin{proof}
The hypothesis on $T$ ensures that $E_{L_{n}}^{T}$ is torsionfree for every $n \geq 0$
(see Remark~\ref{rmk:conditions-on-T}).
Thus taking $p$-minus part of sequence \eqref{eqn:ray-class-sequence}
for each layer $L_{n}$ yields exact sequences
\[
0 \longrightarrow \mu_{p^{\infty}}(L_{n}) \longrightarrow  
\left( \Z_{p} \otimes_{\Z} (\mathcal{O}_{L_{n}}/\mathfrak{M}_{L_{n}}^{T})^{\times} \right)^{-}
\longrightarrow A_{L_{n}}^{T} \longrightarrow A_{L_{n}} \longrightarrow 0,
\] 
where $\mu_{p^{\infty}}(L_{n})$ denotes the group of all $p$-power roots of unity in $L_{n}$. 
Since $\zeta_{p} \in L$ by assumption, we have $\zeta_{p^{n}} \in \mu_{p^{\infty}}(L_{n})$ for all $n \geq 0$.
Thus taking direct limits yields
\[
0 \longrightarrow \Q_p/\Z_p(1) \longrightarrow \bigoplus_{v \in T} 
\left( \ind_{\mathcal{G}_{w_{\infty}}}^{\mathcal{G}} \Q_p/\Z_p(1) \right)^{-}
\longrightarrow A_{L_{\infty}}^{T} \longrightarrow A_{L_{\infty}} \longrightarrow 0,
\]
and then taking Pontryagin duals gives a new exact sequence
\[
0 \rightarrow \Hom(A_{L_{\infty}}, \Q_{p} / \Z_{p}) \rightarrow
\Hom(A_{L_{\infty}}^{T}, \Q_{p} / \Z_{p}) \rightarrow \bigoplus_{v \in T} 
\left( \ind_{\mathcal{G}_{w_{\infty}}}^{\mathcal{G}} \Z_{p}(-1) \right)^{-} \rightarrow
\Z_{p}(-1) \rightarrow 0.
\] 
But by Kummer duality \cite[Theorem 11.4.3]{MR2392026} there is a canonical
isomorphism of $\Lambda(\mathcal{G})_{-}$-modules
$X_{S_p}(-1) \cong \Hom(A_{L_{\infty}}, \Q_{p} / \Z_{p})$, 
and thus we obtain \eqref{eqn:four-term-exact-seq}.
Since the decomposition groups $\mathcal{G}_{w_{\infty}}$ have finite index
in $\mathcal{G}$, each term 
$\ind_{\mathcal{G}_{w_{\infty}}}^{\mathcal{G}} \Z_{p}(-1)$
is finitely generated over $\Z_{p}$.
Therefore the last two non-trivial terms of \eqref{eqn:four-term-exact-seq} 
are finitely generated as $\Z_{p}$-modules and thus are torsion and finitely generated over $R$.
Moreover, $X_{S_p}(-1)$ is also torsion and finitely generated over $R$ (see \S \ref{subsec:Iwasawa-module}) and therefore the same must be true of $\Hom(A_{L_{\infty}}^{T}, \Q_{p} / \Z_{p})$.
\end{proof}

\subsection{A consequence of the strong Brumer--Stark conjecture}
The following result crucially depends on the strong Brumer--Stark conjecture.

\begin{prop}\label{prop:psiST-in-Fitt}
Assume that $\zeta_{p} \in L$.
Let $S$ and $T$ be finite sets of places of $K$ 
such that $S_{p} \cup S_{\ram}(L/K) \cup S_{\infty} \subseteq S$, $S \cap T = \emptyset$, 
and $T$ contains primes of at least two different residue characteristics.
Then
\[
\Psi_{S,T} \in \Fitt_{\Lambda(\mathcal{G})_{-}}(\Hom(A_{L_{\infty}}^{T}, \Q_{p} / \Z_{p})).
\] 
\end{prop}

\begin{proof}
By Lemma \ref{lem:injectivity-ray-class-groups-cyc-Zp-exts},
the transition maps in $A_{L_{\infty}}^{T} = \varinjlim_{n} A_{L_{n}}^{T}$ are injective
and so those in
$\Hom(A_{L_{\infty}}^{T}, \Q_{p} / \Z_{p}) = \varprojlim_{n} (A_{L_{n}}^{T})^{\vee}$ are surjective.
By Lemma \ref{lem:exact-seq-R-torsion-modules}, 
$\Hom(A_{L_{\infty}}^{T}, \Q_{p} / \Z_{p})$ is finitely generated and torsion over $R$. 
Therefore by Theorem \ref{thm:GK-inverse-limit-Fitting-ideals} we have
\begin{equation}\label{eqn:equality-hom-proj-lim-of-fitts}
\Fitt_{\Lambda(\mathcal{G})_{-}}(\Hom(A_{L_{\infty}}^{T}, \Q_{p} / \Z_{p}))
= 
\textstyle{\varprojlim_{n}} \Fitt_{\Z_{p}[G_{n}]_-}((A_{L_{n}}^{T})^{\vee}). 
\end{equation}
Note that, in particular, the assumptions ensure that $\Hyp(L_{n}/K,S,T)$ holds for every $n \geq 0$
(see Remark~\ref{rmk:conditions-on-T}).
Hence by the $p$-part of the strong Brumer--Stark conjecture (Theorem \ref{thm:strong-brumer-stark}) we have
\[
\theta_{S}^{T}(L_{n}/K)^{\#} \in \Fitt_{\Z_{p}[G_{n}]_{-}}((A_{L_{n}}^{T})^{\vee})
\] 
for every $n \geq 0$.
Thus the desired result now follows from Proposition
\ref{prop:PsiST-is-inverse-limit} and \eqref{eqn:equality-hom-proj-lim-of-fitts}. 
\end{proof}

\subsection{The proof of Theorem \ref{thm:EIMC-abelian-exts}}\label{subsec:proof-of-EIMC-abelian-exts}
We now have all the pieces needed to prove our main result.

\begin{theorem}[Theorem \ref{thm:EIMC-abelian-exts}]\label{thm:EIMC-abelian-exts-proof}
Let $p$ be an odd prime and let $K$ be a totally real number field. 
Let $\mathcal{L}/K$ be an abelian admissible one-dimensional $p$-adic Lie extension.
Then the EIMC with uniqueness holds for $\mathcal{L}/K$.
\end{theorem}

\begin{proof}
Let $L/K$ be as in Lemma \ref{lem:enlarge-extension}.
We first prove the EIMC for $L^{+}_{\infty}/K$.
Let $\mathcal{G}=\Gal(L_{\infty}/K)$.
Then $\mathcal{G} = H \times \Gamma$ where $H=\Gal(L_{\infty}/K_{\infty})$
and $\Gamma=\Gal(L_{\infty}/L)$.
Moreover, $\mathcal{G}^{+}:= \Gal(L_{\infty}^{+}/K) = H^{+} \times \Gamma$ where $H^{+}=\Gal(L_{\infty}^{+}/K_{\infty})$
and $\Gamma \cong \Gal(L_{\infty}^{+}/L^{+})$ via restriction.
Let $R=\Z_{p}\llbracket\Gamma\rrbracket$.

Let $S$ be a finite set of places of $K$ containing
$S_{p} \cup S_{\ram}(L/K) \cup S_{\infty}$.
By Proposition~\ref{prop:abelian-EIMC-equiv-local-at-(p)} the 
EIMC for $L^{+}_{\infty}/K$ is equivalent to the assertion that
\begin{equation}\label{eqn:Phi_S-in-Fitt}
\Phi_{S} \in \Fitt_{\Lambda_{(p)}(\mathcal{G}^{+})}((X_{S})_{(p)}). 
\end{equation}
As the decomposition groups $\mathcal{G}_{w_{\infty}}$ have finite index
in $\mathcal{G}$, \cite[Corollary 11.3.6(i)]{MR2392026} shows that the canonical projection
$X_S \rightarrow X_{S_p}$ induces an isomorphism
${(X_S)_{(p)} \cong (X_{S_p})_{(p)}}$ of $\Lambda_{(p)}(\mathcal{G})$-modules.
Moreover, since $t_{\mathrm{cyc}}^{1}$ induces an isomorphism 
$\Lambda_{(p)}(\mathcal{G}^{+})(-1) \cong \Lambda_{(p)}(\mathcal{G})_{-}$, the containment \eqref{eqn:Phi_S-in-Fitt}
is equivalent to the assertion that
\begin{equation} \label{eqn:Psi_S-in-Fitting-ideal}
	\Psi_{S} := t_{\mathrm{cyc}}^{1}(\Phi_{S}) \in \Fitt_{\Lambda_{(p)}(\mathcal{G})_{-}}(X_{S_p}(-1)_{(p)}).
\end{equation}

Let $T$ be a second finite set of places of $K$ containing primes of at least two different residue characteristics and such that $S \cap T = \emptyset$.
Then by Proposition \ref{prop:psiST-in-Fitt} we have
\begin{equation}\label{eq:PsiST-contained-in-Fitt-of-Hom}
\Psi_{S,T} \in \Fitt_{\Lambda(\mathcal{G})_{-}}(\Hom(A_{L_{\infty}}^{T}, \Q_{p} / \Z_{p})). 
\end{equation}

Now consider the exact sequence \eqref{eqn:four-term-exact-seq}
of Lemma \ref{lem:exact-seq-R-torsion-modules}.
Since the last two non-trivial terms are finitely generated as $\Z_{p}$-modules,
they vanish after localisation at $(p)$. 
Hence there is an isomorphism
\begin{equation}\label{eq:XSp-HomATinfty-iso}
X_{S_p}(-1)_{(p)} \cong \Hom(A_{L_{\infty}}^{T}, \Q_{p} / \Z_{p})_{(p)} 
\end{equation}
of $\Lambda_{(p)}(\mathcal{G})_{-}$-modules.
Therefore \eqref{eq:PsiST-contained-in-Fitt-of-Hom} and \eqref{eq:XSp-HomATinfty-iso}
together imply that
\begin{equation}\label{eq:PsiST-contained-in-Fitt-of-XSp}
	\Psi_{S,T} \in \Fitt_{\Lambda_{(p)}(\mathcal{G})_{-}}(X_{S_p}(-1)_{(p)}). 
\end{equation}
Recall from \eqref{eqn:definition-of-Psi} that 
$\Psi_{S,T} :=  t_{\mathrm{cyc}}^{1}(\Phi_{S}) \cdot \textstyle{\prod_{v \in T} \xi_{v}}$.
Since the Euler factors $\xi_{v}$ become units in $\Lambda_{(p)}(\mathcal{G})$
for each $v \in T$ (in fact, $\xi_{v}$ generates the Fitting ideal of the
$\Lambda(\mathcal{G})$-module $\ind_{\mathcal{G}_{w_{\infty}}}^{\mathcal{G}} \Z_{p}(-1)$, 
which is finitely generated as a $\Z_{p}$-module), it follows from
\eqref{eq:PsiST-contained-in-Fitt-of-XSp} that \eqref{eqn:Psi_S-in-Fitting-ideal} holds.
Therefore the EIMC holds for $L_{\infty}^{+}/K$.
Thus the EIMC also holds for $\mathcal{L}/K$ by Lemma \ref{lem:EIMC-implies-EIMC-for-subextensions}.
Moreover, uniqueness holds by Remark \ref{rmk:SK1}.
\end{proof}

\section{Iwasawa algebras and commutator subgroups}

The following theorem is a restatement of a special case of \cite[Proposition 4.5]{MR3092262}. 
We include the proof here for the convenience of the reader and take the opportunity to correct some minor oversights in the proof of loc.\ cit.

\begin{theorem}\label{thm:p-nmid-comm-subgroup}
Let $p$ be a prime, let $\mathcal{G} = H \rtimes \Gamma$ 
be an admissible one-dimensional $p$-adic Lie group
and let $F/\Q_{p}$ be a finite extension with ring of integers $\mathcal{O}$.
Then the commutator subgroup $\mathcal{G}'$ of $\mathcal{G}$ is finite.
Moreover, 
$\Lambda^{\mathcal{O}}(\mathcal{G})$ is a direct product of matrix rings over (complete local)
commutative rings if and only if $p \nmid |\mathcal{G}'|$.
\end{theorem}
 
\begin{proof}
We adopt the setup and notation of \S \ref{subsec:Iwasawa-algebras}. 
We identify $R$ with $\mathcal{O}\llbracket T \rrbracket$ and abbreviate $\Lambda^{\mathcal{O}}(\mathcal{G})$ to $\Lambda$.
Let $\mathfrak{p}$ and $\mathfrak{P}$ denote the maximal ideals of $\mathcal{O}$ and $R$, respectively.
Then $\mathfrak{P}$ is generated by $\mathfrak{p}$ and $T$.
Let $k = R / \mathfrak{P} = \mathcal{O}/\mathfrak{p}$ be the residue field, which is finite and of characteristic $p$. Let $C_{p^n}$ denote the cyclic group of order $p^n$.
Since $\gamma^{p^n} = 1 + T \equiv 1 \bmod \mathfrak{P}$, we have
\begin{equation}\label{eq:overline-lambda}
\overline{\Lambda} := \Lambda / \mathfrak{P} \Lambda = \bigoplus_{i=0}^{p^{n}-1} k[H]\gamma^{i}
= k[H \rtimes C_{p^n}] \cong k \otimes_{R} \Lambda. 
\end{equation}
Since $\mathcal{G} / H \cong \Gamma$ is abelian, $\mathcal{G}'$ is actually a subgroup of $H$
and thus is finite. Moreover, $\mathcal{G}'$ identifies with the commutator subgroup of $H \rtimes C_{p^n}$.

We refer the reader to \cite{MR0121392} for background on separability and 
recall that a ring is said to be an Azumaya algebra if it is separable over its centre. 
We shall show that the following assertions are equivalent.
\begin{enumerate}
\item $\Lambda$ is a direct product of matrix rings over (complete local) commutative rings;
\item $\overline{\Lambda}$ is a direct product of matrix rings over commutative rings;
\item $\Lambda$ is an Azumaya algebra;
\item $\overline{\Lambda}$ is an Azumaya algebra;
\item $p \nmid |\mathcal{G}'|$.
\end{enumerate}

As any matrix ring over a commutative ring is an Azumaya algebra, (i) $\Rightarrow$ (iii)
and (ii) $\Rightarrow$ (iv).
In fact, as remarked after \cite[Corollary, p.\ 390]{MR704622} we have (ii) $\Leftrightarrow$ (iv).
By \cite[Corollary p.\ 389]{MR704622} we have (iv) $\Leftrightarrow$ (v). 

We now show  (iii) $\Leftrightarrow$ (iv).
By \cite[Example 23.3]{MR1838439} $\zeta(\Lambda)$ is semiperfect
and thus a product of local rings by \cite[Theorem 23.11]{MR1838439},
say
$
\zeta(\Lambda) = \bigoplus_{i=1}^{r} R_{i},
$
where each $R_{i}$ contains $R$.
By \cite[Proposition 6.5 (ii)]{MR632548} each $R_{i}$ is in fact a complete local ring.
Let $\mathfrak{P}_{i}$ be the maximal ideal of $R_{i}$ and
$k_{i} := R_{i} / \mathfrak{P}_{i}$ be the residue field.
Note that we have
\begin{equation}\label{eq:zeta-lambda-bar-decomp}
\zeta(\overline{\Lambda})
=\zeta(\Lambda) \otimes_{R} k
= \bigoplus_{i=1}^{r} R_{i} \otimes_{R} k
= \bigoplus_{i=1}^{r}R_{i}/\mathfrak{P}R_{i}. 
\end{equation}
In order to justify the first equality, 
we observe that it is a straightforward consequence of the decomposition \eqref{eq:Lambda-R-decomp}
that the centre $\zeta(\Lambda)$ is a free $R$-module
of rank $c(\mathcal{G}/\Gamma_{0})$, where $c(A)$ denotes the number of
conjugacy classes of a group $A$; a basis is given by the class sums.
Similarly, it follows from \eqref{eq:overline-lambda} that
$\zeta(\overline{\Lambda})$ is a $k$-vector space
of dimension $c(H \rtimes C_{p^n}) = c(\mathcal{G}/\Gamma_0)$. 
Hence the obvious inclusion $\zeta(\Lambda) \otimes_{R} k
\subseteq \zeta(\overline{\Lambda})$ must be an equality.

Moreover, we also have
\begin{equation}\label{eq:lambda-lambda-bar-tensored-with-residue-fields}
\Lambda \otimes_{\zeta(\Lambda)} k_{i} = \Lambda  \otimes_{R_{i}} k_{i} 
\cong (\Lambda \otimes_{R} k) \otimes_{(R_{i} \otimes_{R} k)} (k_{i} \otimes_{R} k) \cong \overline{\Lambda} 
\otimes_{\zeta(\overline{\Lambda})} k_{i}.
\end{equation}
By \cite[Theorem 4.7]{MR0121392} $\Lambda$ is Azumaya if and only if $\Lambda \otimes_{\zeta(\Lambda)} k_{i}$
is separable over $k_{i}$ for each $i$. 
Similarly, by \eqref{eq:zeta-lambda-bar-decomp} and loc.\ cit.\
$\overline{\Lambda}$ is Azumaya if and only if
$\overline{\Lambda} \otimes_{\zeta(\overline{\Lambda})} k_{i}$
is separable over $k_{i}$ for each $i$. Therefore the claim now follows from 
\eqref{eq:lambda-lambda-bar-tensored-with-residue-fields}.

In summary, we have shown that (ii) $\Leftrightarrow$ (iii) $\Leftrightarrow$ (iv) $\Leftrightarrow$ (v) and
(i) $\Rightarrow$ (iii). Thus it remains to show (iii) $\Rightarrow$ (i).
Suppose (iii) holds. Since $\mathfrak{P}R_{i} \subset \mathfrak{P}_{i}$,
the canonical projection $R_{i} \onto k_{i}$ factors through
$R_{i} \onto R_{i} / \mathfrak{P}R_{i} = R_{i} \otimes_{R} k$.
Hence we have the corresponding homomorphisms of Brauer groups
\[
\Br(R_{i}) \rightarrow \Br(R_{i} / \mathfrak{P}R_{i}) \rightarrow \Br(k_{i}).
\]

Now $\Br(R_{i}) \rightarrow \Br(k_{i})$
is injective by \cite[Corollary 6.2]{MR0121392}
and hence $\Br(R_{i}) \rightarrow \Br(R_{i} / \mathfrak{P}R_{i})$ must also be injective.
This yields an embedding
\[
\Br(\zeta(\Lambda)) = \bigoplus_{i=1}^{r} \Br(R_{i}) \hookrightarrow
\bigoplus_{i=1}^{r} \Br(R_{i} \otimes_{R} k) = \Br(\zeta(\overline{\Lambda})).
\]
Since $\Lambda$ is Azumaya, it defines a class $[\Lambda] \in \Br(\zeta(\Lambda))$ which is mapped to
$[\overline{\Lambda}]$ via this embedding.
In particular, (iv) holds and we have already seen that this implies (ii). 
Hence $[\overline{\Lambda}]$ is trivial and thus so is $[\Lambda]$.
Let $\Lambda_{i}$ be the component of $\Lambda$ corresponding to $R_{i}$.
Then $[\Lambda_{i}] \in \Br(R_{i})$ is trivial and so by \cite[Proposition 5.3]{MR0121392} $\Lambda_{i}$
is isomorphic to an $R_{i}$-algebra of the form $\Hom_{R_{i}}(P_{i},P_{i})$ where $P_{i}$ is a finitely generated projective faithful $R_{i}$-module.
Since $R_{i}$ is a local ring, $P_{i}$ must be free and so $\Lambda_{i}$ must be 
isomorphic to a matrix ring over its centre $R_{i}$. Thus (i) holds.
\end{proof}

\begin{corollary}\label{cor:p-nmid-comm-subgroup}
Let $p$ be a prime and let $\mathcal{G}$ be 
an admissible one-dimensional
$p$-adic Lie group such that $p \nmid |\mathcal{G}'|$. 
Let $F/\Q_{p}$ be a finite extension with ring of integers $\mathcal{O}$.
Then $\mathcal{Q}^{F}(\mathcal{G})$ is a direct product of matrix rings over fields
and there is a commutative diagram
\[
\xymatrix{
0 \ar[r]  & K_{1}(\Lambda^{\mathcal{O}}(\mathcal{G})) \ar[r]^{\iota} \ar[d]^{\cong} & K_{1}(\mathcal{Q}^{F}(\mathcal{G})) \ar[r]^{\partial \qquad} \ar[d]_{\nr}^{\cong}  &
K_{0}(\Lambda^{\mathcal{O}}(\mathcal{G}),\mathcal{Q}^{F}(\mathcal{G})) \ar[r] \ar[d]^{\cong} & 0\\
1 \ar[r]  &
\zeta(\Lambda^{\mathcal{O}}(\mathcal{G}))^{\times} \ar[r]  & \zeta(\mathcal{Q}^{F}(\mathcal{G}))^{\times} \ar[r] &
\zeta(\mathcal{Q}^{F}(\mathcal{G}))^{\times} /\zeta(\Lambda^{\mathcal{O}}(\mathcal{G}))^{\times} \ar[r] & 1,
}
\]
with exact rows. 
\end{corollary}

\begin{proof}
Apart from the injectivity of $\iota$, the existence of the top row and its exactness is
\eqref{eqn:Iwasawa-K-sequence}.
The exactness of the bottom row is tautological.
Since $\Lambda^{\mathcal{O}}(\mathcal{G})$ is a direct product of matrix rings over commutative local rings,
\cite[Proposition 45.12]{MR892316} and a Morita equivalence argument show that the left vertical map is an isomorphism. 
Moreover, an extension of scalars argument shows that $\mathcal{Q}^{F}(\mathcal{G})$ is direct product of matrix rings over fields, and so the middle vertical map is also an isomorphism.
The left square commutes since the reduced norm / determinant map is compatible with extensions of scalars. 
The left and middle vertical isomorphisms induce the right vertical isomorphism, and so the right square commutes. 
Finally, commutativity of the diagram shows that $\iota$ is injective. 
\end{proof}

\section{Further algebraic results and the proof of Corollary \ref{cor:EIMC-abelian-Sylow-p}}\label{sec:proof-of-main-corollary}

In this section, we begin by proving
purely algebraic results on the vanishing of $SK_{1}(\mathcal{Q}(\mathcal{G}))$
and on the injectivity of certain products of maps over subquotients of $\mathcal{G}$.
By combining these results with the functorial properties of the EIMC, we then show that 
Theorem \ref{thm:EIMC-abelian-exts} implies Corollary \ref{cor:EIMC-abelian-Sylow-p}.
Some results in this section are stated for all primes $p$ and others are only stated for odd primes $p$;
those in the latter case ultimately
rely on \cite{MR2205173} where it is a standing hypothesis that $p$ is odd. 

\subsection{$F$-$q$-elementary groups}\label{subsec:Qp-q-elementarys-subgroups}
Let $q$ be a prime. A finite group is said to be $q$\nobreakdash-hyper\-elementary if it is of the form 
$C_{n} \rtimes Q$, with $Q$ a $q$-group and 
$C_{n}$ a cyclic group of order $n$ such that $q \nmid n$. 
Let $F$ be a field of characteristic $0$.
A $q$-hyperelementary group $C_{n} \rtimes Q$
is called $F$-$q$-elementary if
\[
\mathrm{Im}(Q \longrightarrow \Aut(C_{n}) \cong (\Z/n\Z)^{\times}) \subseteq \Gal(F(\zeta_{n})/F).
\]
An $F$-elementary group is one that is $F$-$q$-elementary for some prime $q$.
A finite group is said to be $q$-elementary if it is of the form $C_{n} \times Q$ with $q \nmid n$
and $Q$ a $q$-group.

Now let $F/\Q_{p}$ be a finite extension and let $\mathcal{G}$ 
be an admissible one-dimensional $p$-adic Lie group.
Let $\Gamma_{0} \simeq \Z_{p}$ be an open central subgroup of
$\mathcal{G}$.
Then $\mathcal{G}$ is said to be:
\begin{itemize}
\item $F$-$q$-elementary if there is a choice of $\Gamma_{0}$ such that
$\mathcal{G}/\Gamma_{0}$ is $F$-$q$-elementary;
\item $F$-elementary if it is $F$-$q$-elementary for some prime $q$;
\item $q$-elementary if there is a choice of $\Gamma_{0}$ such that $\mathcal{G}/\Gamma_{0}$ 
is $q$-elementary;
\item elementary if it is $q$-elementary for some prime $q$.
\end{itemize}

\begin{lemma}\label{lem:elementary-compatible-RW}
In the case $F=\Q_{p}$ the definition of $F$-$q$-elementary given above is equivalent to the corresponding definitions of \cite[\S 2]{MR2205173} ($p=q$) and \cite[\S 3]{MR2205173} ($p \neq q$).
\end{lemma}

\begin{proof}
Let $\mathcal{G}$ be an admissible one-dimensional $p$-adic Lie group.
It is clear that if $\mathcal{G}$ satisfies the definitions of Ritter and Weiss given in \cite[\S 2, \S 3]{MR2205173}, then it satisfies the definition given above. 
The converse is given by \cite[Lemma 4]{MR2205173} in the case $p \neq q$ and by the following calculation when $p=q$.
Suppose more generally that $F/\Q_{p}$ is a finite extension and that
we have a short exact sequence
\[
1 \longrightarrow \Gamma_{0} \longrightarrow \mathcal{G}
\longrightarrow C_{n} \rtimes Q \longrightarrow 1,
\]
where $C_{n} \rtimes Q$ is $F$-$p$-elementary and $\Gamma_{0} \simeq \Z_{p}$ is an open central subgroup of $\mathcal{G}$.
Let $s \in \mathcal{G}$ be a pre-image of a generator of $C_{n}$.
Since $n$ and $p$ are coprime, we may multiply $s$ by a suitable element in $\Gamma_{0}$
to obtain an element of order $n$. Thus we can and do assume without loss of generality that $s$ itself has order $n$. Let $\mathcal{P} \subseteq \mathcal{G}$ be the pre-image of $Q$.
Then $\mathcal{P}$ is an open pro-$p$ subgroup of $\mathcal{G}$
and $\mathcal{G} \simeq C_{n} \rtimes \mathcal{P}$,
where $C_{n}$ is generated by $s$. Since $\Gamma_{0}$ is central
in $\mathcal{G}$, the action of $\mathcal{P}$ on $C_{n}$ factors through 
$\mathcal{P} \rightarrow Q \rightarrow \Aut(C_{n})$ and thus has image in
$\Gal(F(\zeta_{n})/F)$.
\end{proof}

\subsection{The kernel of the reduced norm map}\label{subsec:nr-map}
Let $p$ be a prime and let $\mathcal{G}$ 
be an admissible one-dimensional $p$-adic Lie group.
Let $F/\Q_{p}$ be a finite extension. 
Define
\begin{equation*}\label{eq:def-SK1}
SK_{1}(\mathcal{Q}^{F}(\mathcal{G})) = \ker(\nr: K_{1}(\mathcal{Q}^{F}(\mathcal{G})) \longrightarrow \zeta(\mathcal{Q}^{F}(\mathcal{G}))^{\times}).
\end{equation*}

\begin{prop} \label{prop:SK_1-reduction-step}
Let $p$ be an odd prime and let $\mathcal{G}$ 
be an admissible one-dimensional $p$-adic Lie group. 
If $SK_{1}(\mathcal{Q}(\mathcal{H}))=0$ for all open $\Q_{p}$-$p$-elementary subgroups
$\mathcal{H}$ of $\mathcal{G}$ then $SK_{1}(\mathcal{Q}(\mathcal{G}))=0$ 
\end{prop}

\begin{proof}
By  \cite[Corollary on p.\ 167]{MR2205173} we have that $SK_{1}(\mathcal{Q}(\mathcal{G})) = 0$ if
$SK_{1}(\mathcal{Q}(\mathcal{H}))$ vanishes for all open $\Q_{p}$-elementary subgroups $\mathcal{H}$ of $\mathcal{G}$.
If $q$ is a prime distinct from $p$ and $\mathcal{H}$ is an open $\Q_{p}$-$q$-elementary subgroup of $\mathcal{G}$ then $SK_{1}(\mathcal{Q}(\mathcal{H}))=0$ by a result of Lau \cite[Theorem 2]{MR2875338}. The case $p=q$ holds by hypothesis.
\end{proof}

\begin{corollary}\label{cor:SK1-vanishes-when-Sylow-p-subgroup-is-abelian}
If $p$ is an odd prime and 
$\mathcal{G}$ is an admissible one-dimensional $p$-adic Lie group with an abelian Sylow $p$-subgroup
then $SK_{1}(\mathcal{Q}(\mathcal{G}))=0$.
\end{corollary}

\begin{proof}
By Proposition \ref{prop:SK_1-reduction-step} it suffices to show that $SK_{1}(\mathcal{Q}(\mathcal{H}))=0$ for all open $\Q_{p}$\nobreakdash-$p$\nobreakdash-ele\-men\-tary subgroups $\mathcal{H}$ of $\mathcal{G}$.
Let $\mathcal{H}$ be such a subgroup.
Then by Lemma \ref{lem:elementary-compatible-RW} $\mathcal{H} = \langle s \rangle \rtimes \mathcal{U}$ where $\langle s \rangle$ is a finite cyclic subgroup of order prime to $p$ and $\mathcal{U}$ is an open pro-$p$ subgroup.
Moreover, $\mathcal{U}$ must be abelian by the hypothesis on $\mathcal{G}$ and so 
the commutator subgroup $\mathcal{H}'$ of $\mathcal{H}$ is necessarily a subgroup of $\langle s \rangle$.
Hence $p \nmid |\mathcal{H}'|$ and so the reduced norm map
$\nr: K_{1}(\mathcal{Q}(\mathcal{H})) \rightarrow \zeta(\mathcal{Q}(\mathcal{H}))^{\times}$
is an isomorphism by Corollary \ref{cor:p-nmid-comm-subgroup} (with $F=\Q_{p}$).
In particular, $SK_{1}(\mathcal{Q}(\mathcal{H}))=0$.
\end{proof}

\begin{remark}\label{rmk:Suslin-conj}
As noted in \cite[Remark E]{MR2114937} (see also \cite[Remark 3.5]{MR3294653}), a conjecture of Suslin implies that in fact $SK_{1}(\mathcal{Q}(\mathcal{G}))$ always vanishes. 
\end{remark}

\subsection{Products of maps over subquotients of $\mathcal{G}$}
For $F/\Q_{p}$ a finite extension with ring of integers $\mathcal{O}$, we abuse notation and let
\[
\nr: K_{1}(\Lambda^{\mathcal{O}}(-)) \longrightarrow \zeta(\mathcal{Q}^{F}(-))^{\times}
\]
denote the composition
of the canonical map
$K_{1}(\Lambda^{\mathcal{O}}(-)) \rightarrow K_{1}(\mathcal{Q}^{F}(-))$
and the reduced norm map $\nr: K_{1}(\mathcal{Q}^{F}(-)) \rightarrow 
\zeta(\mathcal{Q}^{F}(-))^{\times}$.
The purpose of this subsection is to prove the following result. 

\begin{theorem}\label{thm:reduction-to-p-elementary-subquotients}
Let $p$ be an odd prime and let $\mathcal{G}$ be an admissible one-dimensional $p$-adic Lie group
and let $\Gamma_{0} \simeq \Z_{p}$ be an open central subgroup of $\mathcal{G}$.
Let $\mathcal{E}_{p}$ denote the collection of all
$p$-elementary subquotients of $\mathcal{G}$ of the form $\mathcal{U}/N$,
where $\Gamma_{0} \leq \mathcal{U} \leq \mathcal{G}$ and $N$ is a finite normal subgroup
of $\mathcal{U}$.
Then the product of maps
\[
\zeta(\mathcal{Q}(\mathcal{G}))^{\times} / 
\nr(K_{1}(\Lambda(\mathcal{G})))
\xrightarrow{\prod \quot \circ \res}
\prod_{\mathcal{H} \in \mathcal{E}_{p}}
\zeta(\mathcal{Q}(\mathcal{H}))^{\times} / 
\nr(K_{1}(\Lambda(\mathcal{H})))
\]
is injective.
If we further assume that $SK_{1}(\mathcal{Q}(\mathcal{G}))=0$,
then the product of maps 
\[
K_{0}(\Lambda(\mathcal{G}),\mathcal{Q}(\mathcal{G})) \xrightarrow{\prod \quot \circ \res}
\prod_{\mathcal{H} \in \mathcal{E}_{p}}
K_{0}(\Lambda(\mathcal{H}),\mathcal{Q}(\mathcal{H})) 
\]
is also injective. 
\end{theorem}

\begin{remark}\label{rmk:prod-is-finite}
In Theorem \ref{thm:reduction-to-p-elementary-subquotients}, there are only finitely many choices
for $\mathcal{U}$ since $\Gamma_{0}$ is open in $\mathcal{G}$ and only finitely many choices 
for $N$ since $N \leq H$, where $H$ is the finite normal subgroup of $\mathcal{G}$
consisting of all elements of finite order (see \S \ref{subsec:Iwasawa-algebras}).
Therefore $\mathcal{E}_{p}$ is finite.
Note that in the special case $\mathcal{G} = \Gamma_{0} \times H$, the collection
$\mathcal{E}_{p}$ consists of all groups of the form $\Gamma_{0} \times E$ where $E$ ranges over all 
$p$-elementary subquotients of $H$.
\end{remark}

We shall first prove several auxiliary and intermediate results which may be of interest in their own right. 

\begin{lemma}\label{lem:hom-rel-K0-nr}
Let $p$ be a prime and let $\mathcal{G}$ be an admissible one-dimensional $p$-adic Lie group.  
Let $F/\Q_{p}$ be a finite extension with ring of integers $\mathcal{O}$.
Then there exists a commutative diagram 
\begin{equation*}\label{eq:comm-diag-mod-nr-K1}
\xymatrix{
K_{1}(\Lambda^{\mathcal{O}}(\mathcal{G})) \ar[r] \ar[rd]_{\nr} & K_{1}(\mathcal{Q}^{F}(\mathcal{G})) \ar[r]^{\partial \qquad} \ar[d]_{\nr} &
K_{0}(\Lambda^{\mathcal{O}}(\mathcal{G}),\mathcal{Q}^{F}(\mathcal{G})) \ar[r] \ar[d] & 0\\
& \zeta(\mathcal{Q}^{F}(\mathcal{G}))^{\times} \ar[r] &
\zeta(\mathcal{Q}^{F}(\mathcal{G}))^{\times} /\nr(K_{1}(\Lambda^{\mathcal{O}}(\mathcal{G})) \ar[r] & 1
}
\end{equation*}
with exact rows. 
Moreover, if $SK_{1}(\mathcal{Q}^{F}(\mathcal{G}))=0$ then the right vertical map is injective.
\end{lemma}

\begin{proof}
The triangle commutes by definition.
The top row is \eqref{eqn:Iwasawa-K-sequence} and the existence
of the right vertical map follows from the exactness of this row.
The second claim follows from the snake lemma. 
\end{proof}

\begin{lemma}\label{lem:extension-of-scalars-injective-for-tame-extensions}
Let $p$ be a prime and let $\mathcal{G}$ be an admissible one-dimensional $p$-adic Lie group. 
Let $F/\Q_{p}$ be a finite extension that is at most tamely ramified and let $\mathcal{O}$ 
be the ring of integers of $F$.  
Then the canonical map
\[
\zeta(\mathcal{Q}(\mathcal{G}))^{\times} / 
\nr(K_{1}(\Lambda(\mathcal{G}))) \longrightarrow
\zeta(\mathcal{Q}^{F}(\mathcal{G}))^{\times} / 
\nr(K_{1}(\Lambda^{\mathcal{O}}(\mathcal{G})))
\]
is injective.
If we further assume that $SK_{1}(\mathcal{Q}(\mathcal{G}))=0$,
then the extension of scalars map
\[
K_{0}(\Lambda(\mathcal{G}),\mathcal{Q}(\mathcal{G}))
\longrightarrow
K_{0}(\Lambda^{\mathcal{O}}(\mathcal{G}),\mathcal{Q}^{F}(\mathcal{G}))
\]
is also injective.
\end{lemma}

\begin{proof}
By enlarging $F$ if necessary, we can and do assume that $F/\Q_{p}$ is Galois. 
The first claim follows from the equalities
\[
\zeta(\mathcal{Q}(\mathcal{G}))^{\times} \cap 
\nr(K_{1}(\Lambda^{\mathcal{O}}(\mathcal{G}))) = 
\nr(K_{1}(\Lambda^{\mathcal{O}}(\mathcal{G})))^{\Gal(F/\Q_{p})}
= \nr(K_{1}(\Lambda(\mathcal{G}))),
\]
where the last equality is \cite[Theorem 2.12]{MR2905563}.
(We point out that the `notation as above' in the statement of loc.\ cit.\
refers to \cite[Theorem 2.11]{MR2905563} rather than the text between these
two results; the simplifying assumptions are to be understood as
`without loss of generality'. Indeed the proof of
\cite[Theorem 2.12]{MR2905563} remains valid unchanged for finite
tamely ramified extensions of $\mathbb{Q}_p$.)
We have a commutative diagram
\[ \xymatrix{
K_{0}(\Lambda(\mathcal{G}),\mathcal{Q}(\mathcal{G})) \ar[r] \ar[d]
& \zeta(\mathcal{Q}(\mathcal{G}))^{\times} / 
\nr(K_{1}(\Lambda(\mathcal{G}))) \ar[d]\\
K_{0}(\Lambda^{\mathcal{O}}(\mathcal{G}),\mathcal{Q}^{F}(\mathcal{G})) \ar[r]
& \zeta(\mathcal{Q}^{F}(\mathcal{G}))^{\times} / 
\nr(K_{1}(\Lambda^{\mathcal{O}}(\mathcal{G}))),
}
\]
where the existence of the horizontal maps follows from Lemma \ref{lem:hom-rel-K0-nr}.
If $SK_{1}(\mathcal{Q}(\mathcal{G}))=0$ then the top horizontal map is injective by
Lemma \ref{lem:hom-rel-K0-nr}, and so the second claim now follows from the commutativity of the diagram.
\end{proof}

\begin{lemma}\label{lem:extension-of-scalars-when-denom-units-of-centre-of-Iwasawa-algebra}
Let $p$ be a prime and let $\mathcal{G}$ be an admissible one-dimensional $p$-adic Lie group.
Let $F/\Q_{p}$ be a finite extension with ring of integers $\mathcal{O}$.
Then the canonical map
\[
\zeta(\mathcal{Q}(\mathcal{G}))^{\times} / 
\zeta(\Lambda(\mathcal{G}))^{\times} \longrightarrow
\zeta(\mathcal{Q}^{F}(\mathcal{G}))^{\times} / 
\zeta(\Lambda^{\mathcal{O}}(\mathcal{G}))^{\times}
\]
is injective. 
\end{lemma}

\begin{proof}
Write $\mathcal{G} = H \rtimes \Gamma$ where $H$ is finite and $\Gamma \simeq \Z_{p}$.
Let $\Gamma_{0}$ be an open subgroup of $\Gamma$ that is central in $\mathcal{G}$. Let $R=\Z_{p}\llbracket\Gamma_{0}\rrbracket$. Since $\zeta(\Lambda(\mathcal{G}))$ and $\zeta(\Lambda^{\mathcal{O}}(\mathcal{G}))$ are both $R$-orders, all of their elements are integral over
$R$ by \cite[Theorem 8.6]{MR1972204}.
Thus $\zeta(\Lambda^{\mathcal{O}}(\mathcal{G}))^{\times} \cap \zeta(\Lambda(\mathcal{G})) = \zeta(\Lambda(\mathcal{G}))^{\times}$ by \cite[Lemma 9.7]{MR703486}, for example.
Hence we have
\[
\zeta(\Lambda(\mathcal{G}))^{\times}
\subseteq \zeta(\mathcal{Q}(\mathcal{G}))^{\times} \cap \zeta(\Lambda^{\mathcal{O}}(\mathcal{G}))^{\times}
\subseteq \zeta(\Lambda(\mathcal{G})) \cap \zeta(\Lambda^{\mathcal{O}}(\mathcal{G}))^{\times}
=\zeta(\Lambda(\mathcal{G}))^{\times}.
\]
Therefore $\zeta(\mathcal{Q}(\mathcal{G}))^{\times} \cap \zeta(\Lambda^{\mathcal{O}}(\mathcal{G}))^{\times}
= \zeta(\Lambda(\mathcal{G}))^{\times}$, which gives the desired result.
\end{proof}

In the results that follow, the quotient and restriction maps on certain quotients of 
$\zeta(\mathcal{Q}^{F}(\mathcal{G}))^{\times}$
are induced by those defined in \S \ref{subsec:dets-and-nr}.

\begin{prop}\label{prop:injectivity-K-groups-direct-prod-delta}
Let $p$ be a prime and let $\mathcal{G} = \mathcal{H} \times \Delta$ 
where $\mathcal{H}$ is an admissible one-dimensional $p$-adic Lie group 
such that $p \nmid |\mathcal{H}'|$ and $\Delta$ is a finite group with $p \nmid |\Delta|$. 
Let $\mathcal{C}(\Delta)$ denote the collection of cyclic subquotients of $\Delta$. 
Then the products of maps
\begin{align*}
\zeta(\mathcal{Q}(\mathcal{G}))^{\times} / \zeta(\Lambda(\mathcal{G}))^{\times}
\xrightarrow{\prod \quot \, \circ \, \res}
&
\prod_{C \in \mathcal{C}(\Delta)}
\zeta(\mathcal{Q}(\mathcal{H} \times C))^{\times} / 
\zeta(\Lambda(\mathcal{H} \times C))^{\times},\\
\zeta(\mathcal{Q}(\mathcal{G}))^{\times} / \nr(K_{1}(\Lambda(\mathcal{G})))
\xrightarrow{\prod \quot \, \circ \, \res}
& \prod_{C \in \mathcal{C}(\Delta)}
\zeta(\mathcal{Q}(\mathcal{H} \times C))^{\times} / 
\nr(K_{1}(\Lambda(\mathcal{H} \times C))),
\end{align*}
and
\[	
K_{0}(\Lambda(\mathcal{G}),\mathcal{Q}(\mathcal{G})) 
\xrightarrow{\prod \quot \, \circ \, \res}
\prod_{C \in \mathcal{C}(\Delta)}
K_{0}(\Lambda(\mathcal{H} \times C),\mathcal{Q}(\mathcal{H} \times C)) 
\]
are all injective. 
\end{prop}

\begin{proof}
The hypotheses imply that $p \nmid |\mathcal{G}'|$.
Hence Corollary \ref{cor:p-nmid-comm-subgroup} implies that 
injectivity of the second and third displayed maps follows from that of the first displayed map. 

Set $d:=|\Delta|$. Let $F=\Q_p(\zeta_{d})$ and let $\mathcal{O}$ be the ring of integers of $F$.
Then $F/\Q_{p}$ is a finite unramified extension over which every representation of every subgroup of
$\Delta$ can be realised.
Moreover, there is a canonical decomposition $\zeta(\mathcal{O}[\Delta]) \cong \prod_{\psi} \mathcal{O}$, where the sum runs over all $\psi \in \Irr_{\Q_{p}^{c}}(\Delta)$.
This decomposition induces an isomorphism
\begin{equation}\label{eq:decomp-quot-unit-groups}
\zeta(\mathcal{Q}^{F}(\mathcal{G}))^{\times} / 
\zeta(\Lambda^{\mathcal{O}}(\mathcal{G}))^{\times} \cong 
\prod_{\psi \in \Irr_{\Q_{p}^{c}}(\Delta)}
\zeta(\mathcal{Q}^{F}(\mathcal{H}))^{\times} / 
\zeta(\Lambda^{\mathcal{O}}(\mathcal{H}))^{\times}. 
\end{equation}
Analogous observations hold for the quotients 
$\zeta(\mathcal{Q}^{F}(\mathcal{H} \times C))^{\times} / \zeta(\Lambda^{\mathcal{O}}(\mathcal{H} \times C))^{\times}$ 
for each $C \in \mathcal{C}(\Delta)$.
Moreover, we have a commutative diagram
\[
\xymatrix{
\zeta(\mathcal{Q}(\mathcal{G}))^{\times} / 
\zeta(\Lambda(\mathcal{G}))^{\times}  \ar[rr]^{\prod \quot \, \circ \, \res \qquad \qquad \quad} \ar[d]  & &
\prod_{C \in \mathcal{C}(\Delta)}
\zeta(\mathcal{Q}(\mathcal{H} \times C))^{\times} / 
\zeta(\Lambda(\mathcal{H} \times C))^{\times} \ar[d] \\
\zeta(\mathcal{Q}^{F}(\mathcal{G}))^{\times} / 
\zeta(\Lambda^{\mathcal{O}}(\mathcal{G}))^{\times}  
\ar[rr]^{\prod \quot \, \circ \, \res \qquad \qquad \quad} & &
\prod_{C \in \mathcal{C}(\Delta)}
\zeta(\mathcal{Q}^{F}(\mathcal{H} \times C))^{\times} / 
\zeta(\Lambda^{\mathcal{O}}(\mathcal{H} \times C))^{\times},
}
\]
where the products run over all $C \in \mathcal{C}(\Delta)$ 
and the vertical extension of scalars maps are injective by 
Lemma \ref{lem:extension-of-scalars-when-denom-units-of-centre-of-Iwasawa-algebra}.
Thus it suffices to show that the bottom horizontal map is injective; we denote this map by $\iota$.
	
Now let $f$ be an arbitrary element in 
$\zeta(\mathcal{Q}^{F}(\mathcal{G}))^{\times} / \zeta(\Lambda^{\mathcal{O}}(\mathcal{G}))^{\times}$.
Using \eqref{eq:decomp-quot-unit-groups} we write 
$f = (f_{\psi})_{\psi \in \Irr_{\Q_{p}^{c}}(\Delta)}$ with
$f_{\psi} \in \zeta(\mathcal{Q}^{F}(\mathcal{H}))^{\times} / \zeta(\Lambda^{\mathcal{O}}(\mathcal{H}))^{\times}$.
Using the definition of $\iota$, we write $\iota(f)$ as $(f_{C})_{C \in \mathcal{C}(\Delta)}$.
Moreover, for each $C \in \mathcal{C}(\Delta)$ we write
$f_{C} = (f_{C, \lambda})_{\lambda \in \Irr_{\Q_{p}^{c}}(C)}$ with
$f_{C, \lambda} \in \zeta(\mathcal{Q}^{F}(\mathcal{H}))^{\times} / \zeta(\Lambda^{\mathcal{O}}(\mathcal{H}))^{\times}$.
If $C = U/N$ for a subgroup $U$ of $\Delta$ and a normal subgroup $N$ of $U$,
then explicitly we have
\[
f_{C, \lambda}
=
\textstyle{\prod_{\psi \in \Irr_{\Q_{p}^{c}}(\Delta)}} \, f_{\psi}^{\langle \psi, \ind_{U}^{\Delta} \infl_{C}^{U} \lambda \rangle},
\]
where $\langle - , - \rangle$ denotes the inner product of characters of $\Delta$. 
(To see this, one can either use the definitions of the quotient and restriction maps given in
\S \ref{subsec:dets-and-nr} or use an obvious variant of \cite[Lemma 2.4]{breuning-thesis}.)

Suppose that $f \in \ker(\iota)$.
Then for each
$C \in \mathcal{C}(\Delta)$, $\lambda \in \Irr_{\Q_{p}^{c}}(C)$ we have that $f_{C, \lambda} = 1 \in \zeta(\mathcal{Q}^{F}(\mathcal{H}))^{\times} /\zeta(\Lambda^{\mathcal{O}}(\mathcal{H}))^{\times}$.
Fix $\psi \in \Irr_{\Q_{p}^{c}}(\Delta)$.
By Brauer's induction theorem \cite[Theorem 15.9]{MR632548} we may write $\psi$ as a finite sum
$\psi = \sum_{j} z_{j} \ind_{U_{j}}^{\Delta} \infl_{C_{j}}^{U_{j}} \lambda_{j}$ where  
$C_{j} = U_{j}/N_{j}$ is a cyclic subquotient of $\Delta$, $\lambda_{j} \in  \Irr_{\Q_{p}^{c}}(C_{j})$
and $z_{j} \in \Z$.
We compute
\begin{eqnarray*}
f_{\psi} & = & 
\textstyle{\prod_{\psi' \in \Irr_{\Q_{p}^{c}}(\Delta)}} \, f_{\psi'}^{\langle \psi', \psi \rangle}
= 
\textstyle{\prod_{\psi' \in \Irr_{\Q_{p}^{c}}(\Delta)}} \, f_{\psi'}^{\langle \psi', \sum_{j} z_{j} \ind_{U_{j}}^{\Delta} \infl_{C_{j}}^{U_{j}} \lambda_{j} \rangle} \\
& = & 
\textstyle{\prod_{j}} \left(\textstyle{\prod_{\psi' \in \Irr_{\Q_{p}^{c}}(\Delta)}} \, f_{\psi'}^{\langle \psi', \ind_{U_{j}}^{\Delta} \infl_{C_{j}}^{U_{j}} \lambda_{j} \rangle}\right)^{z_{j}}
= \textstyle{\prod_{j}} \, f_{C_{j}, \lambda_{j}}^{z_{j}} 
= \textstyle{\prod_{j}} \, 1^{z_{j}} = 1.
\end{eqnarray*}
Thus $f$ is trivial, as desired.
\end{proof}

\begin{prop}\label{prop:F-elementary-reduction}
Let $p$ be an odd prime and let $\mathcal{G}$ be an admissible one-dimensional $p$-adic Lie group. 
Let $\Gamma_{0} \simeq \Z_{p}$ be an open central subgroup of $\mathcal{G}$,
let $G = \mathcal{G} / \Gamma_{0}$ and let $\pi : \mathcal{G} \rightarrow G$ denote the canonical projection.
Let $F/\Q_{p}$ be a finite extension with ring of integers $\mathcal{O}$.
Let
\[
\mathcal{E}_{F}(G) = \{ \pi^{-1}(E) \mid E \textrm{ is an $F$-elementary subgroup of }G \}.
\]
Then both products of maps 
\[
\zeta(\mathcal{Q}^{F}(\mathcal{G}))^{\times} / 
\nr(K_{1}(\Lambda^{\mathcal{O}}(\mathcal{G})))
\xrightarrow{\prod \res^{\mathcal{G}}_{\mathcal{H}}}
\prod_{\mathcal{H} \in \mathcal{E}_{F}(G)}
\zeta(\mathcal{Q}^{F}(\mathcal{H}))^{\times} / 
\nr(K_{1}(\Lambda^{\mathcal{O}}(\mathcal{H})))
\]
and
\[
K_{0}(\Lambda^{\mathcal{O}}(\mathcal{G}),\mathcal{Q}^{F}(\mathcal{G})) 
\xrightarrow{\prod \res^{\mathcal{G}}_{\mathcal{H}}}
\prod_{\mathcal{H} \in \mathcal{E}_{F}(G)}
K_{0}(\Lambda^{\mathcal{O}}(\mathcal{H}),\mathcal{Q}^{F}(\mathcal{H})) 
\]
are injective.
\end{prop}

\begin{proof}
For an open subgroup $\mathcal{U}$ of $\mathcal{G}$, we denote by $R_{F}(\mathcal{U})$ the ring of all characters of finite dimensional $F$-representations of $\mathcal{U}$ with open kernel.
We view $R_{F}$ as a Frobenius functor of the open subgroups of $\mathcal{G}$
in the sense of \cite[\S 38A]{MR892316} (note that though the definitions of loc.\ cit.\ are only stated for finite groups, they easily extend to the present setting).
By \cite[Lemma 7]{MR2205173} the groups $K_{1}(\Lambda(-))$
and $K_{1}(\mathcal{Q}(-))$ are
Frobenius modules over the Frobenius
functor $\mathcal{U} \mapsto R_{\Q_p}(\mathcal{U})$ (note that the result for $K_{1}(\mathcal{Q}(-))$
is not explicitly stated, but the same proof works, and it is
actually used in the subsequent corollary).
The same argument shows
that $K_{1}(\Lambda^{\mathcal{O}}(-))$
and $K_{1}(\mathcal{Q}^{F}(-))$ are
Frobenius modules over the Frobenius
functor $\mathcal{U} \mapsto R_{F}(\mathcal{U})$.
The canonical map $K_{1}(\Lambda^{\mathcal{O}}(-)) \rightarrow
K_{1}(\mathcal{Q}^{F}(-))$ is a morphism of Frobenius modules
and thus its cokernel $K_{0}(\Lambda^{\mathcal{O}}(-),\mathcal{Q}^{F}(-))$
is also a Frobenius module over the Frobenius functor $\mathcal{U} \mapsto R_{F}(\mathcal{U})$. 

Now let $\overline{\Gamma} = \mathcal{G}/H$ where $H$ is as in \S \ref{subsec:Iwasawa-algebras}.
By \cite[Lemma 7]{MR2205173}
\[
\Det: K_{1}(\Lambda(-)) \longrightarrow \Hom^{\ast}_{G_{\Q_{p}}}(R_{p}(-), \mathcal{Q}^{c}(\overline{\Gamma})^{\times})
\]
is a morphism of 
Frobenius modules  and the same
argument works with $K_{1}(\Lambda(-))$ and $G_{\Q_{p}}$ 
replaced by $K_{1}(\Lambda^{\mathcal{O}}(-))$ and $G_{F}$, respectively.
Together with \eqref{eqn:Det_triangle}, this shows that
$
\nr: K_{1}(\Lambda^{\mathcal{O}}(-)) \rightarrow \zeta(\mathcal{Q}^{F}(-))^{\times}
$
is also a morphism of Frobenius modules.
Hence the cokernel
$\zeta(\mathcal{Q}^{F}(-))^{\times} / \nr(K_{1}(\Lambda^{\mathcal{O}}(-)))$ is a
Frobenius module over the Frobenius
functor $\mathcal{U} \mapsto R_{F}(\mathcal{U})$.

We now proceed as in the proof of \cite[Corollary, p.\ 167]{MR2205173}.
Let $\mathds{1}_{G}$ and $\mathds{1}_{\mathcal{G}}$ denote the trivial characters of $G$ and $\mathcal{G}$, respectively.
An application of the Witt--Bermann induction theorem \cite[Theorem 21.6]{MR632548} 
to the finite group $G$ shows that there are $F$\nobreakdash-elementary subgroups $H_{i} \leq G$
and $\lambda_{i} \in R_{\Q_{p}^{c}}(H_{i})$ such that 
$\mathds{1}_{G} = \sum_{i} \ind_{H_{i}}^{G} \lambda_{i}$.
Let $\mathcal{H}_{i} \leq \mathcal{G}$ denote the full preimage of $H_{i}$ and let $\xi_{i} = \infl_{H_{i}}^{\mathcal{H}_{i}} \lambda_{i}$. Then lifting gives 
$\mathds{1}_{\mathcal{G}} = \sum_{i} \ind_{\mathcal{H}_{i}}^{\mathcal{G}} \xi_{i}$ (finite sum).

Now let $x$ be either in $\zeta(\mathcal{Q}^{F}(\mathcal{G}))^{\times} / 
\nr(K_{1}(\Lambda^{\mathcal{O}}(\mathcal{G})))$ or in
$K_{0}(\Lambda^{\mathcal{O}}(\mathcal{G}),\mathcal{Q}^{F}(\mathcal{G}))$
and denote the trivial element of both of these groups by $0$.
Suppose that $x \in \ker(\prod_{i} \res^{\mathcal{G}}_{\mathcal{H}_{i}})$. 
Then by the defining properties of Frobenius modules over Frobenius functors we have
\[
\textstyle{x 
= \mathds{1}_{\mathcal{G}} \cdot x
= \sum_{i} (\ind_{\mathcal{H}_{i}}^{\mathcal{G}} \xi_{i}) \cdot x
= \sum_{i} \ind_{\mathcal{H}_{i}}^{\mathcal{G}} (\xi_{i} \cdot  \res^{\mathcal{G}}_{\mathcal{H}_{i}} x) = 0.}
\]
Hence the result now follows by the trivial observation that $\mathcal{H}_{i} \in \mathcal{E}_{F}(G)$
for each $i$.
\end{proof}

\begin{corollary}\label{cor:reduction-to-elementary-subgroups}
Let $p$ be an odd prime and let $\mathcal{G}$ be an admissible one-dimensional $p$-adic Lie group. 
Let $\Gamma_{0} \simeq \Z_{p}$ be an open central subgroup of $\mathcal{G}$,
let $G = \mathcal{G} / \Gamma_{0}$ and let
$\mathcal{E}(G) = \{ \pi^{-1}(E) \mid E \textrm{ is an elementary subgroup of }G \}$,
where $\pi : \mathcal{G} \rightarrow G$ is the canonical projection.
Then the product of maps
\[
\zeta(\mathcal{Q}(\mathcal{G}))^{\times} / \nr(K_{1}(\Lambda(\mathcal{G})))
\xrightarrow{\prod \res^{\mathcal{G}}_{\mathcal{H}}}
\prod_{\mathcal{H} \in \mathcal{E}(G)}
\zeta(\mathcal{Q}(\mathcal{H}))^{\times} / \nr(K_1(\Lambda(\mathcal{H})))
\]
is injective.
If we further assume that $SK_{1}(\mathcal{Q}(\mathcal{G}))=0$, then the product of maps 
\[
K_{0}(\Lambda(\mathcal{G}),\mathcal{Q}(\mathcal{G})) 
\xrightarrow{\prod \res^{\mathcal{G}}_{\mathcal{H}}}
\prod_{\mathcal{H} \in \mathcal{E}(G)}
K_{0}(\Lambda(\mathcal{H}),\mathcal{Q}(\mathcal{H})) 
\]
is also injective. 
\end{corollary}

\begin{proof}
Write $\mathcal{G} = H \rtimes \Gamma$ where $H$ is finite and $\Gamma \simeq \Z_{p}$.
Write $|H|=p^{t}k$ for integers $t$ and $k$ such that $t \geq 0$ and $p \nmid k$.
Then $F:=\Q_{p}(\zeta_{pk})$ is a finite tamely ramified extension of $\Q_{p}$.
We now repeat an argument given in the proof of \cite[Proposition 9]{MR1687551}
to show that every $F$-elementary subgroup of any finite quotient of $\mathcal{G}$
is in fact elementary.
Let $q$ be a prime and let $C_{n} \rtimes Q$ be an $F$-$q$-elementary finite quotient of $\mathcal{G}$.
Write $n = p^{s} m$ for integers $s$ and $m$ such that $s \geq 0$ and $p \nmid m$.
Note that $m$ must divide $k$.
Since both $\zeta_{p}$ and $\zeta_{m}$ lie in $F$,
the Galois group $\Gal(F(\zeta_{n})/F)$ has $p$-power order.
Thus if $p \neq q$ then any homomorphism $Q \rightarrow \Gal(F(\zeta_{n})/F)$ must be trivial.
If $p=q$ then $s=0$ and so the extension $F(\zeta_{n})/F$ is trivial, giving the same result.

Now Proposition \ref{prop:F-elementary-reduction} and Lemma \ref{lem:extension-of-scalars-injective-for-tame-extensions} imply the first claim. The second claim follows from Lemma \ref{lem:hom-rel-K0-nr}.
\end{proof}

\begin{proof}[Proof of Theorem \ref{thm:reduction-to-p-elementary-subquotients}]
By Corollary \ref{cor:reduction-to-elementary-subgroups} we can and do replace $\mathcal{G}$
by an elementary subgroup $\mathcal{H}_{1}$ containing $\Gamma_{0}$. 
We need only consider the case that $\mathcal{H}_{1}$ is $q$-elementary for some prime
$q \neq p$.
By Lemma \ref{lem:elementary-compatible-RW} we have 
$\mathcal{H}_{1} \simeq \Gamma \times C \times Q$, where $\Gamma \simeq \Z_{p}$, 
$C$ is finite cyclic of order coprime to $q$, and $Q$ is a finite $q$-group.
Moreover, we can and do choose $\Gamma$ such that $\Gamma_{0} \leq \Gamma$.
Hence we may apply Proposition \ref{prop:injectivity-K-groups-direct-prod-delta} with $\mathcal{G}=\mathcal{H}_{1}$, 
$\mathcal{H} = \Gamma \times C$ and $\Delta=Q$.
It only remains to observe that for all cyclic subquotients $H$ of $\Delta$, 
the finite groups $C \times H$ are cyclic and hence the groups $\mathcal{H} \times H = \Gamma \times C \times H$ are $p$-elementary (in fact, they are $\ell$-elementary for every prime $\ell$).
\end{proof}

\subsection{Application to the EIMC}\label{subsec:application-to-EIMC}

We give an easy reformulation of the EIMC without its uniqueness statement.

\begin{lemma}\label{lem:EIMC-reformulation}
Let $p$ be an odd prime and let $\mathcal{L}/K$ be an admissible
one-dimensional $p$-adic Lie extension of a totally real number field $K$.
Let $\mathcal{G}=\Gal(\mathcal{L}/K)$ and let $S$ be a finite set of places of $K$ containing
$S_{\ram}(\mathcal{L}/K) \cup S_{\infty}$.
Choose any $\zeta_{S} \in K_{1}(\mathcal{Q}(\mathcal{G}))$ such that 
$\partial(\zeta_{S}) = - [C_{S}^{\bullet}(\mathcal{L}/K)]$.
Then the EIMC holds for $\mathcal{L}/K$ if and only if
\[
\nr(\zeta_{S}) \equiv \Phi_{S} \bmod \nr(K_{1}(\Lambda(\mathcal{G}))).
\]
\end{lemma}

\begin{proof}
This is an easy consequence of Lemma \ref{lem:hom-rel-K0-nr} in the case $F = \Q_{p}$. 
\end{proof}

We are now ready to prove the main result of this section, a special case of which will allow us to deduce Corollary \ref{cor:EIMC-abelian-Sylow-p} from Theorem \ref{thm:EIMC-abelian-exts}.

\begin{theorem}\label{thm:EIMC-reduction}
Let $p$ be an odd prime and let $\mathcal{L}/K$ be an admissible
one-dimensional $p$-adic Lie extension of a totally real number field $K$.
Let $\mathcal{I}_{p}$ be the collection of all intermediate admissible extensions
with $p$-elementary Galois group.
Let $\mathcal{G}=\Gal(\mathcal{L}/K)$ and let $\mathcal{F}_{p}$ be the collection of all intermediate
extensions defined by the collection $\mathcal{E}_{p}$ of subquotients of $\mathcal{G}$ defined in Theorem \ref{thm:reduction-to-p-elementary-subquotients}.
The following statements are equivalent.
\begin{enumerate}
\item The EIMC holds for $\mathcal{L}/K$.
\item The EIMC holds for all subextensions in $\mathcal{I}_{p}$.
\item The EIMC holds for all subextensions in $\mathcal{F}_{p}$.
\end{enumerate}
\end{theorem}

\begin{remark}
Note that $\mathcal{F}_{p} \subseteq \mathcal{I}_{p}$ and $\mathcal{I}_{p}$ is infinite, but $\mathcal{F}_{p}$ is finite (see Remark \ref{rmk:prod-is-finite}).
\end{remark}

\begin{proof}[Proof of Theorem \ref{thm:EIMC-reduction}]
Lemma \ref{lem:EIMC-implies-EIMC-for-subextensions} shows that
(i) $\Rightarrow$ (ii). Since
$\mathcal{F}_{p} \subseteq \mathcal{I}_{p}$ it follows trivially that (ii) $\Rightarrow$ (iii).
Together, Propositions \ref{prop:complex-res-quot} and \ref{prop:phi-S-res-quot}, 
Lemma \ref{lem:hom-rel-K0-nr} in the case $F=\Q_{p}$,
Lemma \ref{lem:EIMC-reformulation} and
Theorem \ref{thm:reduction-to-p-elementary-subquotients}
show that (iii) $\Rightarrow$ (i).
\end{proof}

\begin{remark} \label{rem:reduction-mu}
If the extension $\mathcal{L}/K$ satisfies the $\mu = 0$ hypothesis, 
then \cite[Theorem~A]{MR2205173} shows that the equivalence of statements (i) and (ii) in 
Theorem \ref{thm:EIMC-reduction} recovers \cite[Theorem~C]{MR2205173}
(which itself assumes the $\mu = 0$ hypothesis).
\end{remark}

\begin{corollary}[Corollary \ref{cor:EIMC-abelian-Sylow-p}]\label{cor:EIMC-abelian-Sylow-p-proof}
Let $p$ be an odd prime and let $K$ be a totally real number field. 
Let $\mathcal{L}/K$ be an admissible one-dimensional $p$-adic Lie extension such that 
$\Gal(\mathcal{L}/K)$ has an abelian Sylow $p$-subgroup.
Then the EIMC with uniqueness holds for $\mathcal{L}/K$.
\end{corollary}

\begin{proof}
Let $\mathcal{G}=\Gal(\mathcal{L}/K)$.
Since $\mathcal{G}$ has an abelian Sylow $p$-subgroup, every $p$-elementary subquotient of 
$\mathcal{G}$ is abelian.
Hence the EIMC for $\mathcal{L}/K$ holds by
Theorem \ref{thm:EIMC-abelian-exts} and 
the equivalence of statements (i) and (ii) in Theorem \ref{thm:EIMC-reduction}. 
Moreover, $SK_{1}(\mathcal{Q}(\mathcal{G}))=0$ by Corollary \ref{cor:SK1-vanishes-when-Sylow-p-subgroup-is-abelian} and so we also have uniqueness (see Remark \ref{rmk:SK1}).
\end{proof}

\begin{remark}\label{rem:hybrid}
One may ask whether it is possible to deduce the EIMC for further non-abelian 
extensions by considering `hybrid' cases as in \cite{MR3749195}.
This is in fact not the case.
To see this, assume that the Iwasawa algebra $\Lambda(\mathcal{G})$
is `$N$-hybrid' for a finite normal subgroup $N$ of $\mathcal{G}$ 
in the sense of \cite[Definition 3.8]{MR3749195}, that is, 
(i) $p$ does not divide $|N|$ and
(ii) $\Lambda(\mathcal{G})(1-e_{N})$ 
is a maximal $R$-order, where $e_{N}$ is the central idempotent $|N|^{-1}\sum_{n \in N} n$.
Then (i) implies that each Sylow $p$-subgroup of $\mathcal{G}$ is mapped isomorphically
onto a Sylow $p$-subgroup of $\mathcal{G}/N$ under the canonical quotient map
$\mathcal{G} \rightarrow \mathcal{G}/N$.
Thus if $\mathcal{G}/N$ has an abelian Sylow $p$-subgroup then so does $\mathcal{G}$.
\end{remark}

\section{The ETNC and the Coates--Sinnott conjecture}\label{sec:ETNC-and-CS}

\subsection{Further results on Fitting ideals}\label{sec:further-Fitting-results}

We give two further results on Fitting ideals that we shall use in \S \ref{subsec:SCS}.

\begin{lemma}\label{lem:Fitting-four-term}
Let $p$ be a prime, let $G$ be a finite abelian group, and 
let $x \mapsto x^{\#}$ denote the involution on $\Z_{p}[G]$
induced by $g \mapsto g^{-1}$ for $g \in G$.
Let $e \in \Z_{p}[G]$ be an idempotent such that $e = e^{\#}$.
Let
\[
0 \longrightarrow M \longrightarrow C \longrightarrow C' \longrightarrow M' \longrightarrow 0
\]
be an exact sequence of finite $e \Z_p[G]$-modules and assume that
$C$ and $C'$ are of finite projective dimension.
Then we have an equality
\[
\Fitt_{e \Z_{p}[G]}(M^{\vee})^{\#} \cdot \Fitt_{e \Z_{p}[G]}(C') = 
\Fitt_{e \Z_{p}[G]}(C) \cdot \Fitt_{e \Z_p[G]}(M').
\]
\end{lemma}

\begin{proof}
This is a straightforward consequence  of \cite[Lemma 5]{MR2046598}.
See also \cite[Proposition 5.3]{MR2609173}.
\end{proof}

Let $p$ be a prime and let $\mathcal{G}$ be an admissible one-dimensional $p$-adic Lie group.
Let $\Gamma_{0}$ be an open subgroup of $\Gamma$ that is central in $\mathcal{G}$ and
let $R=\Z_{p}\llbracket\Gamma_{0}\rrbracket$.
Then $\Lambda(\mathcal{G})$ is an $R$-order in the separable $Quot(R)$-algebra 
$\mathcal{Q}(\mathcal{G})$.
Now let $e$ be any central idempotent element of $\Lambda(\mathcal{G})$ and define 
$\Lambda := e\Lambda(\mathcal{G})$ and $\mathcal{Q} := e\mathcal{Q}(\mathcal{G})$. 
For each (left) $\Lambda$-module $M$ we set
$E^{1}(M) := \Ext^{1}_{R}(M, R)$, which has a canonical 
right $\Lambda$-module structure.
Let $x \mapsto x^{\#}$ denote the anti-involution on $\Lambda(\mathcal{G})$
induced by $g \mapsto g^{-1}$ for $g \in \mathcal{G}$.
Set $\Lambda^{\#} := 
{\{\lambda^{\#} \mid \lambda \in \Lambda \}} = e^{\#} \Lambda(\mathcal{G})$
and likewise $\mathcal{Q}^{\#} = e^{\#}\mathcal{Q}(\mathcal{G})$.
Then $E^{1}(M)$ is a left $\Lambda^{\#}$-module, as
$\lambda^{\#} \in \Lambda^{\#}$ acts on $f \in E^{1}(M)$ by
$\lambda^{\#} f = f \lambda$.
For each $C^{\bullet} \in \mathcal{D}^{\perf}(\Lambda)$ we write
$(C^{\bullet})^{\ast}$ for the dual complex
$R\Hom_{R}(C^{\bullet}, R)$ in $\mathcal{D}^{\perf}(\Lambda^{\#})$.
Since $\Hom_R(-,R)$ is exact on finitely generated projective $\Lambda$-modules,
this induces a homomorphism of abelian groups 
\[
	(-)^{\ast}: K_{0}(\Lambda,\mathcal{Q}) \rightarrow K_{0}(\Lambda^{\#},\mathcal{Q}^{\#}).
\]

\begin{lemma}\label{lem:Fitt-of-E^1(M)} 
If $\mathcal{G}$ is abelian then $\Lambda$ is commutative and the following statements hold.
\begin{enumerate}
\item 
Let $M$ be a finitely generated $\Lambda$-module that is of projective dimension at most one
and that is also $R$-torsion.
Then $E^{1}(M)$ is a finitely generated $\Lambda^{\#}$-module
of projective dimension at most one and is $R$-torsion.
Moreover, we have
\[
\Fitt_{\Lambda^{\#}}(E^{1}(M)) = \Fitt_{\Lambda}(M)^{\#}.
\]
\item
For each $C^{\bullet} \in \mathcal{D}^{\perf}\tor(\Lambda)$ we have
\[
\Fitt_{\Lambda^{\#}}((C^{\bullet})^{\ast}[-1]) = \Fitt_{\Lambda}(C^{\bullet})^{\#}.
\]
\end{enumerate}
\end{lemma}

\begin{proof}
By Lemma \ref{lemma:fitting-ideal-is-principal} we may choose a quadratic presentation of $M$ as in \eqref{eq:quad-pres}.
We apply the functor $\Hom_{R}(-, R)$ to this sequence.
Since $M$ is $R$-torsion and $\Lambda$ is a projective $R$-module, we have $\Hom_{R}(M,R) = E^{1}(\Lambda) = 0$.
We identify $\Hom_{R}(\Lambda,R)$ and $\Lambda^{\#}$ so that we obtain an exact sequence
\[
0 \longrightarrow (\Lambda^{\#})^{n} \xrightarrow{h^{T, \#}} (\Lambda^{\#})^{n} \longrightarrow E^{1}(M) \longrightarrow 0,
\]
where the second map is obtained from $h$ by applying the involution $^{\#}$ to its transpose. This proves (i).
Let $x \in K_1(\mathcal{Q})$ be arbitrary. We will show that $\partial(x)^{\ast} = -\partial(x^{T, \#})$. 
Since the connecting homomorphism $\partial$ in \eqref{eqn:K-theory-SES-comm-Iwasawa-algebra} is surjective, this implies (ii). 
Each $x\in K_1(\mathcal{Q})$ can we written as the class of $h g^{-1}$, where both $h$ and $g$ are matrices
in $M_n(\Lambda) \cap \GL_n(\mathcal{Q})$ for some $n$. 
Since $(-)^{T,\#}$ is multiplicative and $(-)^{\ast}$ is a homomorphism, we may therefore assume that $x$ is represented by $h$.
The $\Lambda$-module $M := \cok(h)$ is of projective dimension at most one and $R$-torsion.
Hence (ii) follows from (i) once we observe that $M^{\ast} \simeq E^1(M)[-1]$ in $\mathcal{D}(\Lambda^{\#})$.
\end{proof}

\subsection{The ETNC at negative integers}\label{subsec:ETNC}

The equivariant Tamagawa number conjecture (ETNC) has been
formulated by Burns and Flach \cite{MR1884523} in vast generality.
We will only consider the case of Tate motives. 
Let $L/K$ be a finite Galois extension of number fields, let $G=\Gal(L/K)$ and let $r \in \Z$.
We regard $h^{0}(\Spec(L))(r)$ as a motive defined over $K$ and with
coefficients in the semisimple algebra $\Q[G]$.
The ETNC for the pair $(h^{0}(\Spec(L))(r), \Z[G])$ simply asserts that
a certain canonical element $T\Omega(L/K, \Z[G], r) \in K_0(\Z[G], \R[G])$
vanishes.

Now we assume that $L/K$ is a CM extension and let $j \in G$
denote complex conjugation.
For each $r \in \Z$ we define a central idempotent 
$e_r := \frac{1- (-1)^{r} j}{2}$ in $\Z[\half][G]$. 
The ETNC for the pair $(h^0(\Spec(L))(r), e_{r}\Z[\half][G])$ then likewise
asserts that a certain canonical element 
$T\Omega(L/K, e_{r} \Z[\half][G], r)$ in $K_{0}(e_{r} \Z[\half][G], \R[G])$ vanishes. This corresponds to the plus
or minus part of the ETNC (away from $2$) if $r$ is odd or even, respectively.

If $r$ is a negative integer, then a result of Siegel
\cite{MR0285488} implies that $T\Omega(L/K, e_r \Z[\half][G], r)$ actually
belongs to the subgroup
\[
K_{0}(e_{r} \Z[\half][G], \Q[G]) \cong \bigoplus_{p\, \mathrm{ odd}} K_{0}(e_{r} \Z_{p}[G], \Q_{p}[G]).
\]
We say that the $p$-part of the ETNC for the pair 
$(h^0(\Spec(L))(r), e_r\Z[\half][G])$ holds if its image in
$K_{0}(e_{r} \Z_{p}[G], \Q_{p}[G])$ vanishes.

\begin{theorem}\label{thm:known-results-ETNC}
Let $p$ be an odd prime.
Let $L/K$ be a finite Galois CM extension of number fields and let $G=\Gal(L/K)$.
Then the following hold for every negative integer~$r$.
\begin{enumerate}
\item 
The element $T\Omega(L/K, e_{r} \Z[\half][G], r)$ belongs to
$K_{0}(e_{r} \Z[\half][G], \Q[G])_{\tors}$.
\item 
Assume that if $p$ divides $|G|$ then $L(\zeta_p)_{\infty}^+ / K$ satisfies the $\mu = 0$ hypothesis.
Then the $p$-part of the ETNC for the pair $(h^0(\Spec(L))(r), e_r\Z[\half][G])$ holds.	
\end{enumerate} 
\end{theorem}

\begin{proof}
Part (ii) has been shown by Burns \cite[Corollary 2.10]{MR3294653}.
If the extension $L(\zeta_p)_{\infty}^+ / K$ 
satisfies the $\mu = 0$ hypothesis (whether or not $p$ divides $|G|$)
there is an independent proof due to the second author \cite[Corollary 5.11]{MR3072281}.
By a general induction argument \cite[Proposition 6.1 (iii)]{MR2801311}
(ii) implies (i) (if $r$ is odd see also \cite[Corollary 6.2]{MR2801311}).
\end{proof}

In the case that $G$ has an abelian Sylow $p$-subgroup, we now remove the $\mu=0$
hypothesis from Theorem \ref{thm:known-results-ETNC} (ii) and thus obtain 
Theorem \ref{thm:ETNC-at-negative-integers-intro} from the
introduction.

We first introduce some more notation.
If $v$ is a finite place of $K$, we denote the residue field of $K$ at $v$ by $K(v)$.
If $R$ is either $K(v)$ or $\mathcal{O}_{K,S}$ for a finite set $S$ of places of $K$
that contains $S_{\infty}$ and $\mathcal{F}$ is an \'etale (pro-)sheaf on
$\Spec(R)$, then we abbreviate the complex $R\Gamma_{\et}(\Spec(R), \mathcal{F})$
and in each degree $i$ the cohomology group $H^{i}_{\et}(\Spec(R), \mathcal{F})$
to $R\Gamma(R, \mathcal{F})$ and $H^{i}(R, \mathcal{F})$, respectively.

\begin{theorem}[Theorem \ref{thm:ETNC-at-negative-integers-intro}]\label{thm:ETNC-at-negative-integers}
Let $p$ be an odd prime.
Let $L/K$ be a finite Galois CM extension of number fields and let $G=\Gal(L/K)$.
Suppose that $G$ has an abelian Sylow $p$-subgroup.
Then for each negative integer $r$ the $p$-part of the ETNC for the pair
$(h^{0}(\Spec(L))(r), e_{r}\Z[\half][G])$ holds.
\end{theorem}

\begin{proof}
Let $S$ and $T$ be two finite non-empty sets of places of $K$
such that $S$ contains
$S_{p} \cup S_{\ram}(L/K) \cup S_{\infty}$ and $S \cap T = \emptyset$.
We define a complex of $e_{r} \Z_{p}[G]$-modules
\begin{gather*}
R\Gamma_{T}(\mathcal{O}_{K,S}, e_{r} \Z_{p}[G]^{\#}(1-r)) := \\
\cone\Big(R\Gamma(\mathcal{O}_{K,S}, e_{r} \Z_p[G]^{\#}(1-r))
\longrightarrow \bigoplus_{v \in T}
R\Gamma(K(v), e_{r} \Z_{p}[G]^{\#}(1-r))\Big)[-1].
\end{gather*}
By \cite[Theorem 5.10]{MR3072281} this complex is acyclic outside
degree $2$ and the only non-vanishing cohomology
group, which we denote by $H^{2}_{T}(\mathcal{O}_{K,S}, e_r \Z_p[G]^{\#}(1-r))$,
is cohomologically trivial.
Moreover, the ETNC for the pair $(h^{0}(\Spec(L))(r), e_{r}\Z[\half][G])$
holds if and only if $\Theta_{S,T}(r)$ is a generator of the
(non-commutative) Fitting invariant of this $e_{r} \Z_{p}[G]$-module.
	
We now can either work with non-commutative Fitting invariants or we can apply \cite[Proposition 9]{MR1687551} in combination with
Theorem \ref{thm:known-results-ETNC} (i) to reduce to abelian extensions. 
We choose the latter option so that the result follows from Theorem \ref{thm:strong-Coates-Sinnott} below.
\end{proof}

\subsection{The strong Coates--Sinnott conjecture}\label{subsec:SCS}

The following result is a strengthening of the `strong Coates--Sinnott conjecture'
\cite[Conjecture 5.1]{MR3072281} in the case of abelian CM extensions.

\begin{theorem}\label{thm:strong-Coates-Sinnott}
Let $p$ be an odd prime.
Let $L/K$ be a finite abelian CM extension of number fields and let $G=\Gal(L/K)$.
Let $S$ and $T$ be two finite non-empty sets of places of $K$ such that $S$ contains
$S_{p} \cup S_{\ram}(L/K) \cup S_{\infty}$ and $S \cap T = \emptyset$.
Then for each negative integer $r$ we have
\[
\Fitt_{e_{r} \Z_p[G]}(H^2_T(\mathcal{O}_{K,S}, e_r \Z_p[G]^{\#}(1-r))) = 
\Theta_{S,T}(r) e_r \Z_{p}[G].
\]
\end{theorem}

\begin{proof}
We first observe that it suffices to show that $\Theta_{S,T}(r)$ is contained
in the Fitting ideal by \cite[Theorem 5.10]{MR3072281}.
Hence we can and do assume that $\zeta_{p} \in L$
by \cite[Proposition 5.5]{MR3072281}. 
Let $L_{\infty}$ and $K_{\infty}$ be the cyclotomic $\Z_p$-extensions of $L$ and $K$, respectively. 
Let $\mathcal{G} := \Gal(L_{\infty}/K)$. 
Then $\mathcal{G} = H \times \Gamma$ where 
$H = \Gal(L_{\infty}/K_{\infty})$ and $\Gamma \simeq \Z_p$.
Moreover, we have that $\Lambda(\mathcal{G}) = R[H]$ where $R := \Z_{p} \llbracket \Gamma \rrbracket$.

For each integer $n$, we now define a complex of $e_{n} \Lambda(\mathcal{G})$-modules
\begin{gather} 
R\Gamma_T(\mathcal{O}_{K,S}, e_n \Lambda(\mathcal{G})^{\#}(1-n))
:= \\
\cone\Big(R\Gamma(\mathcal{O}_{K,S}, e_n \Lambda(\mathcal{G})^{\#}(1-n))
 \label{eqn:definition-as-cone} \longrightarrow \bigoplus_{v \in T}
R\Gamma(K(v), e_n \Lambda(\mathcal{G})^{\#}(1-n))\Big)[-1]. \nonumber
\end{gather}
In the case $n=0$ this complex has been studied by Burns \cite[\S 5.3.1]{MR4092926}. 
It is acyclic outside degree $2$ and the second cohomology module is of projective dimension at most one
by \cite[Proposition 5.5]{MR4092926}.
We claim that for $v \in T$ the complexes $R\Gamma(K(v), e_n \Lambda(\mathcal{G})^{\#}(1-n)))$ 
are acyclic outside
degree $1$ and we have $e_{n} \Lambda(\mathcal{G})$-isomorphisms
\[
H^{1}(K(v), e_{n} \Lambda(\mathcal{G})^{\#}(1-n)) \simeq 
e_{n} \ind_{\mathcal{G}_{w_{\infty}}}^{\mathcal{G}} \Z_{p}(1-n).	
\]
Since $L_{\infty}$ contains all $p$-power roots of unity,
taking cohomology commutes with Tate twists, so it suffices to show this
for $n=0$. By Shapiro's lemma, we have isomorphisms
\[
H^{i}(K(v), e_n \Z_p[G]^{\#}(1)) \simeq e_n \ind_{G_w}^G H^{i}(L(w), \Z_p(1))
\]
for all $i \in \Z$, where $w$ denotes a place of $L$ above $v$.
It is well known that, since $L(w)$ is a finite field of characteristic not equal to $p$, 
the group $H^{i}(L(w), \Z_p(1))$ vanishes unless $i=1$ and that $H^{1}(L(w), \Z_p(1))$
identifies with $\Z_p \otimes_{\Z} L(w)^{\times}$.
The claim follows by taking inverse limits along the cyclotomic $\Z_p$-extension of $L$.

We have exact sequences of $\Lambda(\mathcal{G}_{w_{\infty}})$-modules
\[
	0 \longrightarrow \Lambda(\mathcal{G}_{w_{\infty}}) \longrightarrow \Lambda(\mathcal{G}_{w_{\infty}}) \longrightarrow \Z_{p}(1-n) \longrightarrow 0,
\]
where the injection is right multiplication by $1 - \chi_{\mathrm{cyc}}(\sigma_{w_{\infty}})^{n-1} \sigma_{w_{\infty}}$.
Therefore we have
\begin{equation} \label{eqn:Fitt-of-local-complex}
	\Fitt_{e_n \Lambda(\mathcal{G})}(H^1(K(v), e_n \Lambda(\mathcal{G})^{\#}(1-n))) = 
	t_{\mathrm{cyc}}^n(\xi_v^{\#})e_n \Lambda(\mathcal{G}).
\end{equation}
Recall that $(C^{\bullet})^{\ast}$ denotes the complex 
$R\Hom_{R}(C^{\bullet}, R)$.
By \eqref{eq:AV-duality} and Artin--Verdier duality (see \cite[Theorem 8.5.6]{MR2333680} or 
\cite[Theorem 4.5.1]{MR3084561}) we have an isomorphism
\[
R\Gamma(\mathcal{O}_{K,S}, e_1 \Lambda(\mathcal{G})^{\#}) \simeq
(C_{S}^{\bullet}(L_{\infty}/K))^{\ast}[-3].
\]
Hence
\[
	\Fitt_{e_1 \Lambda(\mathcal{G})}(R\Gamma(\mathcal{O}_{K,S}, e_1 \Lambda(\mathcal{G})^{\#})[1]) = 
	\Fitt_{e_1 \Lambda(\mathcal{G})}^{-1}(C_{S}^{\bullet}(L_{\infty}/K))^{\#}
	= \Phi_{S}^{\#} e_1 \Lambda(\mathcal{G}),
\]
where the two equalities follow from Lemma \ref{lem:Fitt-of-E^1(M)}
and Theorem \ref{thm:EIMC-abelian-exts}, respectively.
%
%
%
Taking the $(1-n)$-fold Tate twist, we obtain
\begin{equation} \label{eqn:Fitt-of-global-complex}
	\Fitt_{e_n \Lambda(\mathcal{G})}(R\Gamma(\mathcal{O}_{K,S}, e_n \Lambda(\mathcal{G})^{\#}(1-n))[1]) =
	t_{\mathrm{cyc}}^{n-1}(\Phi_{S}^{\#}) e_n \Lambda(\mathcal{G}).
\end{equation}
It follows from \eqref{eqn:fitt-of-sum-of-complexes-in-rel-K-zero}, \eqref{eqn:definition-as-cone}, \eqref{eqn:Fitt-of-local-complex}
and \eqref{eqn:Fitt-of-global-complex} that we have
\begin{eqnarray} 
		\Fitt_{e_n \Lambda(\mathcal{G})}(H^2_T(\mathcal{O}_{K,S}, e_n \Lambda(\mathcal{G})^{\#}(1-n))) & = & 
		t_{\mathrm{cyc}}^{n-1}(\Phi_{S}^{\#}) \prod_{v \in T}
		 t_{\mathrm{cyc}}^n(\xi_v^{\#}) e_n \Lambda(\mathcal{G}) \label{eqn:Fitt-of-H2T} \\
		 & = &	t_{\mathrm{cyc}}^n(\Psi_{S,T}^{\#}) e_n \Lambda(\mathcal{G}).
		 \nonumber
\end{eqnarray}

We now specialise to the case $n=r$.
By \cite[Proposition 1.6.5]{MR2276851}
we have canonical isomorphisms in $\mathcal{D}(\Z_p[G])$ of the form
\[
	\Z_p[G] \otimes^{\mathbb L}_{\Lambda(\mathcal{G})}
	R\Gamma(\mathcal{O}_{K,S}, \Lambda(\mathcal{G})^{\#}(1-r))
	\simeq 
	R\Gamma(\mathcal{O}_{K,S}, \Z_p[G]^{\#}(1-r)),
\]
and
\[
	\Z_p[G] \otimes^{\mathbb L}_{\Lambda(\mathcal{G})}
	R\Gamma(K(v), \Lambda(\mathcal{G})^{\#}(1-r))
	\simeq 
	R\Gamma(K(v), \Z_p[G]^{\#}(1-r)),
\]
for each $v \in T$.
Hence we have a canonical isomorphism in $\mathcal{D}(e_r \Z_p[G])$ of the form
\[
	e_r\Z_p[G] \otimes^{\mathbb L}_{e_r \Lambda(\mathcal{G})}
	R\Gamma_T(\mathcal{O}_{K,S}, e_r\Lambda(\mathcal{G})^{\#}(1-r))
	\simeq 
	R\Gamma_T(\mathcal{O}_{K,S}, e_r\Z_p[G]^{\#}(1-r)).
\]
However, both complexes in this formula are acyclic outside degree $2$ so that we
actually have an isomorphism of $e_r \Z_p[G]$-modules
\[
	H^{2}_{T}(\mathcal{O}_{K,S}, e_{r} \Lambda(\mathcal{G})^{\#}(1-r))_{\Gamma_{L}} \simeq
	H^{2}_{T}(\mathcal{O}_{K,S}, e_{r} \Z_{p}[G]^{\#}(1-r)),
\]
where $\Gamma_{L} := \Gal(L_{\infty}/L)$.
Let $\aug: \Lambda(\mathcal{G}) \rightarrow \Z_{p}[G]$ be the canonical projection map.
Then \eqref{eqn:Fitt-of-H2T} and the fact that Fitting ideals commute with base change 
(see for example \cite[Corollary 20.5]{MR1322960})
imply that $\aug(t_{\mathrm{cyc}}^r(\Psi_{S,T}^{\#}))$ 
generates the Fitting ideal of $H^{2}_{T}(\mathcal{O}_{K,S}, e_{r} \Z_p[G]^{\#}(1-r))$.
By Proposition~\ref{prop:PsiST-is-inverse-limit} we have
$\aug(t_{\mathrm{cyc}}^{r}(\Psi_{S,T}^{\#})) = \Theta_{S,T}(r)$, which completes the proof.
\end{proof}
	
Now let $L/K$ be an arbitrary finite abelian extension of number fields and let $G=\Gal(L/K)$. 
For each integer $r$, we define an idempotent in $\Z[\half][G]$ by
\[
	e_r := \left\{ \begin{array}{ll}
	\prod_{v \in S_{\infty}} \frac{1- (-1)^r j_v}{2} & \mbox{if } K
	\mbox{ is totally real;}\\
	0 & \mbox{otherwise},
	\end{array} \right.
\] 
where $j_{v}$ is the generator of the decomposition group $G_{v}$ for
each $v \in S_{\infty}$. Note that this is compatible with the 
above definition of $e_r$ in the case of CM extensions.

We obtain the following refinement of Theorem \ref{thm:Coates--Sinnott}.

\begin{corollary}\label{cor:Coates--Sinnott}
Let $L/K$ be a finite abelian extension of number fields and let $G=\Gal(L/K)$.
Then for every finite set $S$ of places of $K$ containing $S_{\ram}(L/K) \cup S_{\infty}$ we have
\begin{equation} \label{eqn:refined-CS}
\Ann_{\Z[\frac{1}{2}][G]}(\Z[\half] \otimes_{\Z} K_{1-2r}(\mathcal{O}_{L})_{\tors}) \Theta_S(r) = 
e_{r} \Fitt_{\Z[\frac{1}{2}][G]}(\Z[\half] \otimes_{\Z} K_{-2r}(\mathcal{O}_{L,S}))
\end{equation}
In particular, the Coates--Sinnott conjecture 
\ref{conj:Coates--Sinnott} holds away from $2$, that is,
\[
\Ann_{\Z[\frac{1}{2}][G]}(  \Z[\half] \otimes_{\Z} K_{1-2r}(\mathcal{O}_L)_{\tors}) \Theta_{S}(r)
\subseteq \Ann_{\Z[\frac{1}{2}][G]}(\Z[\half] \otimes_{\Z} K_{-2r}(\mathcal{O}_{L,S})).
\]	
\end{corollary}

\begin{proof}
Fix an odd prime $p$.
In order to verify the $p$-part of \eqref{eqn:refined-CS}, we can and do assume that 
$S_{p} \subseteq S$ as the Euler factors at $v \in S_p$ are units in $\Z_p[G]$ 
by \cite[Lemma 6.13]{MR3383600}.
Moreover, since the $p$-adic Chern class maps \eqref{eqn:Chern-class-maps} are isomorphisms by the norm residue isomorphism theorem
\cite{MR2529300}, we may work with the \'{e}tale cohomological version of \eqref{eqn:refined-CS} as in \cite[\S 6]{MR3383600}
(where the corresponding claim is denoted by $\overline{CS}(L/K,S,p,1-r)$).
By \cite[Lemma 6.14]{MR3383600} we can and do assume that $K$ is totally real. 
Likewise, by \cite[Lemmas 6.15 and 6.16]{MR3383600} we can and do assume that $L$ is a 
CM extension of $K$.
We have an exact sequence of finite $e_r \Z_p[G]$-modules (this follows easily from the definitions; see \cite[(21)]{MR3072281})
\begin{multline*}
0 \longrightarrow H^{1}(\mathcal{O}_{K,S}, e_{r} \Z_{p}[G]^{\#}(1-r))
\longrightarrow \bigoplus_{v \in T} H^1(K(v), e_{r} \Z_{p}[G]^{\#}(1-r)) \\
\longrightarrow H^{2}_{T}(\mathcal{O}_{K,S}, e_{r} \Z_{p}[G]^{\#}(1-r))
\longrightarrow H^{2}(\mathcal{O}_{K,S}, e_{r} \Z_p[G]^{\#}(1-r))
\longrightarrow 0.
\end{multline*}
The middle two terms are finite cohomologically trivial $G$-modules and their Fitting ideals are generated by
$\delta_{T}(r)$ and $\Theta_{S,T}(r)$ by \cite[Lemma~5.4]{MR3072281} and Theorem \ref{thm:strong-Coates-Sinnott}, respectively.
Since $H^{1}(\mathcal{O}_{K,S}, e_{r} \Z_{p}[G]^{\#}(1-r))$ is finite cyclic, we have that
\begin{align*}	
& \Fitt_{e_{r} \Z_{p}[G]}(H^{1}(\mathcal{O}_{K,S}, e_{r} \Z_{p}[G]^{\#}(1-r))^{\vee})^{\#} \\
& = \Ann_{e_{r} \Z_p[G]}(H^{1}(\mathcal{O}_{K,S}, e_{r} \Z_{p}[G]^{\#}(1-r))^{\vee})^{\#}\\
& = e_{r} \Ann_{\Z_p[G]}(H^{1}(\mathcal{O}_{L,S}, \Z_{p}(1-r))_{\tors}),
\end{align*}
where we have used Shapiro's lemma for the last equality.
The result now follows from Lemma \ref{lem:Fitting-four-term}.
\end{proof}

\appendix

\section{Independence of the choice of complex}\label{app:independence-of-choice-of-complex}

The main goal of this appendix is to prove a purely algebraic result
(Theorem \ref{thm:no-dependence-on-choice-of-complex}) that justifies the claim in the introduction 
that the precise choice of complex used in the EIMC does not matter, provided that it is perfect and 
has the prescribed cohomology.

We  need several preliminary results.
Let $p$ be a prime and let $\mathcal{G}$ be an admissible one-dimensional $p$-adic Lie group.
Choose a central open subgroup $\Gamma_{0} \leq \mathcal{G}$ such that $\Gamma_{0} \simeq \Z_{p}$
and let $R=\Z_{p}\llbracket \Gamma_{0} \rrbracket$.
For a height one prime ideal $\mathfrak{p}$ of $R$, let $R_{\mathfrak{p}}$ be the localisation of $R$ at 
$\mathfrak{p}$ and let $\Lambda_{\mathfrak{p}}(\mathcal{G}) = R_{\mathfrak{p}} \otimes_{R} \Lambda(\mathcal{G})$.

\begin{lemma}\label{lem:commutative-diagrams-at-height-one-prime}
For each height one prime ideal $\mathfrak{p}$ of $R$ we have a commutative diagram
\[
\xymatrix{
K_{1}(\Lambda_{\mathfrak{p}}(\mathcal{G})) \ar[r] \ar[rd]_{\nr} & K_{1}(\mathcal{Q}(\mathcal{G})) \ar[r]^{\partial_{\Lambda_{\mathfrak{p}}(\mathcal{G})} \quad} \ar[d]_{\nr} &
K_{0}(\Lambda_{\mathfrak{p}}(\mathcal{G}),\mathcal{Q}(\mathcal{G})) \ar[r] \ar[d]_{\beta_{\mathfrak{p}}} & 0\\
& \zeta(\mathcal{Q}(\mathcal{G}))^{\times} \ar[r] &
\zeta(\mathcal{Q}(\mathcal{G}))^{\times} /\nr(K_{1}(\Lambda_{\mathfrak{p}}(\mathcal{G})) \ar[r] & 1
}
\]
with exact rows. 
Moreover, if $SK_{1}(\mathcal{Q}(\mathcal{G}))=0$ then $\beta_{\mathfrak{p}}$ is injective.
\end{lemma}

\begin{proof}
The existence and exactness of the top row follow from the long exact sequence of $K$-theory \eqref{eqn:long-exact-seq} and the surjectivity of $\partial_{\Lambda_{\mathfrak{p}}(\mathcal{G})}$
\cite[Corollary 2.14]{MR4098596}. 
The triangle commutes by definition.
The existence
of $\beta_{\mathfrak{p}}$ follows from the exactness of the top row.
The second claim follows from the snake lemma. 
\end{proof}

\begin{lemma}\label{lem:commutative-diagram-over-all-height-one-primes}
We have the following commutative diagram
\[
\xymatrix{
K_{1}(\Lambda(\mathcal{G})) \ar[r]^{\iota} \ar[rd]_{\nr} &
K_{1}(\mathcal{Q}(\mathcal{G})) \ar[r]^{\partial \qquad} \ar[d]_{\nr} &
K_{0}(\Lambda(\mathcal{G}),\mathcal{Q}(\mathcal{G})) \ar[r] \ar[d]_{\alpha} & 
\bigoplus_{\mathfrak{p}}
K_{0}(\Lambda_{\mathfrak{p}}(\mathcal{G}),\mathcal{Q}(\mathcal{G}))
\ar[d]_{\coprod_{\mathfrak{p}} \beta_{\mathfrak{p}}}
\\
& \zeta(\mathcal{Q}(\mathcal{G}))^{\times} \ar[r] &
\frac{\zeta(\mathcal{Q}(\mathcal{G}))^{\times}}{\nr(K_{1}(\Lambda(\mathcal{G}))} \ar[r] & 
\prod_{\mathfrak{p}}
\frac{\zeta(\mathcal{Q}(\mathcal{G}))^{\times}}{\nr(K_{1}(\Lambda_{\mathfrak{p}}(\mathcal{G})))}
,}
\] 
where the direct sum and product are over all height one prime ideals of $R$
(note that the rows are not exact). 
Moreover, if $SK_{1}(\mathcal{Q}(\mathcal{G}))=0$ then $\alpha$ and 
$\coprod_{\mathfrak{p}} \beta_{\mathfrak{p}}$
are injective.
\end{lemma}

\begin{proof}
This follows from Lemmas \ref{lem:hom-rel-K0-nr} and \ref{lem:commutative-diagrams-at-height-one-prime}.
\end{proof}

Let $\mathcal{M}(\mathcal{G})$ denote a maximal $R$-order such that
$\Lambda(\mathcal{G}) \subseteq \mathcal{M}(\mathcal{G}) \subseteq \mathcal{Q}(\mathcal{G})$.
Note that $\mathcal{M}(\mathcal{G})$ must exist by \cite[Corollary~10.4]{MR1972204}, though need
not be unique. However, \cite[Theorem 8.6]{MR1972204} implies that 
$\zeta(\mathcal{M}(\mathcal{G}))$ is the integral closure of 
$\zeta(\Lambda(\mathcal{G}))$ in $\zeta(\mathcal{Q}(\mathcal{G}))$ and thus is unique.
For a height one prime ideal $\mathfrak{p}$ of $R$, let $\mathcal{M}_{\mathfrak{p}}(\mathcal{G}) = R_{\mathfrak{p}} \otimes_{R} \mathcal{M}(\mathcal{G})$.

\begin{lemma}\label{lem:units-of-localisation-of-max-order}
We have
\[ 
\bigcap_{\mathfrak{p}}\zeta(\mathcal{M}_{\mathfrak{p}}(\mathcal{G}))^{\times}
= 
\zeta(\mathcal{M}(\mathcal{G}))^{\times},
\] 
where the intersection ranges over all height one prime ideals $\mathfrak{p}$ of $R$.
\end{lemma}

\begin{proof}
Since $\zeta(\mathcal{M}(\mathcal{G}))$ is a maximal $R$-order
in $\zeta(\mathcal{Q}(\mathcal{G}))$
it is reflexive by \cite[Theorem~11.4]{MR1972204}. 
Hence  we have that
$
\bigcap_{\mathfrak{p}}\zeta(\mathcal{M}_{\mathfrak{p}}(\mathcal{G}))
= 
\zeta(\mathcal{M}(\mathcal{G}))
$ by \cite[Lemma 5.1.2(iii)]{MR2392026}.
Moreover, if $x \in \bigcap_{\mathfrak{p}}\zeta(\mathcal{M}_{\mathfrak{p}}(\mathcal{G}))^{\times}$
then also $x^{-1} \in \bigcap_{\mathfrak{p}}\zeta(\mathcal{M}_{\mathfrak{p}}(\mathcal{G}))^{\times}$,
and hence $x \in \zeta(\mathcal{M}(\mathcal{G}))^{\times}$.
This gives one inclusion and the reverse inclusion is clear.
\end{proof}

\begin{theorem}[Ritter--Weiss]\label{thm:nr-of-K1}
We have 
\[
\nr(K_{1}(\Lambda_{(p)}(\mathcal{G}))) \cap \zeta(\mathcal{M}(\mathcal{G}))^{\times} 
\subseteq \nr(K_{1}(\Lambda(\mathcal{G}))).
\] 
\end{theorem}

\begin{proof}
This can be deduced from \cite[Theorem B]{MR2205173} as follows.
Denote the integral closure of $\Z_{p}$ in $\Q_{p}^{c}$ by $\Z_{p}^c$.
Then $\zeta(\mathcal{M}(\mathcal{G}))^{\times}$ corresponds to
$\Hom^{\ast}_{G_{\Q_{p}}}(R_{p}(\mathcal{G}), (\Z_{p}^{c} \otimes_{\Z_{p}} \Lambda(\overline{\Gamma}))^{\times})$ under the identification
in diagram \eqref{eqn:Det_triangle} by \cite[Remark H]{MR2114937}.
We may replace $\Lambda_{(p)}(\mathcal{G})$ by its $p$-adic completion
in the statement of the theorem
as this can only enlarge the left hand side of the inclusion. 
Moreover, the image of $\Det$ is always contained in the 
$\mathrm{HOM}^{\ast}$-group used in \cite{MR2205173} (see \cite[\S 1]{MR2205173}).
\end{proof}

\begin{corollary}\label{cor:canonical-map-injective}
The canonical map
\[
\frac{\zeta(\mathcal{Q}(\mathcal{G}))^{\times}}{\nr(K_{1}(\Lambda(\mathcal{G})))}
\longrightarrow \prod_{\mathfrak{p}}
\frac{\zeta(\mathcal{Q}(\mathcal{G}))^{\times}}{\nr(K_{1}(\Lambda_{\mathfrak{p}}(\mathcal{G})))}
\] 
is injective, where the direct product is taken over all height one prime ideals of $R$.
\end{corollary}

\begin{proof}
It follows from \cite[Theorem 10.1]{MR1972204} that
$\nr(K_{1}(\Lambda_{\mathfrak{p}}(\mathcal{G})))
\subseteq \zeta(\mathcal{M}_{\mathfrak{p}}(\mathcal{G}))^{\times}$ for every height one ideal $\mathfrak{p}$
of $R$.
Thus it suffices to show that 
\[
\nr(K_{1}(\Lambda_{(p)}(\mathcal{G}))) \cap
\bigcap_{\mathfrak{p} \neq(p)}\zeta(\mathcal{M}_{\mathfrak{p}}(\mathcal{G}))^{\times}
\subseteq \nr(K_{1}(\Lambda(\mathcal{G}))).
\] 
The desired result now follows from 
Lemma \ref{lem:units-of-localisation-of-max-order} and Theorem \ref{thm:nr-of-K1}.
\end{proof}

\begin{corollary}\label{cor:injectivity-rel-K-localization}
If $SK_{1}(\mathcal{Q}(\mathcal{G}))=0$ then the canonical map
\[
K_{0}(\Lambda(\mathcal{G}),\mathcal{Q}(\mathcal{G}))
\longrightarrow
\bigoplus_{\mathfrak{p}}
K_{0}(\Lambda_{\mathfrak{p}}(\mathcal{G}),\mathcal{Q}(\mathcal{G}))
\] 
is injective,
where the direct sum is taken over all height one prime ideals of $R$.
\end{corollary}

\begin{proof}
This follows immediately from Corollary \ref{cor:canonical-map-injective} and
Lemma \ref{lem:commutative-diagram-over-all-height-one-primes}. 
\end{proof}

\begin{remark}
In the case $SK_{1}(\mathcal{Q}(\mathcal{G}))=0$,
Corollary \ref{cor:injectivity-rel-K-localization} implies
\cite[Proposition 4]{MR1935024}, which says that the class of any module of finite cardinality 
is trivial in $K_{0}(\Lambda(\mathcal{G}),\mathcal{Q}(\mathcal{G}))$.
This is because such modules vanish after localisation at height one prime ideals of $R$.
\end{remark}

An immediate consequence of the following purely algebraic result is that 
the choice of complex used in the EIMC (Conjecture \ref{conj:EIMC}) does not matter, 
as long as it is perfect and has the cohomology specified in \eqref{eq:cohomology-of-FK-complex}. 

\begin{theorem}\label{thm:no-dependence-on-choice-of-complex}
Let $\mathcal{G}$ be an admissible one-dimensional $p$-adic Lie group
and let\linebreak $\Phi \in \zeta(\mathcal{Q}(\mathcal{G}))^{\times}$. Let $k \in \Z$.
Let $C^{\bullet}, D^{\bullet} \in \mathcal{D}\tor^{\perf}(\Lambda(\mathcal{G}))$
such that
\begin{enumerate}
\item 
$H^{i}(C^{\bullet}) \simeq H^{i}(D^{\bullet})$ as $\Lambda$-modules for all $i \in \Z$;
\item 
$H^{i}(C^{\bullet})$ and $H^{i}(D^{\bullet})$ are finitely generated over $\Z_{p}$
for all $i \in \Z - \{k\}$; and
\item 
there exists $x \in K_{1}(\mathcal{Q}(\mathcal{G}))$ such that $\partial(x) = [C^{\bullet}]$
and $\nr(x) = \Phi$.
\end{enumerate}
Then there exists $y \in K_{1}(\mathcal{Q}(\mathcal{G}))$ such that $\partial(y) = [D^{\bullet}]$ and $\nr(y) = \Phi$. 
\end{theorem}

\begin{proof}
By (ii) we have 
$\Lambda_{(p)}(\mathcal{G}) \otimes_{\Lambda(\mathcal{G})} H^{i}(C^{\bullet})
=\Lambda_{(p)}(\mathcal{G}) \otimes_{\Lambda(\mathcal{G})} H^{i}(D^{\bullet})=0$ 
for all $i \in \Z - \{ k \}$.
Thus by (i), Proposition \ref{prop:perfect-cohomology-acyclic-outside-one-degree} and the fact that localization 
is an exact functor, we have
\begin{align*}
[\Lambda_{(p)}(\mathcal{G}) \otimes_{\Lambda(\mathcal{G})}^{\mathbb{L}} C^{\bullet}]
&= (-1)^{k}[\Lambda_{(p)}(\mathcal{G}) \otimes_{\Lambda(\mathcal{G})} H^{k}(C^{\bullet})] \\
&= (-1)^{k}[\Lambda_{(p)}(\mathcal{G}) \otimes_{\Lambda(\mathcal{G})} H^{k}(D^{\bullet})]
= [\Lambda_{(p)}(\mathcal{G}) \otimes_{\Lambda(\mathcal{G})}^{\mathbb{L}} D^{\bullet}]
\end{align*}
in $K_{0}(\Lambda_{(p)}(\mathcal{G}),\mathcal{Q}(\mathcal{G}))$.
Now fix a height one prime ideal $\mathfrak{p} \neq (p)$ of $R$.
Then every finitely generated $\Lambda_{\mathfrak{p}}(\mathcal{G})$-module 
has projective dimension at most one by \cite[Corollary 3.5]{MR4098596}.
(Alternatively, note that $\Lambda_{\mathfrak{p}}(\mathcal{G})$
is a maximal order over the discrete valuation ring
$R_{\mathfrak{p}}$, and thus is hereditary by \cite[Theorem 18.1]{MR1972204}.)
Hence (i) and Proposition \ref{prop:perfect-cohomology-fpd} imply that
\[
[\Lambda_{\mathfrak{p}}(\mathcal{G})  \otimes_{\Lambda(\mathcal{G})}^{\mathbb{L}} C^{\bullet}]
=[\Lambda_{\mathfrak{p}}(\mathcal{G})  \otimes_{\Lambda(\mathcal{G})}^{\mathbb{L}} D^{\bullet}]
\text{ in } K_{0}(\Lambda_{\mathfrak{p}}(\mathcal{G}) ,\mathcal{Q}(\mathcal{G})).
\]
Now recall the commutative diagram of Lemma \ref{lem:commutative-diagram-over-all-height-one-primes}.
By Corollary \ref{cor:canonical-map-injective} and an easy diagram
chase in the right commutative square, we have $\alpha([C^{\bullet}])=\alpha([D^{\bullet}])$.
Choose any $y_{0} \in K_{1}(\mathcal{Q}(\mathcal{G}))$ such that $\partial(y_{0}) = [D^{\bullet}]$.
Then by (iii) and an easy diagram
chase in the left commutative square, we have
 $\nr(xy_{0}^{-1}) \in \nr(K_{1}(\Lambda(\mathcal{G})))$.
In other words, there exists $z \in K_{1}(\Lambda(\mathcal{G}))$ such that
$\nr(z) = \nr(xy_{0}^{-1})$. 
Now set $y := y_{0} \iota(z)$.
Since $\iota(K_{1}(\Lambda(\mathcal{G})))=\ker(\partial)$ we have
$\partial(y) = \partial(y_{0}) = [D^{\bullet}]$ and $\nr(y) = \nr(x) = \Phi$.
\end{proof}

\begin{corollary}\label{cor:complex-in-rel-K0-when-SK1-vanishes}
Let $\mathcal{G}$ be an admissible one-dimensional $p$-adic Lie group
such that $SK_{1}(\mathcal{Q}(\mathcal{G}))=0$. Let $k \in \Z$.
Let $C^{\bullet}, D^{\bullet} \in \mathcal{D}\tor^{\perf}(\Lambda(\mathcal{G}))$
such that
\begin{enumerate}
\item 
$H^{i}(C^{\bullet}) \simeq H^{i}(D^{\bullet})$ as $\Lambda$-modules for all $i \in \Z$; and
\item 
$H^{i}(C^{\bullet})$ and $H^{i}(D^{\bullet})$ are finitely generated over $\Z_{p}$
for all $i \in \Z - \{k\}$.
\end{enumerate}
Then $[C^{\bullet}]=[D^{\bullet}]$ in $K_{0}(\Lambda(\mathcal{G}),\mathcal{Q}(\mathcal{G}))$.
\end{corollary}

\bibliography{Abelian-MC-Bib}{}
\bibliographystyle{amsalpha}

\end{document}